%% file: main.tex
\documentclass[a4paper]
{cedram-smai-jcm}

\usepackage{tikz}
\usepackage{subcaption}
\input{sources/packages.tex}

\title[Streamfunction-vorticity formulation for flows]{Streamfunction-vorticity formulation for incompressible viscid and inviscid flows on general surfaces}

\author[T. Brüers]{\firstname{Tim} \lastname{Brüers}}
\address{Georg August University of Göttingen, Germany}
\email{t.brueers@math.uni-goettingen.de}
\thanks{}

\author[C. Lehrenfeld]{\firstname{Christoph} \lastname{Lehrenfeld}}
\address{Georg August University of Göttingen, Germany}
\email{lehrenfeld@math.uni-goettingen.de}

\author[M. Wardetzky]{\firstname{Max} \lastname{Wardetzky}}
\address{Georg August University of Göttingen, Germany}
\email{wardetzky@math.uni-goettingen.de}



\begin{document}

\begin{abstract}
This paper presents a streamfunction-vorticity formulation for the Navier--Stokes and Euler equations on general surfaces.
Notably, this includes non-simply connected surfaces, on which the harmonic components of the velocity field play a fundamental role in the dynamics.
By relying only on scalar and finite-dimensional quantities, our formulation ensures that the resulting methods give exactly tangential and incompressible velocity fields, while also being pressure robust.
Compared to traditional methods based on velocity--pressure formulations, where one can only guarantee these structural properties by increasing the computational costs, this is a key advantage.
We rigorously validate our formulation by proving its equivalence to the well understood velocity formulation under additional regularity assumptions.
Furthermore, we demonstrate the applicability of the approach with numerical examples.
\end{abstract}

\maketitle


\input{content/01Introduction}
\input{content/02Preliminaries}
\input{content/03Regularity}
\input{content/04MixedFormulation}
\input{content/05NumericalExamples}

\section*{Use of AI tools}
Generative AI tools (Claude Opus~5.5, Claude Fable~5.1 and Claude Fable~5.2 by Anthropic, and ChatGPT Astra by OpenAI) have been used to assist with the revision of this manuscript.
They were used to check mathematical arguments, notation and cross-references for consistency and to suggest and edit formulations of the revised text.
All suggestions were critically reviewed, verified and, where necessary, corrected by the authors, who take full responsibility for the content of the paper.

\bibliographystyle{plain}
\bibliography{sources/biblio}
\input{content/06Appendix}
\input{content/07AppendixProofs}
\end{document}

%% file: sources/packages.tex
\usepackage[english]{babel}      
\usepackage[utf8]{inputenc}      
\usepackage{microtype}           
\usepackage[T1]{fontenc}         
\usepackage{lmodern}             
\usepackage{textcomp}            
\usepackage{verbatim}            
\usepackage{placeins}            
\usepackage{color}               
\usepackage{graphicx}            
\usepackage{multirow}            
\usepackage{rotating}            
\usepackage{nicefrac}            
\usepackage{todonotes}           
\usepackage{overpic}
\usepackage{wrapfig}             
\usepackage{ragged2e}            
\usepackage{float}               
\usepackage[normalem]{ulem}
\usepackage{latexsym,exscale,stmaryrd,amssymb,amsmath}
\usepackage{mathtools}

\usepackage{amsthm}

\usepackage[nameinlink]{cleveref}

\usepackage{thmtools}
\usepackage{hyperref}
\usepackage{cleveref}

\usepackage{siunitx}
\DeclareSIUnit{\skalenteil}{Skt}

\usepackage{mathrsfs} 

\usepackage{enumerate}

\usepackage{stackrel}

\newcommand\Rey{\mbox{\textit{Re}}}

\AtBeginDocument{
\declaretheorem[name={Definition and Proposition},refname={Def.\&Prop.,Def.\&Prop.},Refname={Def.\&Prop.,Def.\&Prop.},style=plain,numberlike=cdrthm]{defnpro}
}

\newcommand{\refwl}[1]{\autoref{#1}\,(\nameref{#1})} 

\newcommand\nextline{\hspace{0.1cm}\\\hspace{-0.1cm}}

\newcommand{\Sec}{\mathfrak{X}}

\usepackage{mdframed}
\mdfdefinestyle{overviewbox}{%
  linewidth=0.5pt, linecolor=black!60, backgroundcolor=black!4,
  innertopmargin=4pt, innerbottommargin=5pt, innerleftmargin=7pt, innerrightmargin=7pt,
  skipabove=7pt, skipbelow=7pt, nobreak=true
}
\crefname{defi}{Definition}{Definitions}\Crefname{defi}{Definition}{Definitions}
\crefname{theo}{Theorem}{Theorems}\Crefname{theo}{Theorem}{Theorems}
\crefname{lemm}{Lemma}{Lemmas}\Crefname{lemm}{Lemma}{Lemmas}
\crefname{prop}{Proposition}{Propositions}\Crefname{prop}{Proposition}{Propositions}
\crefname{coro}{Corollary}{Corollaries}\Crefname{coro}{Corollary}{Corollaries}
\crefname{rema}{Remark}{Remarks}\Crefname{rema}{Remark}{Remarks}
\newenvironment{summarybox}[1]{\begin{mdframed}[style=overviewbox]\noindent\textbf{#1 (formal summary).}\ }{\end{mdframed}}

%% file: content/01Introduction.tex
\section{Introduction}
The incompressible Navier--Stokes and Euler equations on surfaces model viscous and inviscid flows, usually expressed in the \emph{velocity--pressure formulation}.
In strong form the equations read
\begin{align}
    &\partial_tV+\nabla_{V}V=F-\mathrm{grad}(p)-\nu\mathrm{rot}(\mathrm{curl}(V))+2\nu \kappa V, 
    \quad \mathrm{div}(V)=0.
    \label{eq:V-pstrong}
\end{align}
Here, $V$ denotes the velocity field, $p$ the static pressure, $F$ external forces, $\nu$ the kinematic viscosity, and $\kappa$ the Gaussian curvature, while $\nabla, \mathrm{grad}, \mathrm{div}, \mathrm{rot}, \mathrm{curl}$ are the corresponding surface differential operators. 
Both $p$ and $F$ are normalized by the constant fluid density. 
The terms involving $\nu$ represent the viscous forces, which vanish for the inviscid Euler equations ($\nu=0$), and the divergence constraint expresses incompressibility (or inextensibility) of the surface flow.

Numerical methods for \eqref{eq:V-pstrong} on flat domains and surfaces are well established, see e.g.\ \cite{NavierTime, Rannacher1, John, Voigt, Fries, Olashanskii, olshanskii2021inf, Nitschke2017}.  
However, standard numerical approaches based on the velocity--pressure formulation come with structural challenges:  
(i) they require vector-valued unknowns, (ii) tangentiality and incompressibility of discrete velocities can only be enforced weakly in many discretizations, and (iii) pressure robustness is lost unless special care is taken, leading to velocity errors that scale with pressure errors \cite{LINKE2014782}.  
Moreover, standard mixed velocity--pressure discretizations that enforce incompressibility with the pressure as a Lagrange multiplier lead to saddle-point systems, in particular for implicit or semi-implicit updates; these systems are computationally demanding and require compatible velocity--pressure spaces for stability.

An alternative is the \emph{vorticity formulation}, obtained from \eqref{eq:V-pstrong} by introducing the scalar vorticity $\omega = \mathrm{curl}(V)$ and employing the Hodge decomposition
\begin{align}
  V = \mathrm{rot}(\psi) + H,
  \label{eq:HodgeIntro}
\end{align}
with a scalar streamfunction $\psi$ and a harmonic vector field $H$.
Although the harmonic component $H$ is vector-valued, its corresponding space $H_N$ is finite-dimensional, with a dimension equal to the first Betti number $b_1$ of the surface.
This makes it possible to numerically pre-compute a basis for this space that is exactly tangential and divergence-free.

This change of variables recasts the system into scalar equations for $\psi$ and $\omega$, complemented by a $b_1$-dimensional equation for $H$:
\begin{subequations}
\label{eq:evolution}
\begin{align}
    \omega &=\Delta_{\mathrm{dR}}(\psi),
    \label{eq:evolutionpsi}
    \\\partial_t\omega &=\mathrm{curl}(F)-\langle V,\mathrm{grad}(\omega)\rangle -\nu \Delta_{\mathrm{dR}}(\omega)+2\nu \mathrm{curl}(\kappa V),
    \label{eq:evolutionw}
    \\\partial_t H &=P_{H_N}(F-\omega JV-\nu \mathrm{rot}(\omega)+2\nu \kappa V).
    \label{eq:evolutionH}
\end{align}
\end{subequations}
In this system, $\Delta_{\mathrm{dR}}=-\mathrm{div}(\mathrm{grad}(\cdot))$ denotes the de Rham Laplacian, $P_{H_N}$ is the projection onto the space of harmonic fields, and the velocity $V$ is always given by the decomposition \eqref{eq:HodgeIntro}.
Together, these equations form a coupled problem consisting of two evolution equations for $\omega$ and $H$, and an elliptic (reconstruction) equation for $\psi$.

The formulation is therefore algorithmically attractive: it avoids saddle-point problems, involves primarily scalar unknowns, and yields exactly divergence-free, tangential discrete velocities.
In this formulation, the pressure decouples and is recovered \emph{a posteriori} from a Poisson problem. 
This decoupling, combined with the geometric properties of the Hodge decomposition, ensures that tangentiality, incompressibility, and pressure robustness are all satisfied by construction.
A further motivation for the weak formulation developed below is that the velocity itself is not required to be represented in an $H^1$-conforming discrete space.
The regularity imposed on the vorticity and streamfunction variables is therefore compatible with low-order velocity representations, including the piecewise constant realization used in \Cref{sec:LowOrderAlgo}.

Despite its advantages, a comprehensive theory for the vorticity formulation on general surfaces has remained elusive, both in the differential geometry community (which has provided deep insights into the geometric structure of the Euler equations, see, e.g., \cite{ArnoldKhesin}) and the computational community.
The existing literature has largely sidestepped the full dynamics of the harmonic component by relying on highly restrictive assumptions of either a geometric or dynamic nature.

Geometric restrictions confined research to planar domains \cite{ChorinMarsden,dean1991iterative, achdou1992tuning, LiuDiscont, LiuWeinan, Tian, Ghadi, tezduyar1988petrov, Tezduyar, GlowinskiPironneauMultiply, MizukamiMultiply, LippkeWagnerMultiply, BeiraoDaVeiga2023, Adaketal2021, Adak2024} or simply connected surfaces \cite{nitschke_voigt_wensch_2012, Reuther2015, Reusken2018, Brandner, Brandner2021FiniteED, Gross}, special cases that allow the velocity to be expressed simply as $V=\mathrm{rot}(\psi)$. 
The widely employed dynamics restriction, in contrast, is to assume situations in which $H$ is constant in time \cite{elcott2007stable,Azencot2014,ando2015stream}. 
On general surfaces with evolving flows, however, neither type of assumption is justified, and the full dynamics of $H$ must be accounted for.

Recent advances \cite{FluidCohomology, ZhuYinChern2025}, drawing on the spectral approach of \cite{IdeaHarmonic}, provided a versatile key to lifting these limitations by deriving the evolution of $H$ as in \eqref{eq:evolutionH}.
Similarly, we lift these topological restrictions by carefully incorporating the harmonic fields into the system.

The main contribution of the present work is to establish the rigorous mathematical foundation of the vorticity formulation for both the Navier--Stokes and Euler scenarios. 
We derive the requisite weak forms of the vorticity formulation for both scenarios and prove their respective equivalence to the classical velocity formulations -- thereby validating the vorticity approach in full generality.

As a consequence, the vorticity formulation yields a computationally attractive alternative to velocity--pressure methods.  
It combines exact incompressibility, tangentiality, and pressure robustness with a reduced scalar structure, opening the door for efficient algorithms that exploit the full potential of this formulation.
We demonstrate this potential for a set of interesting examples.

\subsection*{Structure and main contributions}
The presentation of our central theoretical results starts in \Cref{sec:ScopeOfWork} with a formal derivation that motivates the weak vorticity formulations and explains the mechanism behind their equivalence to the standard weak velocity formulations.
This overview is deliberately informal: the formulations and the equivalence results are stated there as formal summaries, while all regularity requirements are postponed.
The requisite rigorous mathematical framework is then established across \Cref{sec:Preliminaries,sec:reg,sec:mixed}.\\
This framework is built in several stages. 
First, in \Cref{sec:Preliminaries}, we establish the necessary groundwork by generalizing key identities and regularity statements from the flat case to surfaces.
Then, one of our main analytical contributions is the generalization of classical Navier--Stokes regularity results by Heywood and Rannacher \cite{Rannacher1} to general surfaces (\Cref{thm:regsol}) -- a result that to our knowledge is not available in the literature. 
We achieve this by extending the Stokes regularity theorem to surfaces (\Cref{thm:StokesReg}).
\Cref{sec:reg} also contains the precise definition of the velocity formulation and the reconstruction of the pressure (\Cref{sec:VeloPrecise}).
Finally, these analytical tools from \Cref{sec:Preliminaries,sec:reg} culminate in \Cref{sec:mixed} in our main theoretical result: the precise definitions of the vorticity formulations and the rigorous proof of the equivalence between the weak vorticity and velocity formulations (\Cref{thm:EquiForm,thm:EquiFormEu}).\\
We conclude by showcasing the practical power of our framework by presenting numerical simulations based on the vorticity formulation for general surfaces (see \Cref{sec:NumericalResults}). 
Readers only interested in implementation details can directly skip to \Cref{sec:NumericalResults}.

\section{Overview of main results}
\label{sec:ScopeOfWork}
Our main objective is to develop and justify weak streamfunction-vorticity formulations for the incompressible Navier--Stokes and Euler equations on general surfaces.
The purpose of this section is deliberately informal: we explain the structure of the formulations and the mechanism behind their equivalence to the weak velocity formulations, while postponing the precise regularity assumptions, Bochner spaces, and dual-space interpretations to \Cref{sec:reg,sec:mixed}.
Throughout, $M$ is a compact, connected, oriented two-dimensional Riemannian manifold with Lipschitz boundary $\partial M$, which may be empty; precisely, $M$ is the closure of an open subset of a smooth surface $W$ (see \Cref{def:surface} and \Cref{fig:domains} for examples), and higher boundary regularity is stated explicitly where needed.
We write $\Sec$ for tangential vector fields on $M$.

\paragraph{Geometric notation and differential operators.}
The Riemannian metric of $M$ is denoted by $\langle\cdot,\cdot\rangle$ and its Levi-Civita connection by $\nabla$ \cite[Ch.~4--5]{Lee2}. In particular, for tangential fields $X,Y,Z$,
\begin{align*}
    Y\langle X,Z\rangle=\langle\nabla_YX,Z\rangle+\langle X,\nabla_YZ\rangle.
\end{align*}
For a surface in $\mathbb{R}^3$ with the induced metric, $\nabla_YX$ is the tangential projection of the ambient derivative of any extension $\widetilde X$ of $X$,
\begin{align*}
    \nabla_YX=P\,D\widetilde X[Y],\qquad P=I-n\otimes n,
\end{align*}
where $D\widetilde X[Y]$ denotes the derivative of $\widetilde X$ in direction $Y$ and $n$ is the unit surface normal.
We regard $\nabla X$ as the endomorphism $Y\mapsto\nabla_YX$ of the tangent plane, with transpose $\langle(\nabla X)Y,Z\rangle=\langle Y,(\nabla X)^TZ\rangle$ and symmetric part $D_X\coloneqq\frac{1}{2}(\nabla X+(\nabla X)^T)$, the rate-of-strain tensor.
Gradient and divergence are given by
\begin{align*}
    \langle\mathrm{grad}(f),Y\rangle=Y(f),\qquad
    \mathrm{div}(X)=\mathrm{tr}(\nabla X)=\sum_{i=1}^{2}\langle\nabla_{E_i}X,E_i\rangle,
\end{align*}
where $(E_1,E_2)$ is any local orthonormal frame.
The orientation defines the rotation $J$ by a positive right angle, $JE_1=E_2$ and $JE_2=-E_1$ in a positively oriented frame, so that $J^2=-\mathrm{id}$ and $\langle JX,Y\rangle=-\langle X,JY\rangle$. For a surface in $\mathbb{R}^3$, $JX=n\times X$.
The rotated gradient and the scalar curl are
\begin{align*}
    \mathrm{rot}(f)\coloneqq-J\mathrm{grad}(f),\qquad
    \mathrm{curl}(X)\coloneqq-\mathrm{div}(JX).
\end{align*}
If $\partial M\neq\emptyset$, $N$ denotes the outward unit conormal -- tangential to $M$ and perpendicular to $\partial M$ -- and $T\coloneqq JN$ the unit tangent along $\partial M$.
These operators are understood for smooth functions and fields here. Their extensions to Sobolev spaces are recalled in \Cref{sec:not}.

We write $L^2$, $H^1$, $H^2$ for the usual Lebesgue and Sobolev spaces of functions on $M$, $H^1_0$ for the functions in $H^1$ with zero trace (zero mean if $\partial M=\emptyset$), $L^2_0$ for the zero-mean functions in $L^2$, and $L^2(\Sec)$, $H^1(\Sec)$, $H^1_0(\Sec)$ for the corresponding spaces of tangential vector fields ($H^1_0(\Sec)=H^1(\Sec)$ if $\partial M=\emptyset$).
Further, $(\cdot,\cdot)$ is the $L^2$ inner product and $\langle F,W\rangle$ the duality pairing of a force $F$ with a test field $W$.
For the precise definitions and the convention $V_t\coloneqq V(t,\cdot)$ for time-dependent fields we refer to \Cref{sec:not}.
For the overview, we write $\mathbb{X}_{\mathrm{kin}}$, $\mathbb{X}_\omega$ and $\mathbb{X}_H$ for the test spaces of the three equations derived below: the kinematic equation (the weak form of $\omega=\mathrm{curl}(V)$), the vorticity equation, and the harmonic equation; an additional index $0$, as in $\mathbb{X}_{\mathrm{kin},0}$ or $\mathbb{X}_{\omega,0}$, indicates zero trace on $\partial M$ (zero mean if $\partial M=\emptyset$).
The Navier--Stokes formulation uses $\mathbb{X}_{\mathrm{kin}}$, $\mathbb{X}_{\omega,0}$, $\mathbb{X}_H$, the Euler formulation $\mathbb{X}_{\mathrm{kin},0}$, $\mathbb{X}_\omega$, $\mathbb{X}_H$; the precise choices are given in \Cref{def:HarmStream,def:HarmStreamEu}.
The test space of the pressure equation is denoted by $\mathbb{X}_p$.
These placeholders allow us to display the algebraic structure of the equations without interrupting the derivation by regularity bookkeeping.

\subsection{Velocity formulation for Navier--Stokes and Euler}
\label{sec:MainContri:velo}
Our starting point is the velocity--pressure formulation \eqref{eq:V-pstrong}, which on surfaces with boundary has to be complemented by boundary conditions:
for Navier--Stokes we impose the no-slip condition $V\vert_{\partial M}=0$, for Euler only the no-penetration condition $\langle V,N\rangle\vert_{\partial M}=0$.
The idea of the velocity formulation is to test the momentum balance of \eqref{eq:V-pstrong} weakly with divergence-free vector fields $W$ that satisfy the respective boundary condition.
For such $W$, integration by parts gives $(\mathrm{grad}(p),W)=-(p,\mathrm{div}(W))=0$, so that the pressure drops out.
The corresponding spaces of test fields are
\begin{align*}
    J_{L^2}\coloneqq\{V\in H(\mathrm{div})\vert~\mathrm{div}(V)=0,~\langle V, N\rangle\vert_{\partial M}=0\},\qquad
    J_{H^1_0}\coloneqq\{V\in H^1_0(\Sec)\vert~\mathrm{div}(V)=0\},
\end{align*}
the divergence-free velocity fields with no-penetration and with no-slip boundary conditions, respectively; here, $H(\mathrm{div})$ consists of the fields $V\in L^2(\Sec)$ with $\mathrm{div}(V)\in L^2$.
We call the resulting weak problem for the velocity alone the \emph{velocity formulation}, as opposed to the \emph{velocity--pressure formulation} \eqref{eq:V-pstrong}, in which the pressure remains an unknown.
Formally, for a given initial velocity $V_0$ and external force $F$, one seeks a velocity $V$ with $V_t\vert_{t=0}=V_0$ satisfying:

\begin{summarybox}{Velocity formulation}
For Navier--Stokes, with $\nu>0$ and no-slip boundary conditions,
\begin{align*}
    \partial_t(V,W)=\langle F,W\rangle-(\nabla_{V}V,W)-2\nu(D_V,\nabla W),\qquad \forall W\in J_{H^1_0},
\end{align*}
where $D_V$ is the rate-of-strain tensor of $V$ introduced above.
For Euler, with $\nu=0$ and no-penetration boundary conditions,
\begin{align*}
    \partial_t(V,W)=\langle F,W\rangle-(\nabla_{V}V,W),\qquad \forall W\in H^1(\Sec)\cap J_{L^2}.
\end{align*}
The precise solution spaces, the interpretation in time, and the initial-data assumptions are given in \Cref{def:VeloForm}.
\end{summarybox}

Although the pressure is absent from the velocity formulation, it can be reconstructed afterwards: at the formal level, the pair $(V,p)$ satisfies the weak \emph{velocity--pressure formulation}
\begin{align}
    -(p_t,\mathrm{div}(W))=\langle F_t-\partial_tV_t,W\rangle-(\nabla_{V_t}V_t,W)-2\nu(D_{V_t},\nabla W),\qquad \forall W\in H^1_0(\Sec),\label{eq:NavStokesP} \tag{$P$}
\end{align}
for a.e.\ $t\in(0,T)$, with the viscous term omitted for Euler.
The precise pressure reconstruction, including its distributional formulation and the uniqueness of $p$ up to a function of time, is stated in \Cref{prop:pres}.
The pressure drops out of the formal derivations in \Cref{sec:MainContriVort,sec:MainContriVortEu} and is recovered a posteriori from a Poisson problem in \Cref{sec:MainContriPres}.

\subsection{Vorticity formulation for Navier--Stokes}
\label{sec:MainContriVort}
The starting point is the \emph{Hodge decomposition}, which splits a vector field uniquely and $L^2$-orthogonally into gradient, rotational, and harmonic components,
\begin{align}
    L^2(\Sec)=\mathrm{grad}(H^1\cap L^2_0)\oplus\mathrm{rot}(H^1_0)\oplus H_N.
    \label{eq:HodgeAss}
    \tag{$\mathrm{As}_{\mathrm{H}}$}
\end{align}
For a divergence-free velocity field $V$ tangential to $\partial M$, the gradient component vanishes, hence
\begin{align*}
    V=\mathrm{rot}(\psi)+H
\end{align*}
with a scalar streamfunction $\psi\in H^1_0$ and a harmonic field $H\in H_N$.
Here, $\mathrm{rot}$ and $\mathrm{curl}$ are the rotated gradient and the scalar curl introduced above, with $\mathrm{curl}$ the formal $L^2$-adjoint of $\mathrm{rot}$ (see \Cref{def:Hodge*}), $H(\mathrm{curl})$ consists of the fields $Y\in L^2(\Sec)$ with $\mathrm{curl}(Y)\in L^2$, and
\begin{align*}
    H_N\coloneqq \{Y\in J_{L^2}\cap H(\mathrm{curl})\vert~\mathrm{curl}(Y)=0\}
\end{align*}
is the space of harmonic fields, whose dimension is the first Betti number $b_1$ of $M$.
The precise Hodge decomposition and its mapping properties are established in \Cref{thm:HodgeDecomp}: the $L^2$ decomposition \eqref{eq:HodgeAss} holds for Lipschitz boundaries ($k=0$ there), whereas the main results below additionally use its $H^1$ version ($k=1$), which holds whenever $\partial M$ is $C^{1,1}$.

We now derive the vorticity formulation formally.
We take the velocity formulation of \Cref{sec:MainContri:velo} and insert the surface identities
\begin{align}
    \nabla_VV=\omega JV+\frac{1}{2}\mathrm{grad}(\vert V\vert^2),\qquad 2\delta D_V=\mathrm{rot}(\omega)-2\kappa V,
    \label{eq:keyidentities}
\end{align}
where $\omega\coloneqq\mathrm{curl}(V)$ denotes the scalar vorticity (the second identity holds for divergence-free $V$) and $\delta$ the formal $L^2$-adjoint of the covariant derivative $\nabla$, so that $2\delta D_V$ is the strong form of the viscous term, $2(D_V,\nabla W)=(2\delta D_V,W)$ for $W\in H^1_0(\Sec)$.
Note that the Gaussian curvature $\kappa$ enters through the second identity.
Here, we use \eqref{eq:keyidentities} formally; the identities are proven in \Cref{sec:Preliminaries} (see \Cref{lem:Nonlinterm,thm:IbPDV}), and their use in the weak setting is justified in \Cref{sec:mixed}.
Inserting \eqref{eq:keyidentities} into the velocity formulation yields the formal equation
\begin{align}
    \partial_t(V,W)=\Big(F-\omega JV-\frac{1}{2}\mathrm{grad}(\vert V\vert^2)-\nu\mathrm{rot}(\omega)+2\nu \kappa V,W\Big)\label{eq:proofLsEq}
\end{align}
for all divergence-free $W$ satisfying the respective boundary condition, i.e., $W\in J_{H^1_0}$ for Navier--Stokes and $W\in H^1(\Sec)\cap J_{L^2}$ for Euler (with $\nu=0$).
The gradient term does not contribute: integration by parts gives
\[
    (\mathrm{grad}(\vert V\vert^2/2),W)=-(\vert V\vert^2/2,\mathrm{div}(W))=0,
\]
since $W$ is divergence-free with $\langle W,N\rangle\vert_{\partial M}=0$.
Moreover, in contrast to the velocity formulation, the right-hand side of \eqref{eq:proofLsEq} contains no derivatives of $W$; since $J_{H^1_0}$ is dense in $J_{L^2}$ with respect to the $L^2$ norm (\Cref{thm:DecompH-1AndDensity}), \eqref{eq:proofLsEq} formally extends to all $W\in J_{L^2}$.

The derivation of the vorticity formulation is now purely algebraic.
First, the definition $\omega=\mathrm{curl}(V)$ gives, after integration by parts, the kinematic equation $(\omega,\chi)=(V,\mathrm{rot}(\chi))$ for $\chi\in\mathbb{X}_{\mathrm{kin}}$, where the boundary term $\langle\langle V,T\rangle,\chi\rangle_{\partial M}$ (formally $\int_{\partial M}\langle V,T\rangle\chi\,\mathrm{d}s$, see \Cref{sec:not}) vanishes due to the no-slip condition; this is the weak form of \eqref{eq:evolutionpsi}.
Next, choose a divergence-free test field
\begin{align*}
    W=\mathrm{rot}(\phi)+Z,\qquad \phi\in\mathbb{X}_{\omega,0},\quad Z\in\mathbb{X}_H,
\end{align*}
where $\mathbb{X}_{\omega,0}\subset H^1_0$ and $\mathbb{X}_H\subset H_N$ are subspaces, so that $\mathrm{rot}(\phi)$ and $Z$ test precisely the rotational and the harmonic summand of \eqref{eq:HodgeAss}, respectively.
Such test fields belong to $J_{L^2}$, so that \eqref{eq:proofLsEq} applies.
Moreover, $(\partial_tV,\mathrm{rot}(\phi))=\partial_t(V,\mathrm{rot}(\phi))=\partial_t(\omega,\phi)$ by the kinematic equation (note $\mathbb{X}_{\omega,0}\subset\mathbb{X}_{\mathrm{kin}}$), and $(\partial_tV,Z)=\partial_t(V,Z)=\partial_t(H,Z)$ by the orthogonality of $\mathrm{rot}(\psi)$ and $Z$.
Hence the momentum equation splits into an evolution equation for the scalar vorticity and a $b_1$-dimensional evolution equation for the harmonic component -- the weak forms of \eqref{eq:evolutionw} and \eqref{eq:evolutionH}.
This leads to the following compact form of the system.

\begin{summarybox}{Navier--Stokes vorticity formulation}
Seek $(\psi,H,\omega)$ in suitable spaces, set $V=\mathrm{rot}(\psi)+H$, and require
\begin{align*}
    (\omega,\chi)&=(V,\mathrm{rot}(\chi)), &&\forall\chi\in\mathbb{X}_{\mathrm{kin}},\\
    \partial_t(\omega,\phi)&=(F-\omega JV-\nu \mathrm{rot}(\omega)+2\nu \kappa V,\mathrm{rot}(\phi)), &&\forall\phi\in\mathbb{X}_{\omega,0},\\
    \partial_t(H,Z)&=(F-\omega JV-\nu \mathrm{rot}(\omega)+2\nu \kappa V,Z), &&\forall Z\in\mathbb{X}_H,
\end{align*}
with $\nu>0$ and initial conditions $(\psi_t,H_t)\vert_{t=0}=(\psi_0,H_0)$.
These are weak forms of \eqref{eq:evolutionpsi}--\eqref{eq:evolutionH}.
The precise spaces behind $\mathbb{X}_{\mathrm{kin}}$, $\mathbb{X}_{\omega,0}$ and $\mathbb{X}_H$, together with the regularity of the unknowns and the initial conditions, are given in \Cref{def:HarmStream}.
\end{summarybox}

The formulation makes the structural decomposition transparent: the first equation determines the vorticity, including its boundary values, from the velocity, while the second and third equations evolve the rotational and the harmonic component, respectively, without pressure.
Note also that the velocity itself carries no derivatives in the equations above, only $\omega$ does; this is what admits velocity fields of low regularity, such as the piecewise constant discrete velocities in \Cref{sec:NumericalResults}.

\begin{remark}
\label{rem:SpecCases}
Two simple special cases are worth noting.
If $\kappa$ is constant, the orthogonality of the rotational and harmonic subspaces decouples the curvature contribution, since $2\nu(\kappa V,\mathrm{rot}(\phi))=2\nu\kappa(\mathrm{rot}(\psi),\mathrm{rot}(\phi))$ in the vorticity equation and $2\nu(\kappa V,Z)=2\nu\kappa(H,Z)$ in the harmonic equation.
If $\partial M=\emptyset$, then $(\mathrm{rot}(\omega),Z)=0$ for harmonic $Z$, so the viscous term $\nu\mathrm{rot}(\omega)$ acts only on the rotational component and drops out of the harmonic equation.
\end{remark}

The derivation is reversible: for a vorticity solution, $V=\mathrm{rot}(\psi)+H$ satisfies $\mathrm{curl}(V)=\omega$ and $\langle V,T\rangle\vert_{\partial M}=0$ by the kinematic equation, hence $V\vert_{\partial M}=0$ (made rigorous in \Cref{cor:LsJL2ToJH10}).
Moreover, by the $H^1$ version of \eqref{eq:HodgeAss}, every $W\in J_{H^1_0}$ splits as $W=\mathrm{rot}(\phi)+Z$ with $\phi\in H^2\cap H^1_0$ and $Z\in H_N$ admissible in \Cref{def:HarmStream}, so that undoing the above splitting recovers the velocity equation.

\begin{summarybox}{Equivalence for Navier--Stokes}
Under the assumptions made precise in \Cref{thm:EquiForm}, the transformations $V\longmapsto(\psi,H,\omega)$ with $\omega=\mathrm{curl}(V)$, where $V=\mathrm{rot}(\psi)+H$, and $(\psi,H,\omega)\longmapsto V=\mathrm{rot}(\psi)+H$ map solutions of the Navier--Stokes velocity and vorticity formulations to one another.
\end{summarybox}

\subsection{Vorticity formulation for Euler}
\label{sec:MainContriVortEu}
The Euler derivation follows the same pattern, with $\nu=0$, but the weaker boundary condition changes one essential point.
Since Euler imposes only the no-penetration condition (\Cref{sec:MainContri:velo}), the tangential boundary term $\langle\langle V,T\rangle,\chi\rangle_{\partial M}$ remains when integrating the identity $\omega=\mathrm{curl}(V)$ by parts, unless the scalar test function $\chi$ has zero trace.
Consequently, the Euler kinematic test space is smaller, $\mathbb{X}_{\mathrm{kin},0}\subsetneq\mathbb{X}_{\mathrm{kin}}$, and for zero-trace test functions $\chi$ we obtain
\begin{align*}
    (\omega,\chi)=(V,\mathrm{rot}(\chi))=(\mathrm{rot}(\psi),\mathrm{rot}(\chi)),
\end{align*}
using $H\perp\mathrm{rot}(\chi)$.

Testing the Euler momentum equation \eqref{eq:proofLsEq} with rotational fields $\mathrm{rot}(\phi)$, $\phi\in\mathbb{X}_{\omega,0}$ (as in \Cref{sec:MainContriVort}; note $\mathbb{X}_{\omega,0}\subset\mathbb{X}_{\mathrm{kin},0}$), and integrating by parts gives the transport equation
\begin{align*}
    \partial_t(\omega,\phi)=(\mathrm{curl}(F)-V(\omega),\phi),\qquad V(\omega)\coloneqq\langle V,\mathrm{grad}(\omega)\rangle,
\end{align*}
where we used $\mathrm{curl}(\omega JV)=\mathrm{div}(\omega V)=V(\omega)$ for divergence-free $V$.
All terms in this equation, including $\partial_t\omega$, are functions rather than functionals, so it can be written as $(\partial_t\omega,\phi)=(\mathrm{curl}(F)-V(\omega),\phi)$ and tested with a larger space $\mathbb{X}_\omega$ of scalar functions without boundary condition (see \Cref{def:HarmStreamEu}).
The harmonic equation is obtained exactly as before, simply with the viscous terms removed.

\begin{summarybox}{Euler vorticity formulation}
Seek $(\psi,H,\omega)$ in suitable spaces, set $V=\mathrm{rot}(\psi)+H$, and require
\begin{align*}
    (\omega,\chi)&=(\mathrm{rot}(\psi),\mathrm{rot}(\chi)), &&\forall\chi\in\mathbb{X}_{\mathrm{kin},0},\\
    (\partial_t\omega,\phi)&=(\mathrm{curl}(F)-V(\omega),\phi), &&\forall\phi\in\mathbb{X}_\omega,\\
    \partial_t (H,Z)&=(F-\omega JV,Z), &&\forall Z\in\mathbb{X}_H,
\end{align*}
with initial conditions $(H_t,\omega_t)\vert_{t=0}=(H_0,\omega_0)$; on a closed surface, $\omega$ is additionally required to have zero mean.
The precise spaces, the additional surface assumption needed in the Euler case, and the initial-data convention are given in \Cref{def:HarmStreamEu}.
\end{summarybox}

For Euler it is natural to prescribe $(H_0,\omega_0)$; the initial streamfunction is then reconstructed from the kinematic equation.
Again, the derivation is reversible.

\begin{summarybox}{Equivalence for Euler}
Under the assumptions made precise in \Cref{thm:EquiFormEu}, the transformations $V\longmapsto(\psi,H,\omega)$ with $\omega=\mathrm{curl}(V)$, where $V=\mathrm{rot}(\psi)+H$, and $(\psi,H,\omega)\longmapsto V=\mathrm{rot}(\psi)+H$ map solutions of the Euler velocity and vorticity formulations to one another.
\end{summarybox}

\begin{remark}[Test spaces for Navier--Stokes and Euler]
\label{rem:TestSpaces}
The roles of the scalar test spaces are exchanged between the two cases: Navier--Stokes tests the kinematic equation with $\mathbb{X}_{\mathrm{kin}}$ (no boundary condition) and the vorticity equation with $\mathbb{X}_{\omega,0}$ (zero trace), whereas Euler uses $\mathbb{X}_{\mathrm{kin},0}$ and $\mathbb{X}_\omega$.
This reflects the different boundary conditions (cf.\ \cite[Sec.~1.3]{ChorinMarsden}): under no-slip, the kinematic equation can be tested without boundary condition and thereby imposes $\langle V,T\rangle\vert_{\partial M}=0$, whereas the viscous term $\nu\mathrm{rot}(\omega)$ requires zero-trace test functions in the vorticity equation, since otherwise Green's formula yields the boundary term $\nu\langle\langle\mathrm{grad}(\omega),N\rangle,\phi\rangle_{\partial M}$, a flux of $\omega$ tied to the eliminated pressure (cf.\ $\mathrm{rot}(\phi)\in J_{L^2}$ in \Cref{sec:MainContriVort}); under no-penetration, it is the other way around (\Cref{sec:MainContriVortEu}).

In a finite element discretization (\Cref{sec:AbstractSetting}) with a zero-trace streamfunction space $\mathbb{S}_h$ and a vorticity space $\mathbb{W}_h\supset\mathbb{S}_h$ without boundary condition, this exchange decides which discrete equation determines which unknown.
For Navier--Stokes, the kinematic equation, tested with $\mathbb{W}_h$, determines $\omega_h$ from $V_h$ (a mass-matrix problem), and the vorticity equation, tested with $\mathbb{S}_h$, evolves $\psi_h$; for Euler, the vorticity equation, tested with $\mathbb{W}_h$, evolves $\omega_h$, and the kinematic equation, tested with $\mathbb{S}_h$, determines $\psi_h$ (a discrete Laplace problem).
In both cases this gives $\dim\mathbb{S}_h+\dim\mathbb{W}_h$ equations for $(\psi_h,\omega_h)$, complemented by $b_1$ equations for $H_h$ in the space $\mathbb{H}_{h,N}$ of discrete harmonic fields: the test functions lost in the Euler kinematic equation are exactly compensated by the larger test space of the vorticity equation.
\end{remark}

\subsection{Pressure recovery}
\label{sec:MainContriPres}
The vorticity formulations of \Cref{sec:MainContriVort,sec:MainContriVortEu} determine $(\psi,H,\omega)$ without the pressure.
To recover it, we return to the velocity--pressure formulation \eqref{eq:NavStokesP},
\begin{align*}
    -(p,\mathrm{div}(W))=\langle F-\partial_tV,W\rangle-(\nabla_VV,W)-2\nu(D_V,\nabla W),\qquad\forall W\in H^1_0(\Sec),
\end{align*}
whose test fields are not divergence-free, so that neither the pressure nor the gradient part of the convective term drops out.
Inserting the identities \eqref{eq:keyidentities}, $\nabla_VV=\omega JV+\frac{1}{2}\mathrm{grad}(\vert V\vert^2)$ and $2\delta D_V=\mathrm{rot}(\omega)-2\kappa V$, and writing $-(p,\mathrm{div}(W))=(\mathrm{grad}(p),W)$ gives, formally,
\begin{align*}
    (\partial_tV,W)=\Big(F-\mathrm{grad}(p)-\omega JV-\frac{1}{2}\mathrm{grad}(\vert V\vert^2)-\nu\mathrm{rot}(\omega)+2\nu\kappa V,W\Big),\qquad\forall W\in H^1_0(\Sec),
\end{align*}
with $\nu=0$ for Euler.
The two gradient terms are combined into the \emph{total pressure}
\begin{align*}
    p^*\coloneqq p+\frac{\vert V\vert^2}{2}-\frac{\Vert V\Vert_{L^2}^2}{2\mathrm{Vol}(M)},
\end{align*}
where the constant is chosen such that $p^*$, like $p$, has zero mean.
This yields the momentum equation
\begin{align}
    (\partial_tV,W)=(F-\mathrm{grad}(p^*)-\omega JV-\nu\mathrm{rot}(\omega)+2\nu\kappa V,W),\qquad\forall W\in H^1_0(\Sec),
    \label{eq:momPstar}\tag{$P^*$}
\end{align}
which reduces to \eqref{eq:proofLsEq} for divergence-free $W$, since then $(\mathrm{grad}(p^*),W)=0$.

To isolate $p^*$, we test \eqref{eq:momPstar} with gradient fields $W=\mathrm{grad}(q)$, where the scalar function $q$ carries no boundary condition.
The rigorous justification of this step is given in the proofs of \Cref{thm:EquiForm,thm:EquiFormEu}.
As $V$ is divergence-free and tangential to $\partial M$, integration by parts gives $(V,\mathrm{grad}(q))=0$ and hence $(\partial_tV,\mathrm{grad}(q))=\partial_t(V,\mathrm{grad}(q))=0$.
What remains is a Poisson problem for $p^*$ with natural boundary conditions.

\begin{summarybox}{Pressure Poisson equation}
Given $(\psi,H,\omega)$ with $V=\mathrm{rot}(\psi)+H$, the total pressure $p^*=p+\frac{\vert V\vert^2}{2}-\frac{\Vert V\Vert_{L^2}^2}{2\mathrm{Vol}(M)}$ with zero mean is determined by
\begin{align*}
    (\mathrm{grad}(p^*),\mathrm{grad}(q))=(F-\omega JV-\nu\mathrm{rot}(\omega)+2\nu\kappa V,\mathrm{grad}(q)),\qquad\forall q\in\mathbb{X}_p,
\end{align*}
for Navier--Stokes ($\nu>0$) and Euler ($\nu=0$) alike.
The precise test space $\mathbb{X}_p$ and the regularity of $p^*$ are given in \Cref{def:HarmStream,def:HarmStreamEu}.
\end{summarybox}

Writing $G\coloneqq F-\omega JV-\nu\mathrm{rot}(\omega)+2\nu\kappa V$ for the right-hand side of the pressure equation, this is the weak form of $\Delta_{\mathrm{dR}}p^*=-\mathrm{div}(G)$ with the natural (Neumann) boundary condition $\langle\mathrm{grad}(p^*),N\rangle=\langle G,N\rangle$ on $\partial M$.

The technical work needed to justify these formal manipulations is deliberately postponed: in \Cref{sec:Preliminaries} we establish the geometric and functional-analytic tools, in \Cref{sec:reg} we prove the required regularity of the velocity formulations, and in \Cref{sec:mixed} we give the precise definitions, well-definedness results, and rigorous equivalence theorems.

%% file: content/02Preliminaries.tex
\section{Preliminaries}
\label{sec:Preliminaries}

Before proceeding to the rigorous proofs of the equivalence theorems, we must first collect and generalize the necessary technical tools.
Although to our knowledge some of these results have not previously been shown for surfaces in this generality, we postpone several of the more technical proofs to \Cref{sec:proofs} in order to maintain the flow of the main argument.

We begin by focusing on the functional-analytic framework, establishing Sobolev embeddings, interpolation inequalities, and crucial estimates for the convective and diffusive terms (see \Cref{sec:TheoryEquivThm}).
Following this, we review the Hodge decomposition on general surfaces (\Cref{sec:HodgeDecomp}). This decomposition underpins the entire vorticity formulation by decomposing any tangential, divergence-free vector field into a streamfunction and a harmonic component.
Function spaces are explained briefly where they first appear; their precise definitions are collected in \Cref{sec:not}.

\subsection{Key Estimates and Identities}
\label{sec:TheoryEquivThm}

\paragraph{General Principle.} Throughout this section and the following, we rely on the standard principle that many results for bounded domains in $\mathbb{R}^2$ carry over to compact surfaces via a partition of unity argument.
We do not repeat this justification for subsequent theorems where it applies.

We begin with a key estimate that is a specific consequence of the general Sobolev embedding theorem.
Here and below, $L^p(\Sec^{(r,s)})$ and $W^{m,p}(\Sec^{(r,s)})$, with $H^m\coloneqq W^{m,2}$, denote the Lebesgue and Sobolev spaces of $(r,s)$-tensor fields on $M$; for scalar functions we omit the argument and write, e.g., $L^q$ or $W^{1,r}$ (see \Cref{sec:not}).
For any $A\in H^1(\Sec^{(r,s)})$, we have
\begin{align}
    \Vert A \Vert_{L^p} &\leq C(p,M)\Vert A\Vert_{H^1},~1\leq p<\infty, \label{eq:SobEmbLp} \tag{$L^p_\leq$}
\end{align}
where $C(p,M)$ is a positive constant.
This inequality is essential for our formulation, both for handling general Sobolev scales and for controlling nonlinear terms like convection.

Next, we apply the Sobolev embedding and Hölder's inequality to analyze the vorticity space $W^{1,r}$. This analysis is necessary, for example, to justify the spatial well-definedness of the equations in the vorticity formulations.

\begin{lemma}[Embedding of $W^{1,r}$]
\label{lem:s&r}
\nextline Let $s\in (2,\infty)$ with dual exponent $r\in(1,2)$, i.e., $\frac{1}{s}+\frac{1}{r}=1$.
Then the embeddings $W^{1,r}\hookrightarrow L^q\hookrightarrow L^2$ hold for $\frac{1}{q}=\frac{1}{2}-\frac{1}{s}$.
Consequently, for any $w\in W^{1,r}$, $X\in L^2(\Sec)$, and $Y\in L^s(\Sec)$, the product $w\cdot \langle X,Y\rangle$ belongs to $L^1(\Sec)$.
\end{lemma}

Another important inequality is the Gagliardo--Nirenberg interpolation inequality, which we need in the upcoming \Cref{thm:reg(V.d)V} to derive regularity results for $\vert V\vert^2$.
It follows from the corresponding results for bounded domains in $\mathbb{R}^2$ \cite{GagNirenberg1,GagNirenberg2} combined with the Poincaré inequality.
For completeness, the proof is given in \Cref{sec:proofs}.

\begin{theorem}[Gagliardo--Nirenberg interpolation inequality]
\label{thm:GagNirenberg}
\nextline Let $a\in [1,\infty]$, $p\in[1,\infty)$ and $\theta\in [0,1]$ satisfy the relation $\frac{2}{p}=\frac{2}{a}-\theta$.
Then there exists a positive constant $C(p,\theta,M)$ such that for any $T\in W^{1,a}(\Sec^{(r,s)})$, the following estimate holds
\begin{align}
    \Vert T\Vert_{L^p}\leq C(p,\theta,M)\left(\Vert \nabla T\Vert_{L^a}^\theta\Vert T\Vert_{L^a}^{1-\theta}+\Vert T\Vert_{L^a}\right).
    \label{eq:GagNirenberg}
    \tag{GN$_\leq$}
\end{align}
Furthermore, if $T\in W^{1,a}_0$ (zero trace if $\partial M\neq\emptyset$; zero mean for functions if $\partial M=\emptyset$, see \Cref{sec:not}), this estimate tightens to
\begin{align}
    \Vert T\Vert_{L^p}\leq C(p,\theta,M)\Vert \nabla T\Vert_{L^a}^\theta\Vert T\Vert_{L^a}^{1-\theta},
    \label{eq:GagNirenbergPoincare}
    \tag{GN$^0_{\leq}$}
\end{align}
provided that either (i) $\partial M\neq \emptyset$ and $a\in (1,\infty)$ or (ii) $\partial M=\emptyset$, $a\in [1,\infty)$ and $r=0=s$.
\end{theorem}

Next, we present a generalization of Gaffney's inequality \cite{gaffney1951harmonic}.
Here, $H(\mathrm{div})$ and $H(\mathrm{curl})$ denote the graph spaces of all $X\in L^2(\Sec)$ with $\mathrm{div}(X)\in L^2$, respectively $\mathrm{curl}(X)\in L^2$, and $H_0(\mathrm{div})$, $H_0(\mathrm{curl})$ the kernels of the normal trace $\langle X,N\rangle\vert_{\partial M}$, respectively the tangential trace $\langle X,T\rangle\vert_{\partial M}$ with $T\coloneqq JN$ (see \Cref{sec:not}).
For vector fields $X \in H_0(\mathrm{div})\cap H_0(\mathrm{curl})$, this result explicitly bounds the $L^2$ norm of $\nabla X$ by the norms of $\mathrm{div}(X)$ and $\mathrm{curl}(X)$. 
The proof employs the de Rham Laplacian, defined as $\Delta_{\mathrm{dR}}(V)\coloneqq \mathrm{rot}(\mathrm{curl}(V))-\mathrm{grad}(\mathrm{div}(V))$.

This inequality is of fundamental importance for the Navier--Stokes equations. 
Due to the no-slip boundary condition, the velocity field $V$ naturally belongs to this function space.
Since $V$ is also divergence-free, this estimate allows us to derive regularity based solely on its vorticity -- a crucial step in proving the equivalence of the formulations.

\begin{lemma}[Gaffney's inequality]
\label{lem:GaffneysIneq}
\nextline There exists a positive constant $C(M)$ such that for all $X\in H_0(\mathrm{div})\cap H_0(\mathrm{curl})$, the following estimates hold.
If $\partial M\neq \emptyset$, then
\begin{align}
    \Vert \nabla X\Vert_{L^2}\leq C(M) \left(\Vert\mathrm{div}(X) \Vert_{L^2}^2+\Vert\mathrm{curl}(X)\Vert_{L^2}^2\right)^\frac{1}{2}.
    \label{eq:GaffneyIneq}
    \tag{G$^0_{\leq}$}
\end{align}
If $\partial M=\emptyset$, an additional $L^2$-term is required to control $\mathrm{ker}(\mathrm{div})\cap \mathrm{ker}(\mathrm{curl})$, leading to
\begin{align}
    \Vert \nabla X\Vert_{L^2}\leq C(M) \left(\Vert X\Vert_{L^2}^2+\Vert\mathrm{div}(X) \Vert_{L^2}^2+\Vert\mathrm{curl}(X)\Vert_{L^2}^2\right)^\frac{1}{2}.
    \label{eq:GenGaffneyIneq}
    \tag{G$_\leq$}
\end{align}
In particular, these estimates imply the equality $H^1_0(\Sec)=H_0(\mathrm{div})\cap H_0(\mathrm{curl})$.
\end{lemma}

\begin{proof}
We first prove the general estimate \eqref{eq:GenGaffneyIneq}. 
By standard density results, it suffices to show the inequality for smooth vector fields $V \in \Sec_0$, i.e., smooth tangential fields with compact support in the interior of $M$ (see \Cref{sec:not}).
Using integration by parts -- noting that boundary terms vanish since $V\in \Sec_0$ -- followed by the Weitzenböck formula on surfaces, $\delta \nabla V = \Delta_{\mathrm{dR}}V - \kappa V$ (see, e.g., \cite[p.26]{schwarz2006hodge}), we obtain
\begin{align*}
    \Vert\nabla V\Vert_{L^2}^2&=(\delta \nabla V,V)_{L^2}=(\Delta_{\mathrm{dR}}V,V)_{L^2}-(\kappa V,V)_{L^2}
    \\&=\Vert\mathrm{div}(V)\Vert_{L^2}^2+\Vert\mathrm{curl}(V)\Vert_{L^2}^2-(\kappa V,V)_{L^2}.
\end{align*}
Since $M$ is compact, the negative Gaussian curvature $-\kappa$ is bounded from above.
This implies $-(\kappa V,V)\leq K(M)\Vert V\Vert^2$, proving \Cref{eq:GenGaffneyIneq} and consequently the equality $H^1_0(\Sec)=H_0(\mathrm{div})\cap H_0(\mathrm{curl})$.
\\
To prove the boundary estimate \eqref{eq:GaffneyIneq} when $\partial M\neq\emptyset$, we argue by contradiction.
Assume the inequality does not hold. 
Then there exists a sequence $Y_n\in H^1_0(\Sec)$ such that $\Vert \nabla Y_n\Vert_{L^2}=1$ and $\Vert \mathrm{div}(Y_n)\Vert_{L^2}+\Vert \mathrm{curl}(Y_n)\Vert_{L^2}<\frac{1}{n}$.
By Rellich's lemma and the general estimate \eqref{eq:GenGaffneyIneq}, we extract a subsequence that is Cauchy in $\Vert \cdot\Vert_{H^1}$.
Its limit $Y\in H^1_0(\Sec)$ must satisfy $\mathrm{div}(Y)=0=\mathrm{curl}(Y)$, but $\Vert \nabla Y\Vert_{L^2}=1$.\\
If $\partial M \in C^\infty$, this immediately contradicts \cite[Thm 3.4.4]{schwarz2006hodge}, which states that the only harmonic field vanishing on the boundary is zero.
For the general case $\partial M \in C^{0,1}$, we utilize the ambient manifold $W$ from our definition of such a surface (cf.\ \Cref{def:surface}).
We choose a larger surface $\tilde{M}$ with smooth boundary such that $M \subset \tilde{M} \subset W$. 
Extending $Y$ by zero yields a field $\tilde{Y}$ on $\tilde{M}$ that retains all key properties: $\tilde{Y}\in H^1_0(\Sec,\tilde{M})$, $\mathrm{div}(\tilde{Y})=0$, $\mathrm{curl}(\tilde{Y})=0$, and $\Vert \nabla \tilde{Y}\Vert_{L^2}=1$.
Applying \cite[Thm 3.4.4]{schwarz2006hodge} to $\tilde{M}$ yields $\tilde{Y}=0$, again a contradiction. Thus, the estimate \Cref{eq:GaffneyIneq} must hold.
\end{proof}

A very practical consequence of Gaffney's inequality is the following corollary, which is of great importance for the Navier--Stokes case in proving the equivalence of the formulations. It allows us to relate velocity fields in $J_{H^1_0}$ to those with less regularity in $L^s(\Sec)\cap J_{L^2}$. In particular, the ability to recover higher regularity from lower regularity assumptions is a key feature.

\begin{corollary}[Connection between $J_{H^1_0}$ and $L^s(\Sec)\cap J_{L^2}$]
\label{cor:LsJL2ToJH10}
\nextline Let $s\in (2,\infty)$ with dual exponent $r\in(1,2)$, i.e., $\frac{1}{s}+\frac{1}{r}=1$.
Then the following equality of sets holds
\begin{align*}
    \{(V,\omega)\in &J_{H^1_0}\times L^2\vert~(\mathrm{curl}(V),\chi)=(\omega,\chi),~\forall \chi \in L^2\}
    \\&=\{(V,\omega)\in \left(L^s(\Sec)\cap J_{L^2}\right)\times L^2\vert~(V,\mathrm{rot}(\chi))=(\omega,\chi),~\forall \chi \in W^{1,r}\}.
\end{align*}
\end{corollary}

\begin{proof}
This proof follows the same steps as \cite[Lem III.2.1]{Navier} and is provided in \Cref{sec:proofs} for completeness.
\end{proof}

We next present a lemma and theorem crucial for analyzing the regularity of the nonlinear term $\nabla_VV$. 
We begin with the decomposition that is central to our earlier symbolic derivation, as it introduces the vorticity while isolating a gradient component.

\begin{lemma}[name={[Nonl.\ term]Nonlinear term}]
\label{lem:Nonlinterm}
\nextline Let $V\in W^{1,p}(\Sec)$ for $p\in \left[\frac{4}{3},\infty\right]$, then
\begin{align*}
    \nabla_VV=\mathrm{curl}(V)JV+\frac{1}{2}\mathrm{grad}\left(\vert V\vert^2\right)\in L^1(\Sec).
\end{align*}
\end{lemma}

\begin{proof}
For smooth vector fields $V,X\in \Sec$, we compute
\begin{align*}
    \frac{1}{2}\langle \mathrm{grad}(\vert V\vert^2),X\rangle=\frac{1}{2}X(\vert V\vert^2)=\langle \nabla_X V,V\rangle=\langle X,\nabla_VV+[(\nabla V)^T-(\nabla V)](V)\rangle.
\end{align*}
Here, $(\nabla V)^T$ denotes the metric transpose of $\nabla V\colon X\mapsto\nabla_XV$, i.e., $\langle\nabla_XV,Y\rangle=\langle X,(\nabla V)^T(Y)\rangle$ (used with $Y=V$ and $(\nabla V)(V)=\nabla_VV$), so it remains to show $(\nabla V)^T-\nabla V=-\mathrm{curl}(V)J$.
Since $(\nabla V)^T-\nabla V$ is antisymmetric and the antisymmetric endomorphisms of a two-dimensional tangent plane are spanned by $J$, we have $(\nabla V)^T-\nabla V=cJ$ for some function $c$.
To identify $c$, we use $\mathrm{div}(X)=\mathrm{tr}(\nabla X)=\sum_{i=1}^2\langle\nabla_{E_i}X,E_i\rangle$ for any vector field $X$ and any local orthonormal frame $(E_1,E_2)$, $\mathrm{curl}(V)=-\mathrm{div}(JV)$, and that $J$ is parallel, antisymmetric and satisfies $J^2=-\mathrm{id}$ (\Cref{def:Hodge*}).
Hence $\mathrm{tr}(J\nabla V)=\mathrm{tr}(\nabla(JV))=\mathrm{div}(JV)$ and $\mathrm{tr}(J(\nabla V)^T)=-\mathrm{tr}(J\nabla V)$, so
\begin{align*}
    -2c=\mathrm{tr}(cJ^2)=\mathrm{tr}\big(J(\nabla V)^T\big)-\mathrm{tr}(J\nabla V)=-2\,\mathrm{div}(JV)=2\,\mathrm{curl}(V).
\end{align*}
Thus $c=-\mathrm{curl}(V)$, which proves the desired equation for $V\in \Sec$.
Using a density argument combined with the Sobolev embedding for $p=\frac{4}{3}$ we conclude that the equality holds at least in $L^1(\Sec)$.
All remaining cases of $p$ follow from Hölder's inequality due to the compactness of $M$.
\end{proof}

In the previous lemma, the condition $V\in W^{1,p}(\Sec)$ guarantees that $\vert V\vert^2$ is at least in $W^{1,1}$.
More precisely, for $V \in H^1(\Sec)$, a simple application of Hölder's inequality and the Sobolev embedding \eqref{eq:SobEmbLp} yields
$$\Vert\mathrm{grad}(\vert V\vert^2)\Vert_{L^r}\leq 2\Vert \nabla V\Vert_{L^2} \Vert V\Vert_{L^q}\leq C(M)\Vert V\Vert_{H^1}^2,$$
for any $r\in (1,2)$ and $q\in(2,\infty)$ such that $\frac{1}{r}=\frac{1}{2}+\frac{1}{q}$.
This implies $\vert V\vert^2\in W^{1,r}$.\\
However, the situation is more involved when $V \in L^2_T(J_{H^1_0})\cap L^\infty_T(J_{L^2})$ and we aim to prove that $\nabla_VV\in L^2_T(L^r(\Sec))$, where $L^p_T(B)$ denotes the Bochner space of $L^p$-integrable functions $(0,T)\to B$ with values in a Banach space $B$ (see \Cref{sec:not}).
In this Bochner space setting, the spatial estimates from before are insufficient.
To resolve this, we need to assume additional regularity for $\mathrm{curl}(V)$.
The core idea is that via the Gagliardo--Nirenberg interpolation inequality (\Cref{thm:GagNirenberg}) we can control the problematic term $\Vert V\Vert_{L^q}$ using $\Vert \nabla V\Vert_{L^2}$. 
Then, since $\mathrm{div}(V)=0$, Gaffney's inequality (\Cref{lem:GaffneysIneq}) allows the additional vorticity regularity to control this resulting term.
We rigorously establish this result in the next theorem.

\begin{theorem}[$\vert V\vert^2$ regularity]
\label{thm:reg(V.d)V}
\nextline Let $r\in(1,2)$, $T\in(0,\infty)$, $V\in  L^2_T(J_{H^1_0})\cap L^\infty_T(J_{L^2})$ and $\mathrm{curl}(V)\in  L^2_T(W^{1,r})\cap L^\infty_T(L^r)$.
Then $\vert V\vert^2\in L^2_T(W^{1,r})$.
In particular this implies that $\nabla_VV\in L^2_T(L^r(\Sec))$.
\end{theorem}

\begin{proof}
The proof follows the strategy outlined above; full details are provided in \Cref{sec:proofs}.
\end{proof}

Lastly, we present the theorem used to express the viscosity term $2\nu(D_V,\nabla W)$ in terms of vorticity and curvature.

\begin{theorem}[$\delta D_V$]
\label{thm:IbPDV}
\nextline For any $V,W\in H^1(\Sec)$, the equality $(D_V,D_W)=(D_V,\nabla W)$ holds.
Furthermore, if $V\in J_{H^1_0}$ and $\mathrm{curl}(V)\in W^{1,r}$ for some $r\in(1,2)$, we have
\begin{align}
    2(D_V,\nabla W)=(\mathrm{rot}(\mathrm{curl}(V))-2\kappa V,W),
    \label{eq:divDV}
\end{align}
for all $W\in H^1_0(\Sec)$.
\end{theorem}

\begin{proof}
First, for any $V,W\in \Sec$ the identity $(D_V,D_W)=(D_V,\nabla W)$ follows from the fact that $D_V$ is symmetric.
This is easily seen for smooth fields and extends to $H^1(\Sec)$ by density.

Next, we prove \eqref{eq:divDV} following an idea from \cite[p.89]{Navier}.
Let $W\in \Sec_0$.
Using the first part of the proof and Green's formula, we have
\begin{align*}
    2(D_V,\nabla W)&=2(\delta(D_W),V)=(\Delta_{\mathrm{dR}}W,V)-2(\kappa W,V)
    \\&=(\mathrm{rot}(\mathrm{curl}(V)),W)-2(\kappa V,W).
\end{align*}
The second equality uses $(2\delta (D_W),V)=(\Delta_{\mathrm{dR}}W-2\kappa W,V)$, which follows from the identities
\begin{align*}
    2\delta (D_W)=\delta(\nabla W)-\kappa W-\mathrm{grad}(\mathrm{div}(W))
    \qquad\text{and}\qquad
    \delta (\nabla W)=\Delta_{\mathrm{dR}}W-\kappa W,
\end{align*}
see \cite[Lem 2.1; II.E]{NavStokesSphere,MembraneDyna} and \cite[p.26]{schwarz2006hodge}, together with
\begin{align*}
    (\mathrm{grad}(\mathrm{div}(W)),V)=-(\mathrm{div}(W),\mathrm{div}(V))=0,
\end{align*}
since $\mathrm{div}(V)=0$ and $W$ is compactly supported.
Finally, a density argument combined with the Sobolev embedding \eqref{eq:SobEmbLp} extends this result to all $W\in H^1_0(\Sec)$.
\end{proof}

\subsection{Hodge decomposition}
\label{sec:HodgeDecomp}
We now turn to the Hodge decomposition, which forms the geometric backbone of all vorticity equations discussed in this work. 
As the earlier symbolic derivation demonstrates, this decomposition is essential for splitting the velocity field into a scalar streamfunction and a harmonic component. 
While this fundamental geometric tool is well-known, precise statements for general $H^k(\Sec)$ spaces on surfaces with a $C^{k,1}$ boundary can be surprisingly difficult to locate in the literature.

To properly set the stage for this theorem, we first need a deeper understanding of the space of harmonic fields $H_N$. We therefore begin by deriving a stronger version of Gaffney's inequality (\Cref{lem:GaffneysIneq}) that relies on relaxed boundary conditions.

\begin{theorem}[Embedding of $H_0(\mathrm{div})\cap H(\mathrm{curl})$]
\label{thm:H0divHcurlInH1}
\nextline Let $M$ be a surface with $\partial M\in C^{k,1}$ for $k \geq 1$. 
Then the following continuous embedding holds
$$ \{V\in H_0(\mathrm{div})\cap H(\mathrm{curl})\vert~\mathrm{div}(V),\mathrm{curl}(V)\in H^{k-1}\}\hookrightarrow H^{k}(\Sec).$$
Define the space of harmonic fields as $H_N\coloneqq \{X\in J_{L^2}\cap H(\mathrm{curl})\vert~\mathrm{curl}(X)=0\}$.
Then $H_N$ is a finite-dimensional subspace of $H^k(\Sec)$.
\end{theorem}

\begin{proof}
For the base case $k=1$, the result for bounded domains in $\mathbb{R}^2$ is established in \cite[Prop I.3.1]{Navier}. 
Crucially, this result also holds for general Riemannian metrics, which allows the extension to surfaces via a standard partition of unity argument.
For $k\geq 2$, the result is obtained by combining the base case with higher regularity results for the Dirichlet and Neumann Laplace equations (see \cite[Thm I.1.8, I.1.10]{Navier} and \cite[Lem 3.4.7, Cor 3.4.8]{schwarz2006hodge}).
\\
To prove the finite dimensionality of $H_N$, we apply this theorem (with $k=1$) to any $H \in H_N$. 
Since $\mathrm{div}(H)=0$ and $\mathrm{curl}(H)=0$, we obtain the norm equivalence $\Vert H\Vert_{H^1}\leq C(M)\Vert H\Vert_{L^2}$.
This estimate, combined with Rellich's lemma, implies that the unit ball of $H_N$ in $L^2(\Sec)$ is compact. 
Therefore, $H_N$ must be finite-dimensional.
\end{proof}

\begin{remark}
The dimension of the space of harmonic fields is determined solely by a topological invariant of the surface: $\mathrm{dim}(H_N)=b_1(M)$, where $b_1(M)$ is the first Betti number. 
Even for surfaces with only $C^{0,1}$ boundary, one can show that $H_N$ remains finite-dimensional and satisfies this property, though the arguments are more involved (see, e.g., \cite[Thm 11.1]{mitrea2001layer}).
\end{remark}

\begin{theorem}[Hodge decomposition]
\label{thm:HodgeDecomp}
\nextline Let $M$ be a surface with a $C^{k,1}$ boundary for $k\in \mathbb{N}_0$.
Then the following unique $L^2$-orthogonal decomposition holds
\begin{align*}
    &H^k(\Sec)=\mathrm{grad}(H^{k+1}\cap L^2_0)\oplus \mathrm{rot}(H^{k+1}\cap H^1_0)\oplus H^k(\Sec)\cap H_N.
\end{align*}
For $k=0$, the final summand is simply $H_N$.
For every $k\geq1$, \Cref{thm:H0divHcurlInH1} gives $H_N\subset H^k(\Sec)$, so the final summand again equals $H_N$.
\end{theorem}

\begin{proof}
For surfaces with $C^{0,1}$ boundary, the base case ($k=0$) is a consequence of \cite[Thm 8.2]{Mitrea2008}, obtained by dualizing their results for differential forms to vector fields.
When $k\geq 1$ and $V\in H^k(\Sec)$ we use this base case to write $V=\mathrm{grad}(q)+\mathrm{rot}(\psi)+H$ for $q\in H^1\cap L^2_0$, $\psi\in H^1_0$ and $H\in H_N$.
The regularity of the harmonic component, $H\in H^k(\Sec)$, follows immediately from \Cref{thm:H0divHcurlInH1}.
For the potentials $q$ and $\psi$, applying $\mathrm{div}$ and $\mathrm{curl}$ to $V$ reveals that they solve Neumann and Dirichlet Laplace problems, respectively, with source terms in $H^{k-1}$.
Consequently, standard elliptic regularity results for surfaces (see \cite[Thm I.1.8, I.1.10]{Navier} and \cite[Lem 3.4.7, Cor 3.4.8]{schwarz2006hodge}) imply that $q\in H^{k+1}$ and $\psi\in H^{k+1}$.
\end{proof}

\begin{remark}
We conclude with three important notes on the scope of the Hodge decomposition.\\
First, the standard $H^1(\Sec)$ decomposition used in our main theorems (e.g., \Cref{thm:EquiForm} and \Cref{thm:EquiFormEu}) generally requires a $C^{1,1}$ boundary and does not hold for $C^{0,1}$ boundaries.\\
Second, this requirement can be relaxed in specific cases. 
For $s\in(2,4]$, stronger regularity results for the Laplace problem on $C^{0,1}$ surfaces exist \cite{GENG20122427,shen2005bounds}. 
By arguments similar to those used above, these results imply a corresponding $L^s(\Sec)$ Hodge decomposition, which would allow proving \Cref{thm:EquiForm} for this range of $s$ without assuming a $C^{1,1}$ boundary.
We do not pursue this generalization here, but mention it for completeness.\\
Finally, if $M$ is topologically a disc with a finite number of disjoint subdiscs removed, the harmonic space $H_N$ can be parameterized by potentials in $\mathrm{rot}(H^{k+1})$.
In this special case, the entire Hodge decomposition can be expressed using purely scalar functions (see \cite[Thm I.3.2]{Navier} and \cite[Prop 5]{FluidCohomology}).
\end{remark}

%% file: content/03Regularity.tex
\section{Regularity of the velocity formulation}
\label{sec:reg}
The equivalence theorems (\Cref{thm:EquiForm,thm:EquiFormEu}) constitute a central result of this work, yet they rely on specific regularity assumptions for the velocity field. 
This section is dedicated to providing the rigorous justification for these conditions.
We first state the velocity formulation precisely and reconstruct the pressure (\Cref{sec:VeloPrecise}); the regularity results follow in \Cref{sec:regNS} and \Cref{sec:regEuler}.

We achieve this by establishing regularity results for the velocity formulation on surfaces, generalizing classical theorems from bounded domains in $\mathbb{R}^2$. This extension necessitates the development of specific analytic tools, such as a generalized Stokes regularity theorem and a Poincaré--Morrey inequality, which are essential for handling the viscous term on general surfaces. Beyond verifying the equivalence, these results also provide the necessary functional-analytic basis for subsequent convergence analysis of the equivalent vorticity formulations.

\subsection{Precise velocity formulation and pressure reconstruction}
\label{sec:VeloPrecise}
We first collect the precise statements underlying the velocity formulation introduced formally in \Cref{sec:MainContri:velo}.
We use standard notation for Sobolev spaces and the Bochner spaces $L^p_T(B)$ introduced in \Cref{sec:TheoryEquivThm}, and write $f_t\coloneqq f(t)$, e.g., $V_t$.
Moreover, $C([0,T];B)$ denotes the continuous functions $[0,T]\to B$, and we allow $T=\infty$, e.g., $L^p_\infty(B)$ and $C([0,\infty);B)$ (see \Cref{sec:not} for details).
Further, $B'$ denotes the dual space of $B$, in particular $H^{-1}(\Sec)\coloneqq (H^1_0(\Sec))'$, with duality pairing $\langle\cdot,\cdot\rangle$.
An equation holds in $(C^\infty_0(0,T))'$ if it holds in the sense of distributions in time, i.e., after multiplication with any $\varphi\in C^\infty_0(0,T)$ and integration over $(0,T)$.

\begin{definition}[name={[Velocity formulation]Velocity formulation {\cite[p.158]{NavierTime}, \cite[p.533]{TaylorPDE3}}}]
\label{def:VeloForm}
\nextline Let $T\in (0,\infty)$ and $V_0\in J_{L^2}$ be given. Then:
\begin{enumerate}[i.)]
    \item Given $F\in L^2_T(H^{-1}(\Sec))$, $V\in L^2_T(J_{H^1_0})\cap L^\infty_T(J_{L^2})$ is a solution of the \emph{nonstationary Navier--Stokes equations in velocity formulation} if, for $\nu>0$, the equation
    \begin{align}
        &\partial_t(V,W)=\langle F,W\rangle-(\nabla_{V}V,W)-2\nu(D_V,\nabla W),~\forall W\in J_{H^1_0} \label{eq:NavStokes}\tag{$\mathrm{NS}_{V}$}
    \end{align}
    holds in $(C^\infty_0(0,T))'$ and $V$ satisfies the initial condition $V_t\vert_{t=0}=V_0$.
    \item Given $F\in L^2_T((H^1(\Sec))')$, $V\in L^\infty_T(H^1(\Sec)\cap J_{L^2})$ is a solution of the \emph{nonstationary Euler equations in velocity formulation} if the equation
    \begin{align}
        &\partial_t(V,W)=\langle F,W\rangle-(\nabla_{V}V,W),~\forall W\in H^1(\Sec)\cap J_{L^2} \label{eq:Euler} \tag{$\mathrm{E}_{V}$}
    \end{align}
    holds in $(C^\infty_0(0,T))'$ and $V$ satisfies the initial condition $V_t\vert_{t=0}=V_0$.
\end{enumerate}
\end{definition}

The time continuity required to interpret the initial conditions is proved in \Cref{lem:VeloForm} below.
For the Navier--Stokes equations, the existence of unique weak solutions in the (no-slip) space $L^2_T(J_{H^1_0})\cap L^\infty_T(J_{L^2})$ is a classical result for flat domains \cite{Temam, NavierTime}, which can be extended to general surfaces with the techniques below.
In contrast, a corresponding well-posedness result does not hold in general for the Euler equations in the (no-penetration) space $L^\infty_T(H^1(\Sec)\cap J_{L^2})$; existence and uniqueness typically require stronger regularity assumptions (cf.\ \Cref{thm:EulerReg}).
We adopt this particular weak definition for the Euler equations to allow for a unified treatment of both cases.

\begin{lemma}[Well-defined initial condition of \Cref{def:VeloForm} (\nameref{def:VeloForm})]
\label{lem:VeloForm}
\nextline Let $V$ be a solution to the velocity formulation of the Navier--Stokes \eqref{eq:NavStokes} or Euler \eqref{eq:Euler} equations. Then $V\in C([0,T];J_{L^2})$.
\end{lemma}

\begin{proof}
For the Navier--Stokes case, the linear terms on the right-hand side of \eqref{eq:NavStokes} define elements of $L^2_T((J_{H^1_0})')$, and the convective term satisfies $\vert(\nabla_VV,W)\vert=\vert(\nabla_VW,V)\vert\leq\Vert V\Vert_{L^4}^2\Vert\nabla W\Vert_{L^2}\leq C(M)\Vert V\Vert_{H^1}\Vert V\Vert_{L^2}\Vert W\Vert_{H^1}$ by the Gagliardo--Nirenberg interpolation inequality \eqref{eq:GagNirenberg} (with $p=4$ and $a=2$), so that it also lies in $L^2_T((J_{H^1_0})')$ for $V\in L^2_T(J_{H^1_0})\cap L^\infty_T(J_{L^2})$ (cf.\ \cite[p.157-159]{NavierTime} for bounded domains in $\mathbb{R}^2$). Thus, identifying $\partial_t V$ with this expression in the sense of the weak Bochner derivative for Gelfand triples, we obtain
\begin{align}
    \partial_tV= F-(\nabla_{V}V,\cdot)-2\nu(D_V,\nabla \cdot) \in L^2_T((J_{H^1_0})').
    \label{eq:dtVNavStokes}
    \tag{$\mathrm{NS}_{\partial_tV}$}
\end{align}
Since $J_{H^1_0}\subset J_{L^2}\subset (J_{H^1_0})'$ is a Gelfand triple and $V\in L^2_T(J_{H^1_0})$, this regularity of $\partial_tV$ implies $V\in C([0,T];J_{L^2})$ (see the criterion recalled in \Cref{sec:not} and \cite[p.150-151]{NavierTime}), ensuring that the evaluation at $t=0$ is well-defined.\\
The same argument applies to the Euler equation \eqref{eq:Euler}, with the modification that the derivative lies in a different dual space
\begin{align}
    \partial_tV= F-(\nabla_{V}V,\cdot)\in L^2_T((H^1(\Sec)\cap J_{L^2})').
    \label{eq:dtVEuler}
    \tag{$\mathrm{E}_{\partial_tV}$}
\end{align}
Since $V\in L^\infty_T(H^1(\Sec)\cap J_{L^2})\subset L^2_T(H^1(\Sec)\cap J_{L^2})$, the same Gelfand triple argument (now for $H^1(\Sec)\cap J_{L^2}\subset J_{L^2}\subset(H^1(\Sec)\cap J_{L^2})'$) yields the corresponding continuity result for the inviscid case.
\end{proof}

The formulations above determine the velocity field; the corresponding pressure can be recovered \emph{a posteriori}.
Its construction relies on two standard results regarding the divergence operator: the first shows that $\mathrm{div}$ is an isomorphism on an appropriate subspace, while the second characterizes functionals that vanish on divergence-free fields.
Since the proofs of these statements for bounded domains in $\mathbb{R}^2$ rely on standard Banach space arguments, they generalize directly to surfaces.

\begin{lemma}[Divergence is an isomorphism]
\label{thm:divIsom}
\nextline Let $H^1_0(\Sec)=J_{H^1_0}\oplus (J_{H^1_0})^\perp$ be the orthogonal decomposition with respect to the full $H^1$ inner product $(\cdot,\cdot)+(\nabla\cdot,\nabla\cdot)$.
Then the divergence operator induces an isomorphism $\mathrm{div}:(J_{H^1_0})^\perp\to L^2_0$.
\end{lemma}

\begin{proof}
For bounded domains in $\mathbb{R}^2$, this is proven in \cite[Cor I.2.4]{Navier}, and the proof carries over to surfaces: $\mathrm{div}:H^1_0(\Sec)\to L^2_0$ is bounded and surjective with kernel $J_{H^1_0}$, so its restriction to the closed complement $(J_{H^1_0})^\perp$ is a bounded bijection onto $L^2_0$ and hence an isomorphism by the open mapping theorem.
\end{proof}

\begin{lemma}[Density results and Pressure]
\label{thm:DecompH-1AndDensity}
\nextline Let $\mathscr{V}\coloneqq \{V\in\Sec_0\mid\mathrm{div}(V)=0\}$. If $F\in H^{-1}(\Sec)$ satisfies $\langle F, W\rangle = 0$ for all $W\in \mathscr{V}$, then there exists a unique $p\in L^2_0$ such that $\langle F,\cdot\rangle= -(p,\mathrm{div}(\cdot))$.
In particular, this implies the density results $J_{L^2}=\overline{\mathscr{V}}^{\Vert\cdot\Vert_{L^2}}$ and $J_{H^1_0}= \overline{\mathscr{V}}^{\Vert\cdot\Vert_{H^1}}$.
\end{lemma}

\begin{proof}
For bounded domains in $\mathbb{R}^2$, the existence of the pressure is established in \cite[Thm I.2.3]{Navier}, while the density results are proven in \cite[Cor I.2.5, Thm I.2.8]{Navier}.
\end{proof}

With these tools at hand, we can state the pressure reconstruction precisely.
We recall that $\delta$ denotes the formal $L^2$-adjoint of the covariant derivative $\nabla$, defined by the relation $(\delta A,B)=(A,\nabla B)$ for all smooth tensor fields $A\in \Sec^{(r,s+1)}_0$ and $B\in \Sec^{(r,s)}_0$ with compact support in the interior of $M$ (see \Cref{sec:not}).

\begin{proposition}[name={[Pressure]Pressure (Navier--Stokes and Euler)}]
\label{prop:pres}
\nextline Let V be a weak solution to \eqref{eq:NavStokes} or \eqref{eq:Euler}, $\nu>0$ or $\nu=0$ respectively.
Then there exists a (distributional) pressure $p\in (C^\infty_0(M\times(0,T)))'$, i.e., a distribution in space and time, solving
\begin{align}
    &\mathrm{grad}(p)=F-\partial_tV-\nabla_VV-2\nu\cdot\delta(D_V). \label{eq:DistribP} \tag{$P'$}
\end{align}
The pressure is unique up to a distribution depending only on time; the construction in \Cref{sec:proofs} selects a distinguished representative, built from spatially mean-zero primitives in $L^2_0$.
Furthermore, if $\partial_tV\in L^2_T(H^{-1}(\Sec))$, then this representative satisfies $p\in L^2_T(L^2_0)$ and the velocity--pressure formulation \eqref{eq:NavStokesP} for a.e.\ $t\in(0,T)$ and all $W\in H^1_0(\Sec)$.
\end{proposition}

We note that while this construction is standard for the Navier--Stokes equations, its application to the Euler equations is less common in this functional setting; consistent with our unified definition, it remains valid and applies analogously.

\begin{proof}[Proof of \Cref{prop:pres}]
The proof strategy corresponds to the standard Navier--Stokes approach found in \cite[p.160]{NavierTime} and relies fundamentally on \Cref{thm:divIsom} and \Cref{thm:DecompH-1AndDensity} established above. For completeness, the full detailed proof is provided in \Cref{sec:proofs}.
\end{proof}

\subsection{Regularity results for the Navier--Stokes equation}
\label{sec:regNS}

In this section, we extend a regularity result for the Navier--Stokes equations by Heywood and Rannacher \cite[p.7-14]{Rannacher1}, which they proved for bounded domains in $\mathbb{R}^2$. However, the original proof relies on the specific characterization of the viscosity term ($2\nu(D_V,\nabla W)=\nu(\nabla V,\nabla W)$) that holds only due to the flatness, meaning it cannot be applied directly to general surfaces.

Thus, our aim is to investigate the general structure of the viscosity term and prove that the underlying techniques can be carried over to surfaces. 
We begin by stating the corresponding theorem and subsequently discuss the required adjustments for the proof.
In preparation for this, we define the finite-dimensional space of Killing fields -- fields whose flows are isometries -- as $\mathcal{K}\coloneqq\{K\in H^1(\Sec)\vert ~D_K=0,~\langle K,N\rangle\vert_{\partial M}=0\}$. 
Note that for closed surfaces ($\partial M=\emptyset$), this space plays a fundamental role in establishing the regularity results.
We denote by $P_{\mathcal{K}}:L^2(\Sec)\to \mathcal{K}$ the $L^2$-orthogonal projection onto this space.

\begin{theorem}[Navier--Stokes regularity]
\label{thm:regsol}
    \nextline Let $M$ have a $C^{2,1}$ boundary (possibly empty).
    Assume $V_0\in H^2(\Sec)\cap J_{H^1_0}$ and $F,\partial_t F\in L^{\infty}_{\infty}(L^2(\Sec))$ with $\partial_tF\in C([0,\infty);L^2(\Sec))$.
    Further let $P_{\mathcal{K}}(F)\in L^1_{\infty}(L^2(\Sec))$ if $\partial M=\emptyset$ and define $\tau_t\coloneqq \mathrm{min}\{1,t\}$.
    Then there exists a unique solution $(V,p)$ of \eqref{eq:NavStokes} and \eqref{eq:NavStokesP} such that the tuple $(V,\partial_tV,p,\sqrt{\tau} \partial_tV)\in C([0,\infty);H^2(\Sec)\cap J_{H^1_0}\times J_{L^2}\times H^1\cap L^2_0\times J_{H^1_0})$ and
    \begin{align}
        &\stackrel[t\in [0,\infty)]{}{\mathrm{sup}}\left\{\Vert V_t\Vert_{H^2}+\Vert \partial_tV_t\Vert_{L^2}+\Vert p_t\Vert_{H^1}+\sqrt{\tau_t}\Vert \partial_tV_t\Vert_{H^1}\right\}\leq C,
        \label{Est:NS1}
        \tag{$\mathrm{NS}^1_{\mathrm{R}}$}
    \end{align}
    where $C$ is a positive constant depending only on $M$, $\nu$, $\Vert V_0\Vert_{H^2}$ and the force norms $\Vert F\Vert_{L^\infty_\infty(L^2)}$, $\Vert \partial_tF\Vert_{L^\infty_\infty(L^2)}$ and, if $\partial M=\emptyset$, $\Vert P_{\mathcal{K}}(F)\Vert_{L^1_\infty(L^2)}$.
    In particular, for every finite $T>0$, the solution satisfies
    \begin{align*}
        &\partial_tV\in L^2_T(J_{H^1_0}),\qquad &&\sqrt{\tau}\partial_tV\in L^2_T(H^2(\Sec)\cap J_{H^1_0}),\\
        &\sqrt{\tau}\partial_t^2V\in L^2_T(J_{L^2}),\qquad &&\sqrt{\tau}\partial_tp\in L^2_T(H^1\cap L^2_0),
    \end{align*}
    so all quantities entering the estimate below are well-defined.
    In addition, $\Vert V_t-V_0 \Vert_{H^2}\to 0$ as $t\to 0$ and, with a constant $C$ of the same type as in \eqref{Est:NS1}, it holds that
    \begin{align}
        \stackrel[t\in [0,\infty)]{}{\mathrm{sup}}\left\{e^{-t}\int_0^t e^{s}\left(\Vert \partial_s V_s\Vert_{H^1}^2+\tau_s\left(\Vert\partial_s V_s\Vert_{H^2}^2+\Vert \partial_s^2 V_s\Vert_{L^2}^2+\Vert \partial_s p_s\Vert_{H^1}^2\right)\right) \mathrm{d}s\right\}\leq C.
        \label{Est:NS2}
        \tag{$\mathrm{NS}^2_{\mathrm{R}}$}
    \end{align}
\end{theorem}

\begin{remark}
The time weighting factors in the estimates serve two distinct purposes. 
First, the factor $\tau_t\coloneqq \mathrm{min}\{1,t\}$ is introduced to damp potentially diverging contributions near $t=0$, as uniform bounds cannot be guaranteed for general initial data.
Conversely, the exponential factor $e^{-t}$ in \Cref{Est:NS2} guarantees that the expression $\sup_{t\in [0,\infty)}e^{-t}\int_0^t e^{s}(\dots)\,\mathrm{d}s$ stays bounded as long as the integrand is bounded.
If one is only interested in estimates for finite time $T>0$, the exponential factors may be replaced by $1$.
The regularity of $\partial_tV$, $\partial_t^2V$ and $\partial_tp$ stated in the theorem is exactly what is needed for the left-hand side of \eqref{Est:NS2} to be well-defined.\\
Finally, we note that while Heywood and Rannacher \cite[Thm 2.3]{Rannacher1} proved this theorem for bounded domains in $\mathbb{R}^2$ with $C^2$ boundaries, the slightly stronger assumption of $C^{2,1}$ boundary regularity used here is required by the generalized Stokes regularity theorem (\Cref{thm:StokesReg}), below.
\end{remark}

Next we briefly discuss some adaptations that have to be made to the proof of Heywood and Rannacher in order to extend it to surfaces. 
We begin by detailing the necessary steps for surfaces with a boundary, i.e.\ $\partial M\neq \emptyset$.
Afterwards, we address the case $\partial M=\emptyset$.
\begin{enumerate}
    \item \textbf{Poincaré Inequality Adaptation:} The original proof frequently relies on the standard Poincaré inequality, i.e., $\Vert \cdot\Vert_{H^1}\leq C(M)\Vert \nabla\cdot\Vert_{L^2}$, which often leads to simplifications in combination with the Euclidean viscosity term $\nu(\nabla V,\nabla \cdot)$. 
    Since on surfaces the viscosity term is $2\nu(D_V, D_{(\cdot)})$ (see \Cref{thm:IbPDV}), we must replace the classical Poincaré inequality with a Poincaré--Morrey inequality (derived in \Cref{cor:PoincareMorrey} below), which establishes the analogous bound $\Vert \cdot\Vert_{H^1}\leq C(M)\Vert D_{(\cdot)}\Vert_{L^2}.$
    \item \textbf{Stokes Regularity Adaptation:} For the second-order estimates, i.e., the bounds on the $H^2$-norm of $V$, the original proof relies on controlling the $H^2$-norm by the $L^2$-norm of $P(\Delta_{\mathrm{dR}}\cdot)$, where $P:L^2(\Sec)\to J_{L^2}$ is the orthogonal projection into $J_{L^2}$.
    This statement is a consequence of the Euclidean Stokes regularity theorem. 
    The operator $P(\Delta_{\mathrm{dR}}\cdot)$ arises naturally because the Euclidean viscosity term takes the strong form $\delta\nabla=\Delta_{\mathrm{dR}}$. 
    However, since the strong form of the viscosity term on surfaces is $2\delta D$, we require a generalized Stokes regularity theorem (derived in \Cref{thm:StokesReg} below).
    This allows us to control the $H^2$-norm by the $L^2$-norm of $P(\delta D_{\cdot})$.
\end{enumerate}
With these two adapted estimates, the exact norm estimates and bounds from the original proof can be carried over directly, which concludes the derivation for surfaces with nonempty boundaries.

Finally, we address the case of closed surfaces, i.e.\ $\partial M=\emptyset$.
The primary distinction here is that the solution space $J_{H^1_0}$ allows nontrivial elements of the finite-dimensional space of Killing fields $\mathcal{K}$. 
The presence of such fields prevents the direct application of the Poincaré--Morrey inequality and the generalized Stokes regularity theorem (established below in \Cref{cor:PoincareMorrey} and \Cref{thm:StokesReg}). 
Note that this issue does not arise when $\partial M\neq\emptyset$, as the boundary conditions inherent to $J_{H^1_0}$ force all Killing fields to vanish (see \Cref{thm:KillingField} below).

To control the contributions from Killing fields, we decompose the space as $J_{H^1_0}=\mathcal{K}\oplus J_{\mathcal{K}}^\perp$, where $J_{\mathcal{K}}^\perp\coloneqq \{V\in J_{H^1_0}\vert~V\in \mathrm{ker}(P_{\mathcal{K}})\}$ forms the $L^2$-orthogonal complement to $\mathcal{K}$.
This $L^2$-orthogonal decomposition is well-defined since $\mathcal{K}$ is finite-dimensional, see \Cref{thm:KornIneq} below.
Accordingly, we uniquely split the velocity field into $V=V_{\mathcal{K}}+V_{\perp}$.\\
The identities $0=(\nabla_WW,K)=(D_W,\nabla K)=(\nabla_KK,W)$ hold for any $W\in J_{H^1_0}$ and $K\in \mathcal{K}$:
since $D_K=0$, the endomorphism $\nabla K$ is antisymmetric, so that, pointwise, $\langle D_W,\nabla K\rangle=0$ and $\langle\nabla_WK,W\rangle=0$, hence $(D_W,\nabla K)=0$ and $(\nabla_WW,K)=\int_MW(\langle W,K\rangle)\,\mathrm{d}V=0$ by Green's formula, and $\nabla_KK=-\frac{1}{2}\mathrm{grad}(\vert K\vert^2)$ is a gradient field and thus $L^2$-orthogonal to $W$.
By exploiting these identities, we can split the Navier--Stokes equations in velocity formulation \eqref{eq:NavStokes} for a.e.\ $t\in (0,T)$ into:
\begin{align}
    \langle\partial_tV_{\mathcal{K},t},K\rangle=&(F_t,K),~\forall K\in \mathcal{K},\label{eq:VK}\tag{$\mathcal{K}$}
    \\\langle\partial_tV_{\perp,t},W\rangle=&(F_t,W)-(\nabla_{V_{\perp,t}}V_{\perp,t},W)-(\nabla_{V_{\mathcal{K},t}}V_{\perp,t},W)\nonumber
    \\&-(\nabla_{V_{\perp,t}}V_{\mathcal{K},t},W)-2\nu(D_{V_{\perp,t}},\nabla W),~\forall W\in J_{\mathcal{K}}^\perp.\label{eq:VJKperp}\tag{$\mathcal{K}^\perp$}
\end{align}
The initial conditions for $V_{\mathcal{K}}$ and $V_{\perp}$ are determined by the projection of $V_0=V_{\mathcal{K},0}+V_{\perp,0}$. Thus, it remains to show that this system for $V_{\mathcal{K}}$ and $V_{\perp}$ has a unique solution with the required regularity.

First, we observe that the equation for the Killing component \eqref{eq:VK} admits the explicit unique solution $V_{\mathcal{K},t}=V_{\mathcal{K},0}+\int_0^tP_{\mathcal{K}}(F_s)\mathrm{d}s$.
Given the assumptions on the force $F$ and the properties of Killing fields (see \Cref{thm:KillingField} below), it follows that $V_{\mathcal{K}},\partial_tV_{\mathcal{K}},\partial_t^2V_{\mathcal{K}}\in C([0,\infty);\mathcal{K})\cap L^\infty_{\infty}(\mathcal{K})$, satisfying all necessary regularity conditions.\\
Next, we address the equation for the orthogonal component \eqref{eq:VJKperp}. 
Since the space $J_{\mathcal{K}}^\perp$ admits both the Poincaré--Morrey inequality and the Stokes regularity theorem, we can proceed with the same norm estimates as in the case with boundary ($\partial M\neq \emptyset$).
The only potential difference arises from the convective terms involving $V_{\mathcal{K}}$. 
However, using the identity $(\nabla_K K, W)=0$, these terms can always be rewritten in the form $(\nabla_V V, W)$. 
This structural property, combined with the established regularity of $V_{\mathcal{K}}$, allows us to close the required estimates for $V_{\perp}$.

The remainder of this section is dedicated to establishing the Poincaré--Morrey inequality and the Stokes regularity theorem, thereby providing the rigorous justification for the adaptations discussed above. 
We begin with Korn's inequality on $J_{H^1_0}$, which serves as a fundamental prerequisite for the Poincaré--Morrey estimate.

\begin{theorem}[Korn inequality]
\label{thm:KornIneq}
\nextline There exists a positive constant $C(M)$ such that $\Vert \nabla V\Vert_{L^2}\leq C(M)\left(\Vert V\Vert_{L^2}^2+\Vert D_V\Vert_{L^2}^2\right)^{\frac{1}{2}}$ holds for all $V\in J_{H^1_0}$.
In particular, for closed surfaces ($\partial M=\emptyset$), this estimate implies that the space $\mathcal{K}\subset J_{H^1_0}$ is finite-dimensional.
\end{theorem}

\begin{proof}
Let $V\in\mathscr{V}\coloneqq\{X\in\Sec_0\mid\mathrm{div}(X)=0\}$.
We compute
\begin{align*}
    2\Vert D_V\Vert_{L^2}^2
    =2(D_V,\nabla V)=\Vert\nabla V\Vert_{L^2}^2-(\kappa V,V)\geq\Vert\nabla V\Vert_{L^2}^2-C(M)\Vert V\Vert_{L^2}^2.
\end{align*}
We obtain the first equality via \Cref{thm:IbPDV}, while the second follows from Green's formula and the identity $2\delta (D_V)=\delta(\nabla V)-\kappa V-\mathrm{grad}(\mathrm{div}(V))$ (see \cite[Lem 2.1]{NavStokesSphere} or \cite[II.E]{MembraneDyna}).
For the final inequality, we rely on the compactness of $M$, which implies that the Gaussian curvature $\kappa$ is bounded from above by some positive constant $C(M)$.
The result extends to $J_{H^1_0}$ by density (see \Cref{thm:DecompH-1AndDensity}).

For closed surfaces ($\partial M=\emptyset$), combining this Korn inequality with Rellich's lemma proves that the unit ball of the Killing field space $\mathcal{K}$ is compact in $L^2(\Sec)$. Consequently, $\mathcal{K}$ is finite-dimensional.
\end{proof}

\begin{remark}
As one would expect, the space $\mathcal{K}$ remains finite-dimensional even for surfaces with non-empty boundary (see, e.g., \cite[Lemma 6]{priebe1994solvability}).
However, in this case, Killing fields do not contribute to the solution space: as shown in \Cref{thm:KillingField}, item 3.), below, the intersection $J_{H^1_0}\cap \mathcal{K}$ is trivial.
\end{remark}

Using Korn's inequality, we next establish the Poincaré--Morrey inequality on appropriate subspaces of $J_{H^1_0}$.

\begin{theorem}[Poincar\'e--Morrey inequality on subspaces of $J_{H^1_0}$]
\label{thm:PoincareMorrey}
\nextline Let $\mathcal{V}\subset J_{H^1_0}$ be a closed subspace such that $\mathcal{V}\cap \mathcal{K}=\{0\}$.
Then there exists a positive constant $C(M,\mathcal{V})$, depending on $M$ and on the subspace $\mathcal{V}$ (through its separation from the Killing fields $\mathcal{K}$), such that the Poincar\'e--Morrey inequality
\begin{align}
    \Vert V\Vert_{H^1}\leq C(M,\mathcal{V})\Vert D_V\Vert_{L^2}
    \label{eq:PoincareMorrey}
    \tag{PM$_{\leq}$}
\end{align}
holds for all $V\in \mathcal{V}$.
\end{theorem}

\begin{proof}
We prove \Cref{eq:PoincareMorrey} by contradiction. 
Assume the inequality does not hold. 
Then there exists a sequence $Y_n\in \mathcal{V}$ such that $\Vert Y_n\Vert_{H^1}=1$ and $\Vert D_{Y_n}\Vert_{L^2}<\frac{1}{n}$.
By Rellich's lemma and Korn's inequality (\Cref{thm:KornIneq}), we can extract a subsequence that is Cauchy in $\Vert \cdot\Vert_{H^1}$.
Since $\mathcal{V}$ is closed, the limit $Y$ belongs to $\mathcal{V}$ and must satisfy $D_Y=0$ and $\Vert Y\Vert_{H^1}=1$.
The condition $D_Y=0$ implies $Y \in \mathcal{K}$. 
However, since $\mathcal{V}\cap\mathcal{K}=\{0\}$, we must have $Y=0$, which contradicts the normalization $\Vert Y\Vert_{H^1}=1$.
Thus, the estimate holds.
\end{proof}

Finally, we apply the general result to the specific solution spaces of the Navier--Stokes equations.

\begin{corollary}[Poincar\'e--Morrey inequality]
\label{cor:PoincareMorrey}
\nextline The Poincar\'e--Morrey inequality \eqref{eq:PoincareMorrey} holds on $J_{H^1_0}$ if $\partial M\neq \emptyset$, and on $J_{\mathcal{K}}^\perp$ if $\partial M=\emptyset$; in both cases the constant depends only on $M$, as the respective subspace is itself determined by $M$.
\end{corollary}

\begin{proof}
In both cases, the spaces $J_{H^1_0}$ and $J_{\mathcal{K}}^\perp$ are closed subspaces of $H^1(\Sec)$. 
Furthermore, the requirement that the intersection with $\mathcal{K}$ is trivial is satisfied: for $J_{H^1_0}$ (when $\partial M \neq \emptyset$), this follows from \Cref{thm:KillingField}, item 3.), below, while for $J_{\mathcal{K}}^\perp$, it holds by construction. 
Thus, the result follows directly from \Cref{thm:PoincareMorrey}.
\end{proof}

With the Poincaré--Morrey inequality established, we prove the following Stokes regularity theorem.
We observe that for closed surfaces, the variational formulation is only solvable if the force vanishes on the Killing fields, as $D$ and $\mathrm{div}$ vanish on this subspace.

\begin{theorem}[Stokes regularity theorem]
\label{thm:StokesReg}
\nextline Let $F\in H^{-1}(\Sec)$.
In the case of closed surfaces ($\partial M=\emptyset$), we additionally assume the compatibility condition $\langle F,K\rangle=0$ for all $K\in \mathcal{K}$.
Then the problem
\begin{align}
    (D_V,D_W)-(p,\mathrm{div}(W))=\langle F,W\rangle,~\forall W\in H^1_0(\Sec)
    \label{eq:StokesEquMfd}
\end{align}
admits a unique solution $V\in \begin{cases}J_{H^1_0},&\mathrm{if}~\partial M\neq\emptyset\\J_{\mathcal{K}}^\perp,&\mathrm{if}~\partial M=\emptyset\end{cases}$ and $p\in L^2_0$ satisfying
\begin{align}
    \Vert V\Vert_{H^1}+\Vert p\Vert_{L^2}\leq C(M)\Vert F\Vert_{H^{-1}},
    \label{eq:StokesH1}
\end{align}
for a positive constant $C(M)$.
Furthermore, if $M$ has a $C^{2,1}$ boundary and $F\in L^2(\Sec)$, then $V\in H^2(\Sec)$, $p\in H^1(\Sec)$ and
\begin{align*}
    \Vert V\Vert_{H^2}+\Vert p\Vert_{H^1}\leq C(M)\Vert F\Vert_{L^2}.
\end{align*}
\end{theorem}

\begin{proof}
The existence and uniqueness of the solution pair $(V,p)$ follow from Brezzi's splitting theorem \cite[Thm III.4.3]{Braess}, which applies here due to the Poincaré--Morrey inequality (\Cref{cor:PoincareMorrey}) and \Cref{thm:divIsom}.\\
Next, we establish the higher regularity results assuming $F\in L^2(\Sec)$. 
Our approach is inspired by \cite[Ch 7.3.3]{Grisvard} and \cite[Thm I.5.5]{Navier}, relating the Stokes solution to the biharmonic equation.
We employ the Hodge decomposition (\Cref{thm:HodgeDecomp}) to write $V=\mathrm{rot}(\psi)+H$, with $\psi\in H^2\cap H^1_0$ and $H\in H_N$.
Crucially, due to the $C^{2,1}$ boundary, \Cref{thm:H0divHcurlInH1} establishes $H_N \hookrightarrow H^2(\Sec)$. 
This yields the bound $\Vert H\Vert_{H^2}\leq C(M)\Vert H\Vert_{L^2}\leq C(M)\Vert V\Vert_{L^2}$, where the last step uses the $L^2$-orthogonality of the decomposition.
It therefore remains to prove that $\psi\in H^3$.\\
To this end, we first rewrite the viscous term in terms of the vorticity. For $V\in H^1(\Sec)$ with $\mathrm{div}(V)=0$, we have
\begin{align}
    2(D_V,D_W)=(\mathrm{curl}(V),\mathrm{curl}(W))-2(\kappa V,W),\qquad\forall W\in H^1_0(\Sec).
    \label{eq:DVweak}
\end{align}
Indeed, for $W\in\Sec_0$ the proof of \Cref{thm:IbPDV} gives $2(D_V,D_W)=(\Delta_{\mathrm{dR}}W,V)-2(\kappa W,V)$, and $(\Delta_{\mathrm{dR}}W,V)=(\mathrm{curl}(W),\mathrm{curl}(V))$ by Green's formula and $\mathrm{div}(V)=0$. Both sides of \eqref{eq:DVweak} extend to all $W\in H^1_0(\Sec)$ by density.
Combining \eqref{eq:DVweak} with \eqref{eq:StokesEquMfd}, we obtain
\begin{align}
    (\mathrm{curl}(V),\mathrm{curl}(W))=(2F+2\kappa V,W)+(2p,\mathrm{div}(W)),\qquad\forall W\in H^1_0(\Sec).
    \label{eq:StokesCurl}
\end{align}
We now choose test functions of the form $W=\mathrm{rot}(\varphi)$ with $\varphi \in H^2_0$ (the closure of $C^\infty_0$ in $H^2$, with zero mean if $\partial M=\emptyset$).
Since $\mathrm{div}(\mathrm{rot}(\varphi))=0$ and $\mathrm{curl}(V)=\mathrm{curl}(\mathrm{rot}(\psi))=\Delta_{\mathrm{dR}}\psi$ (as $\mathrm{curl}(H)=0$), substituting into \eqref{eq:StokesCurl} yields
\begin{align*}
    (\Delta_{\mathrm{dR}}\psi,\Delta_{\mathrm{dR}}\varphi)=(2F+2\kappa V,\mathrm{rot}(\varphi)),\qquad \forall\varphi \in H^2_0.
\end{align*}
In other words, $\psi$ is a weak solution to the biharmonic equation $\Delta_{\mathrm{dR}}^2\psi=\mathrm{curl}(2F+2\kappa V)$, whose source term lies in $H^{-1}$, the dual of $H^1_0$ (since $2F+2\kappa V\in L^2(\Sec)$, as $F\in L^2(\Sec)$ and $\kappa\in L^\infty$).
Testing against $\varphi\in H^2_0$ alone does not determine $\psi$.
The test space $H^2_0$ only encodes the fourth-order equation in the interior of $M$.
If $\partial M\neq\emptyset$, the equation has to be complemented by two boundary conditions, the Dirichlet data $\psi\vert_{\partial M}$ and the Neumann data $\partial_N\psi\vert_{\partial M}$.
The Dirichlet data vanish, as $\psi\in H^1_0$ by the Hodge decomposition.
If $\partial M=\emptyset$, no boundary data are needed, and $\psi$ is determined by the zero-mean normalization.\\
Standard elliptic regularity for the biharmonic equation (see, e.g., \cite[Thm I.1.11]{Navier} for homogeneous boundary data, where the generalization to surfaces with $C^{2,1}$ boundary is straightforward but tedious), combined with an $H^3$-lifting of the Neumann data when $\partial M\neq\emptyset$, yields the a priori estimate
\begin{align}
    \Vert\psi\Vert_{H^3}\leq C(M)\left(\Vert\mathrm{curl}(2F+2\kappa V)\Vert_{H^{-1}}+\Vert\partial_N\psi\Vert_{H^{3/2}(\partial M)}+\Vert\psi\Vert_{H^2}\right),
    \label{eq:BiharmEst}
\end{align}
where the boundary term is omitted if $\partial M=\emptyset$.
We bound the three terms on the right-hand side of \eqref{eq:BiharmEst} separately.\\
First, since $F\in L^2(\Sec)$, $\kappa\in L^\infty$ and $\mathrm{curl}$ maps $L^2(\Sec)$ boundedly into $H^{-1}$, the source term satisfies
\begin{align*}
    \Vert\mathrm{curl}(2F+2\kappa V)\Vert_{H^{-1}}\leq C(M)\Vert 2F+2\kappa V\Vert_{L^2}\leq C(M)\bigl(\Vert F\Vert_{L^2}+\Vert V\Vert_{L^2}\bigr).
\end{align*}
Second, we identify the Neumann data.
Since $\psi$ vanishes on $\partial M$, its tangential derivative vanishes there, and hence
\begin{align*}
    \mathrm{rot}(\psi)=-\partial_N\psi\,T\qquad\text{on }\partial M.
\end{align*}
The no-slip condition $V=\mathrm{rot}(\psi)+H=0$ on $\partial M$ therefore gives
\begin{align*}
    \partial_N\psi\vert_{\partial M}=\langle H,T\rangle.
\end{align*}
As $H\in H^2(\Sec)$, the trace theorem and the bound for $H$ established above yield
\begin{align*}
    \Vert\partial_N\psi\Vert_{H^{3/2}(\partial M)}=\Vert\langle H,T\rangle\Vert_{H^{3/2}(\partial M)}\leq C(M)\Vert H\Vert_{H^2}\leq C(M)\Vert V\Vert_{L^2}.
\end{align*}
Third, we bound the lower-order term.
Applying $\mathrm{curl}$ to the Hodge decomposition $V=\mathrm{rot}(\psi)+H$ and using $\mathrm{curl}(H)=0$ gives $\Delta_{\mathrm{dR}}\psi=\mathrm{curl}(V)\in L^2$, as $V\in H^1(\Sec)$.
Hence $\psi$ solves a Poisson problem with homogeneous Dirichlet data (with zero-mean normalization if $\partial M=\emptyset$), and standard elliptic regularity, as in the proof of \Cref{thm:HodgeDecomp}, gives $\Vert\psi\Vert_{H^2}\leq C(M)\Vert\mathrm{curl}(V)\Vert_{L^2}\leq C(M)\Vert V\Vert_{H^1}$.
Inserting these three bounds into \eqref{eq:BiharmEst}, we obtain $\Vert\psi\Vert_{H^3}\leq C(M)(\Vert F\Vert_{L^2}+\Vert V\Vert_{H^1})$.
The $H^1$-estimate \eqref{eq:StokesH1} and $\Vert F\Vert_{H^{-1}}\leq C(M)\Vert F\Vert_{L^2}$ finally give $\Vert\psi\Vert_{H^3}\leq C(M)\Vert F\Vert_{L^2}$, which establishes the desired regularity of $V$.\\
Finally, the regularity of the pressure follows from its construction and the established regularity of $F$ and $V$, implying $p\in H^1$ with $\Vert p\Vert_{H^1}\leq C(M)\Vert F\Vert_{L^2}$.
\end{proof}

\begin{remark}
\label{rem:StokesReg}
To establish the regularity result, we reduced the Stokes problem to the biharmonic equation. 
This approach, known from \cite[Ch 7.3.3]{Grisvard} and \cite[Thm I.5.5]{Navier}, allows for a very elementary proof but has the drawback of requiring a $C^{2,1}$ boundary. 
For bounded domains in $\mathbb{R}^2$, it is established that these regularity results hold for $C^2$ boundaries (see \cite{cattabriga1961}).
Thus, we expect that the present result could be extended to $C^2$ surfaces by different techniques.
In planar convex domains, analogous $H^2$ Stokes estimates are known even for Lipschitz boundaries \cite{KelloggOsborn1976,Dauge1989}.
We therefore also expect extensions below $C^2$ in suitably convex surface settings, but identifying the appropriate geometric convexity assumptions and proving the corresponding elliptic estimates is beyond the scope of the present work.
\end{remark}

\subsection{Regularity results for the Euler equation}
\label{sec:regEuler}
For the Euler equations, a substantial body of regularity theory already exists in the literature (see, e.g., \cite[Appendix III]{Temam} for an overview). 
Unlike the Navier--Stokes case, where we had to generalize planar results to surfaces, regularity results for the Euler equations on surfaces with smooth boundaries are well-established.

Therefore, we state the following theorem on finite time intervals.
The specific formulation presented here -- including the functional setting and the treatment of the time-dependent external force -- follows the work of Temam \cite[Thm 2.1]{Temam1976}, who established existence, uniqueness and $H^3$-regularity for bounded planar domains with a $C^5$ boundary.
For general surfaces, we refer to the geometric framework developed by Ebin and Marsden \cite{EbinMarsden1970}.
In their seminal work, they transfer the problem from the classical equations to the task of finding geodesics on the group of volume-preserving diffeomorphisms, a reformulation that allows them to apply methods from global analysis and infinite-dimensional geometry.
Although their main exposition focuses on time-independent forces for simplicity, they explicitly remark that the external force may depend on time.
Neither reference states the following result for surfaces in this exact form; an inspection of the arguments in \cite{Temam1976,EbinMarsden1970} yields the velocity regularity stated below, from which the regularity of $\partial_tV$ and of the pressure follows via the momentum equation.
We state the result in the setting in which it is used in \Cref{thm:EquiFormEu} and \Cref{rem:EquiFormEuReg}.

\begin{theorem}[Euler regularity]
\label{thm:EulerReg}
\nextline Let $M$ have a $C^\infty$ boundary (possibly empty) and let $T\in(0,\infty)$.
Assume $V_0\in H^3(\Sec)\cap J_{L^2}$ and $F\in C([0,T];H^3(\Sec))$.
Then there exists a unique solution $V$ of \eqref{eq:Euler} with
\begin{align*}
    V\in C([0,T];H^3(\Sec)\cap J_{L^2}),\qquad \partial_tV\in C([0,T];H^2(\Sec)\cap J_{L^2});
\end{align*}
in particular, $\Vert V_t-V_0\Vert_{H^3} \to 0$ as $t\to 0$.
Moreover, the mean-zero pressure $p$ associated with $V$ by \Cref{prop:pres} satisfies $p\in C([0,T];H^3\cap L^2_0)$, and $(V,p)$ solves \eqref{eq:NavStokesP} with $\nu=0$.
\end{theorem}

%% file: content/04MixedFormulation.tex
\section{Equivalence of velocity and vorticity formulations}
\label{sec:mixed}
Having established the precise velocity formulation and the necessary regularity of its solutions in the previous section, we now state the precise versions of the vorticity formulations presented in \Cref{sec:ScopeOfWork}, and we provide the rigorous justification for the formal derivations given there. This involves two main tasks: verifying that the vorticity formulations (including their initial conditions) are well-defined, and proving the equivalence theorems for both the Navier--Stokes and Euler equations.

\subsection{Vorticity formulation (Navier--Stokes)}
\label{sec:mixedNavStokes}
In this subsection, we state the precise version of the Navier--Stokes vorticity formulation of \Cref{sec:MainContriVort}, including the pressure equation of \Cref{sec:MainContriPres}, and complete the proof of the equivalence between the velocity and vorticity formulations for the Navier--Stokes equations.
The formal derivations of \Cref{sec:MainContriVort,sec:MainContriPres} are made rigorous in the proof of \Cref{thm:EquiForm}.

\begin{definition}[name={[vorticity formulation NS]Harmonic streamfunction-vorticity formulation (Navier--Stokes)}]
\label{def:HarmStream}
\nextline Let $s\in (2,\infty)$ with dual exponent $r\in(1,2)$, i.e., $\frac{1}{s}+\frac{1}{r}=1$, and let $T\in(0,\infty)$.
Further, let $\psi_0\in H^1_{0}$, $H_0\in H_N$ and $F\in L^2_T(L^r(\Sec))$ be given.
In the notation of \Cref{sec:MainContriVort,sec:MainContriPres} (an index $0$ indicates zero trace on $\partial M$, zero mean if $\partial M=\emptyset$, cf.\ \Cref{sec:ScopeOfWork}), the test spaces are
\begin{align*}
    \mathbb{X}_{\mathrm{kin}}\coloneqq W^{1,r},\qquad \mathbb{X}_{\omega,0}\coloneqq W^{1,s}\cap H^1_0,\qquad \mathbb{X}_H\coloneqq L^s(\Sec)\cap H_N,\qquad \mathbb{X}_p\coloneqq W^{1,s}.
\end{align*}
We seek
\begin{align*}
    (\psi,H)\in L^2_T(W^{1,s}\cap H^1_0) \times L^2_T(L^s(\Sec)\cap H_N),\qquad \omega\in L^2_T(W^{1,r})\cap L^\infty_T(L^r),\qquad p^*\in L^2_T(W^{1,r}\cap L^2_0),
\end{align*}
such that $V\coloneqq \mathrm{rot}(\psi)+H$ satisfies $\partial_tV\in L^2_T(L^r(\Sec))$ and, for a.e.\ $t\in(0,T)$, the kinematic equation
\begin{align}
    (\omega,\chi)&=(V,\mathrm{rot}(\chi)), &&\forall\chi\in\mathbb{X}_{\mathrm{kin}}.
    \label{eq:001}\tag{$\mathrm{NS}^1_{\omega}$}
\end{align}
The tuple $(\psi,H,\omega,p^*)$ solves the \emph{harmonic streamfunction-vorticity formulation} if, with $\nu>0$,
\begin{align}
    \partial_t(\omega,\phi)&=(F-\omega JV-\nu \mathrm{rot}(\omega)+2\nu \kappa V,\mathrm{rot}(\phi)), &&\forall\phi\in\mathbb{X}_{\omega,0},
    \label{eq:002}\tag{$\mathrm{NS}^2_{\omega}$}\\
    \partial_t(H,Z)&=(F-\omega JV-\nu \mathrm{rot}(\omega)+2\nu \kappa V,Z), &&\forall Z\in\mathbb{X}_H,
    \label{eq:002H}\tag{$\mathrm{NS}^3_{\omega}$}\\
    (\mathrm{grad}(p^*),\mathrm{grad}(q))&=(F-\omega JV-\nu \mathrm{rot}(\omega)+2\nu \kappa V,\mathrm{grad}(q)), &&\forall q\in\mathbb{X}_p,
    \label{eq:press01}\tag{$\mathrm{NS}_{p^*}$}
\end{align}
hold as identities in $L^2_T((W^{1,s}\cap H^1_0)')$, $L^2_T((L^s(\Sec)\cap H_N)')$ and $L^2_T((W^{1,s})')$, respectively, and if the initial conditions $(\psi_0,H_0)=(\psi_t,H_t)\vert_{t=0}$ are satisfied.
Here, $\partial_t(\omega,\cdot)$ and $\partial_t(H,\cdot)$ denote the weak Bochner derivatives of $(\omega,\cdot)\in L^2_T((W^{1,s}\cap H^1_{0})')$ and $(H,\cdot)\in L^2_T((L^s(\Sec)\cap H_N)')$, respectively.
\end{definition}

\begin{remark}
\label{rem:HarmStreamWD}
The formulation is deliberately stated using only $L^2_T$- and $L^\infty_T$-type assumptions in time, as this simplifies the subsequent proof of equivalence.
The fact that all terms in \eqref{eq:002}, \eqref{eq:002H} and \eqref{eq:press01} are well-defined, and that the initial conditions are attained, is proved in \Cref{lem:HarmStreamVort} below; in particular, the restriction $2<s<\infty$ is needed to control the convective term in the chosen dual spaces, and the spatial well-definedness of all terms follows from \Cref{lem:s&r}.
In practice, $(\psi_0,H_0)$ are obtained from a given initial velocity $V_0\in J_{L^2}$ by its Hodge decomposition.
\end{remark}

\begin{lemma}[Well-definedness of \refwl{def:HarmStream}]
\label{lem:HarmStreamVort}
\nextline Let $s\in (2,\infty)$ with dual exponent $r\in(1,2)$, i.e., $\frac{1}{s}+\frac{1}{r}=1$ and $T\in(0,\infty)$.
For any tuple $(\psi,H,\omega,p^*)$ satisfying the regularity assumptions of \Cref{def:HarmStream}, the following holds:
\begin{enumerate}[i.)]
    \item $V\coloneqq\mathrm{rot}(\psi)+H\in L^2_T(J_{H^1_0})\cap C([0,T];J_{L^2})$ and $(\psi,H)\in C([0,T];H^1_0\times H_N)$, ensuring that the initial conditions are well-defined. In particular, the data in \Cref{def:HarmStream} are exactly $\psi_0=\psi_t\vert_{t=0}$ and $H_0=H_t\vert_{t=0}$.
    \label{InitCond} 
    \item Equations \eqref{eq:002}, \eqref{eq:002H} and \eqref{eq:press01} are well-defined.
     \label{EqWellDef}
\end{enumerate}
\end{lemma}

\begin{proof}~
\begin{enumerate}[i.)]
    \item We first establish the time continuity of the resulting velocity field $V=\mathrm{rot}(\psi)+H$.
    By construction, $V\in L^2_T(L^s(\Sec)\cap J_{L^2})$. 
    The kinematic equation \eqref{eq:001} implies that $\mathrm{curl}(V)=\omega\in L^2_T(W^{1,r})$. 
    Applying \Cref{lem:s&r}, we find that $\omega\in L^2_T(L^2)$. Consequently, \Cref{cor:LsJL2ToJH10} applies, yielding the enhanced regularity $V\in L^2_T(J_{H^1_0})$.
    To conclude the continuity of $V$, we examine its time derivative.
    Since $\partial_tV\in L^2_T(L^r(\Sec))$, the Sobolev Embedding Theorem (\Cref{eq:SobEmbLp}) implies $\partial_t V \in L^2_T((J_{H^1_0})')$ within the context of the Gelfand triple $J_{H^1_0}\subset J_{L^2}\subset(J_{H^1_0})'$. 
    This regularity implies that $V \in C([0,T];J_{L^2})$.\\
    Since the decomposition $V=\mathrm{rot}(\psi)+H$ is $L^2$-orthogonal, the continuity of $V$ directly implies $(\psi,H) \in C([0,T];H^1_0\times H_N)$, ensuring the initial conditions are well-defined.
    We also note that this continuity of $V$ extends to the vorticity viewed as a functional, implying $(\omega,\cdot)\in C([0,T];(H^1)')$.
    \item Equations \eqref{eq:002}, \eqref{eq:002H} and \eqref{eq:press01} are required to hold in $L^2_T$ with values in the appropriate spatial dual spaces. 
    For most linear terms, this follows immediately from Hölder's inequality. 
    We therefore focus on the three critical terms: the nonlinear convection $\omega JV$, the time derivatives, and the curvature term $\kappa V$.\\
    \textbf{Nonlinear Term:} Since $V \in C([0,T];J_{L^2})\subset L^\infty_T(J_{L^2})$ and $\omega \in L^2_T(W^{1,r})$, we can estimate the term $\omega JV$ against a test function $X\in L^s(\Sec)$ following the reasoning in \Cref{lem:s&r}: $\Vert(\omega JV,X)\Vert_{L^2_T(\mathbb{R})}\leq \Vert\omega\Vert_{L^2_T(W^{1,r})}\Vert V\Vert_{L^\infty_T(J_{L^2})}\Vert X\Vert_{L^s}$.
    Thus, the nonlinear term lies in the required dual spaces.
    Crucially, note that this estimate requires $s\in (2,\infty)$ and would fail for smaller $s$.\\
    \textbf{Time Derivatives:} We need to show that $\partial_t(\omega,\cdot)\in L^2_T((W^{1,s}\cap H^1_0)')$ and $\partial_t (H,\cdot)\in L^2_T((L^s(\Sec)\cap H_N)')$.
    Let $\phi\in C^\infty_0(0,T)$.
    Using \Cref{eq:001} and the fact that $\partial_tV\in L^2_T(L^r(\Sec))$, we compute for any $\chi\in W^{1,s}$
    \begin{align*}
        \int_0^T(\omega_t,\chi)(-\partial_t\phi_t)\mathrm{d}t&=\int_0^T (V_t,\mathrm{rot}(\chi))(-\partial_t\phi_t)\mathrm{d}t=\int_0^T \langle\partial_t V,\mathrm{rot}(\chi)\rangle\phi_t\mathrm{d}t.
    \end{align*}
    Similarly, utilizing the orthogonality $\mathrm{rot}(H^1_0)\perp^{L^2}H_N$, we find for any $Z\in L^s(\Sec)\cap H_N$
    \begin{align*}
        \int_0^T(H_t,Z)(-\partial_t\phi_t)\mathrm{d}t&=\int_0^T(V_t,Z)(-\partial_t\phi_t)\mathrm{d}t=\int_0^T\langle\partial_t V,Z \rangle\phi_t\mathrm{d}t.
    \end{align*}
    Since the spaces $W^{1,s}$ and $L^s(\Sec)\cap H_N$ are reflexive, these identities imply the existence of weak Bochner derivatives, denoted by $\partial_t(\omega,\cdot)\in L^2_T((W^{1,s})')$ and $\partial_t(H,\cdot)\in L^2_T((L^s(\Sec)\cap H_N)')$.  
    The embedding $(W^{1,s})' \hookrightarrow (W^{1,s}\cap H^1_0)'$ ensures the vorticity derivative has the required regularity.\\
    \textbf{Curvature Term:} Finally, since $\kappa\in L^\infty$ ($M$ is compact), the Sobolev Embedding Theorem (\Cref{eq:SobEmbLp}) allows us to bound the curvature term
    \begin{align*}
        \Vert(\kappa V,X)\Vert_{L^2_T(\mathbb{R})}\leq \Vert \kappa\Vert_{L^\infty}\Vert V\Vert_{L^2_T(L^2(\Sec))}\Vert X\Vert_{L^s}, \quad \forall X\in L^s(\Sec).
    \end{align*}
\end{enumerate}    
\end{proof}

With the well-definedness of the formulation established, we can now state the main equivalence theorem in its precise form (cf.\ the formal summary in \Cref{sec:MainContriVort}).

\begin{theorem}[Equivalence of Navier--Stokes formulations]
\label{thm:EquiForm}
\nextline Let $s\in (2,\infty)$ with dual exponent $r\in(1,2)$, i.e., $\frac{1}{s}+\frac{1}{r}=1$, and let $T\in(0,\infty)$.
Assume that $M$ has a $C^{1,1}$ boundary (possibly empty), so that the $H^1$ Hodge decomposition (\Cref{thm:HodgeDecomp} with $k=1$) applies.
\begin{enumerate}[1.)]
    \item Let $V$ solve the Navier--Stokes equations in velocity formulation (\Cref{def:VeloForm}) with $F,\partial_t V\in L^2_T(L^r(\Sec))$ and $\mathrm{curl}(V)\in L^2_T(W^{1,r})\cap L^\infty_T(L^r)$, and let $p$ be the associated pressure of \Cref{prop:pres}, so that $(V,p)$ satisfies \eqref{eq:NavStokesP}.
    Define $(\omega,p^*)\coloneqq\left(\mathrm{curl}(V),p+\frac{\vert V\vert^2}{2}-\frac{\Vert V\Vert_{L^2}^2}{2\mathrm{Vol}(M)}\right)$, and choose $\psi$ and $H$ uniquely such that $V=\mathrm{rot}(\psi)+H$.
    If $V_0$ denotes the initial velocity, set $(\psi_0,H_0)$ to be its Hodge components, $V_0=\mathrm{rot}(\psi_0)+H_0$.
    Then the tuple $(\psi,H,\omega,p^*)$ solves the Navier--Stokes equations in vorticity formulation (\Cref{def:HarmStream}), including the Poisson equation \eqref{eq:press01} for the total pressure.
    \label{Velo->Vort}
    \item Let $(\psi,H,\omega,p^*)$ solve the Navier--Stokes equations in vorticity formulation. Then $V\coloneqq\mathrm{rot}(\psi)+H$ has the regularity assumed in \ref{Velo->Vort}.) and solves the Navier--Stokes equations in velocity formulation with initial velocity $V_0=\mathrm{rot}(\psi_0)+H_0$, and $p\coloneqq p^*-\frac{\vert V\vert^2}{2}+\frac{\Vert V\Vert_{L^2}^2}{2\mathrm{Vol}(M)}$ is the associated pressure, i.e., $(V,p)$ satisfies \eqref{eq:NavStokesP}.
\end{enumerate}
\end{theorem}

\begin{remark}
\label{rem:EquiFormReg}
The assumptions $F,\partial_tV\in L^2_T(L^r(\Sec))$ ensure that the forcing and time-derivative terms in the momentum balance have the spatial integrability required for the $W^{1,r}$ pressure estimate.
The condition $\mathrm{curl}(V)\in L^2_T(W^{1,r})\cap L^\infty_T(L^r)$ supplies the vorticity regularity required in \Cref{def:HarmStream} and, together with the velocity regularity, controls the nonlinear product $\omega JV$ in the corresponding dual spaces.
These requirements are naturally satisfied by the solutions whose regularity is established in \Cref{thm:regsol}, provided the external force has the corresponding regularity.
\end{remark}

We are now in a position to conclude the proof of the equivalence between the velocity and vorticity formulations by deriving the required regularity statements.

\begin{proof}[Regularity proof of \refwl{thm:EquiForm}]
In this proof, we provide the rigorous justification for the formal derivations in \Cref{sec:MainContriVort,sec:MainContriPres} by establishing the missing regularity statements.
The proof uses three observations, which relate to choosing the test vector field in \eqref{eq:momPstar} as $W=\mathrm{rot}(\phi)+Z$ or $W=\mathrm{grad}(q)$, respectively.
\begin{enumerate}[(a)]
    \item \label{obs:ext}
    By density of $H^1_0(\Sec)$ in $L^s(\Sec)$, \eqref{eq:momPstar} extends to all $W\in L^s(\Sec)$ if all other terms, i.e., the fields paired with $W$, lie in $L^r(\Sec)$.
    In particular, we can choose $W=\mathrm{grad}(q)$ with $q\in\mathbb{X}_p=W^{1,s}$, which need not vanish on $\partial M$, or $W=\mathrm{rot}(\phi)+Z$ with $\phi\in\mathbb{X}_{\omega,0}$ and $Z\in\mathbb{X}_H$.
    \item \label{obs:grad}
    For $g\in W^{1,r}$ and $Y\in L^s(\Sec)\cap J_{L^2}$, we have $(\mathrm{grad}(g),Y)=0$.
    Indeed, by density of smooth functions in $W^{1,r}$ there are $g_n\in H^1$ with $g_n\to g$ in $W^{1,r}$, Green's formula gives $(\mathrm{grad}(g_n),Y)=-(g_n,\mathrm{div}(Y))+\langle\langle Y,N\rangle,g_n\rangle_{\partial M}=0$, and $(\mathrm{grad}(g_n),Y)\to(\mathrm{grad}(g),Y)$ since $Y\in L^s(\Sec)$.
    In particular, if $p^*\in W^{1,r}$, the pressure term $(\mathrm{grad}(p^*),W)$ in \eqref{eq:momPstar} vanishes for $W=\mathrm{rot}(\phi)+Z$ with $\phi\in\mathbb{X}_{\omega,0}$ and $Z\in\mathbb{X}_H$, as $\mathrm{rot}(\phi)+Z\in L^s(\Sec)\cap J_{L^2}$.
    \item \label{obs:dt}
    For $q\in\mathbb{X}_p=W^{1,s}\subset H^1$ and $V\in C([0,T];J_{L^2})$ with $\partial_tV\in L^2_T(L^r(\Sec))$, we have $(V_t,\mathrm{grad}(q))=0$ for all $t$ by \eqref{eq:HodgeAss}, and hence $(\partial_tV,\mathrm{grad}(q))=\partial_t(V,\mathrm{grad}(q))=0$, since $\partial_tV$ is the weak time derivative of $V$ and $\mathrm{grad}(q)\in L^s(\Sec)$.
\end{enumerate}
\begin{enumerate}[1.)]
    \item \label{1.)NS}
    We proceed in two steps.
    First, in \ref{RegP}.), we establish the pressure regularity $p,p^*\in L^2_T(W^{1,r}\cap L^2_0)$ and the momentum equation \eqref{eq:momPstar}.
    Then, in \ref{PsiHwp}.), we define $(\psi,H,\omega)$ and verify that $(\psi,H,\omega,p^*)$ possesses the regularity required by \refwl{def:HarmStream} and satisfies its equations.
    \begin{enumerate}[a.)]
        \item \label{RegP} Let $V\in L^2_T(J_{H^1_0})\cap L^\infty_T(J_{L^2})$ be a solution to \eqref{eq:NavStokes} with the regularity assumed in the theorem, and let $p\in L^2_T(L^2_0)$ be the associated pressure of \refwl{prop:pres}, so that $(V,p)$ satisfies \eqref{eq:NavStokesP} for a.e.\ $t\in(0,T)$.
        We show that $\mathrm{grad}(p)\in L^2_T(L^r(\Sec))$.
        For a.e.\ $t$, we have $V_t\in J_{H^1_0}$ and $\omega_t\coloneqq\mathrm{curl}(V_t)\in W^{1,r}$, so that \refwl{lem:Nonlinterm} and \refwl{thm:IbPDV} provide the identities \eqref{eq:keyidentities} in weak form.
        Inserting them into \eqref{eq:NavStokesP} gives
        \begin{align*}
            \mathrm{grad}(p)=F-\partial_tV-\omega JV-\frac{1}{2}\mathrm{grad}(\vert V\vert^2)-\nu\,\mathrm{rot}(\omega)+2\nu\kappa V
        \end{align*}
        for a.e.\ $t$, as an identity in $H^{-1}(\Sec)$.
        All terms on the right-hand side belong to $L^2_T(L^r(\Sec))$.
        For $F$ and $\partial_tV$ this holds by assumption.
        For the convective term, $\omega\in L^2_T(W^{1,r})$ and $V\in L^\infty_T(J_{L^2})$ give $\omega JV\in L^2_T(L^r(\Sec))$ by \refwl{lem:s&r} and Hölder's inequality, exactly as in the proof of \Cref{lem:HarmStreamVort}.
        \Cref{thm:reg(V.d)V} yields $\vert V\vert^2\in L^2_T(W^{1,r})$ and hence $\mathrm{grad}(\vert V\vert^2)\in L^2_T(L^r(\Sec))$.
        Finally, $\mathrm{rot}(\omega)\in L^2_T(L^r(\Sec))$ since $\omega\in L^2_T(W^{1,r})$, and $\kappa V\in L^2_T(L^2(\Sec))\subset L^2_T(L^r(\Sec))$ since $\kappa\in L^\infty$.
        Consequently, $p\in L^2_T(W^{1,r}\cap L^2_0)$.
        Since also $\vert V\vert^2\in L^2_T(W^{1,r})$ and $V\in L^\infty_T(J_{L^2})$, the total pressure $p^*=p+\frac{\vert V\vert^2}{2}-\frac{\Vert V\Vert_{L^2}^2}{2\mathrm{Vol}(M)}$ belongs to $L^2_T(W^{1,r}\cap L^2_0)$ as well, and $(V,p^*)$ satisfies \eqref{eq:momPstar}, where the fields $\partial_tV$ and $F-\mathrm{grad}(p^*)-\omega JV-\nu\mathrm{rot}(\omega)+2\nu\kappa V$ paired with $W$ belong to $L^2_T(L^r(\Sec))$.
        \item \label{PsiHwp}
        \textbf{Streamfunction and Harmonic Field:}
        By \eqref{eq:HodgeAss}, $V_t=\mathrm{rot}(\psi_t)+H_t$ uniquely with $\psi_t\in H^1_0$ and $H_t\in H_N$. For a.e.\ $t$, $V_t\in H^1(\Sec)$, so that the $H^1$ Hodge decomposition (\Cref{thm:HodgeDecomp} with $k=1$) yields $\psi_t\in H^2\cap H^1_0$ and $H_t\in H^1(\Sec)\cap H_N$.
        Consequently, there exist unique potentials $\psi\in L^2_T(H^2\cap H^1_0)\cap L^\infty_T(H^1_0)$ and $H\in L^2_T(H^1(\Sec)\cap H_N)\cap L^\infty_T(H_N)$ such that $V=\mathrm{rot}(\psi)+H$.
        The Sobolev Embedding Theorem (\Cref{eq:SobEmbLp}) then ensures that $(\psi,H)\in L^2_T(W^{1,s}\cap H^1_0)\times L^2_T(L^s(\Sec)\cap H_N)$ for $s\in(2,\infty)$.
        Furthermore, since $V\in C([0,T];J_{L^2})$ (see \Cref{lem:VeloForm}), we conclude -- analogously to \Cref{lem:HarmStreamVort} \ref{InitCond}.) -- that $(\psi,H) \in C([0,T];H^1_0\times H_N)$. Hence $(\psi_t,H_t)\vert_{t=0}$ are the Hodge components $(\psi_0,H_0)$ of $V_0$.
        Finally, $\partial_t(\mathrm{rot}(\psi)+H)=\partial_tV\in L^2_T(L^r(\Sec))$ by assumption.\\
        \textbf{Vorticity:} By assumption, $\omega=\mathrm{curl}(V)\in L^2_T(W^{1,r})\cap L^\infty_T(L^r)$. Using the fact that $V\in L^2_T(J_{H^1_0})$ and the embedding $W^{1,r}\subset L^2$ (\refwl{lem:s&r}), we confirm that \eqref{eq:001} holds in $L^2_T((W^{1,r})')$ (\Cref{cor:LsJL2ToJH10}).
        Regarding the time derivative, we proceed exactly as in \Cref{lem:HarmStreamVort} ii.). The regularity of $\partial_t V$ implies that $\partial_t(\omega,\phi)=(\partial_tV,\mathrm{rot}(\phi))$ for $\phi\in\mathbb{X}_{\omega,0}$ and $\partial_t(H,Z)=(\partial_tV,Z)$ for $Z\in\mathbb{X}_H$, so that the functionals satisfy $\partial_t(\omega,\cdot)\in L^2_T((W^{1,s}\cap H^1_0)')$ and $\partial_t(H,\cdot)\in L^2_T((L^s(\Sec)\cap H_N)')$.
        By Step~\ref{RegP}.), the fields paired with $W$ in \eqref{eq:momPstar} belong to $L^2_T(L^r(\Sec))$, and $p^*\in L^2_T(W^{1,r})$.
        By observations (\ref{obs:ext}) and (\ref{obs:grad}), \eqref{eq:momPstar} may therefore be tested with $\mathrm{rot}(\phi)+Z$, $\phi\in\mathbb{X}_{\omega,0}$, $Z\in\mathbb{X}_H$, and the pressure term drops out.
        This rigorously justifies the derivation of \eqref{eq:002} and \eqref{eq:002H} in \Cref{sec:MainContriVort}.\\
        \textbf{Total Pressure:} By Step~\ref{RegP}.), $p^*\in L^2_T(W^{1,r}\cap L^2_0)$, as required in \refwl{def:HarmStream}.
        By observation (\ref{obs:ext}), \eqref{eq:momPstar} may be tested with $\mathrm{grad}(q)$, $q\in\mathbb{X}_p=W^{1,s}$, and by observation (\ref{obs:dt}) the time-derivative term vanishes.
        This rigorously justifies the derivation of \eqref{eq:press01} in \Cref{sec:MainContriPres}.
    \end{enumerate}
    \item \label{2.)NS}
    Let $(\psi,H,\omega,p^*)$ be a solution to \eqref{eq:001} to \eqref{eq:press01}, and define $V\coloneqq \mathrm{rot}(\psi)+H$.
    According to \Cref{lem:HarmStreamVort}, we have $V\in L^2_T(J_{H^1_0})\cap C([0,T];J_{L^2})$.
    This ensures that $V \in L^\infty_T(J_{L^2})$ and that $V_t\vert_{t=0}=\mathrm{rot}(\psi_0)+H_0=V_0\in J_{L^2}$.
    Furthermore, the regularity $\partial_t V\in L^2_T(L^r(\Sec))$ is satisfied by definition (see \Cref{def:HarmStream}).
    Additionally, \Cref{eq:001} implies $\mathrm{curl}(V)=\omega$ (\Cref{cor:LsJL2ToJH10}).
    Since the solution space for $\omega$ is $L^2_T(W^{1,r})\cap L^\infty_T(L^r)$, we conclude that $\mathrm{curl}(V)$ shares this regularity.\\
    It remains to recover the momentum equation.
    By the $H^1$ Hodge decomposition (\Cref{thm:HodgeDecomp} with $k=1$), every $W\in H^1_0(\Sec)$ splits as $W=\mathrm{grad}(q)+\mathrm{rot}(\phi)+Z$ with $q\in H^2\cap L^2_0\subset\mathbb{X}_p$, $\phi\in H^2\cap H^1_0\subset\mathbb{X}_{\omega,0}$ and $Z\in H_N\subset\mathbb{X}_H$ (by the Sobolev embedding \eqref{eq:SobEmbLp} and $H_N\subset H^1(\Sec)$, see \Cref{thm:H0divHcurlInH1}).
    Testing \eqref{eq:press01}, \eqref{eq:002} and \eqref{eq:002H} with these three components and adding the results gives \eqref{eq:momPstar} for all $W\in H^1_0(\Sec)$.
    Indeed, the pressure term pairs to zero with $\mathrm{rot}(\phi)+Z$ by observation (\ref{obs:grad}), $(\partial_tV,\mathrm{grad}(q))=0$ by observation (\ref{obs:dt}), and $\partial_t(\omega,\phi)=(\partial_tV,\mathrm{rot}(\phi))$ and $\partial_t(H,Z)=(\partial_tV,Z)$ by \eqref{eq:001} and $\mathrm{rot}(H^1_0)\perp^{L^2}H_N$.
    As $V\in L^2_T(J_{H^1_0})$ and $\mathrm{curl}(V)=\omega$, \Cref{lem:Nonlinterm,thm:IbPDV} turn \eqref{eq:momPstar} into \eqref{eq:NavStokesP} with $p=p^*-\frac{\vert V\vert^2}{2}+\frac{\Vert V\Vert_{L^2}^2}{2\mathrm{Vol}(M)}\in L^2_T(W^{1,r}\cap L^2_0)$ (\Cref{thm:reg(V.d)V}), where $(\mathrm{grad}(p),W)=-(p,\mathrm{div}(W))$ for $W\in H^1_0(\Sec)$ by Green's formula. Restricting \eqref{eq:NavStokesP} to $W\in J_{H^1_0}$ gives \eqref{eq:NavStokes}.
\end{enumerate}
\end{proof}

\subsection{Vorticity formulation (Euler)}
\label{sec:mixedEuler}
In this subsection, we state the precise version of the Euler vorticity formulation of \Cref{sec:MainContriVortEu} and complete the proof of the equivalence between the velocity and vorticity formulations of the Euler equations.
Many steps in this analysis are analogous to the Navier--Stokes case discussed previously.

\begin{definition}[name={[vorticity formulation (E)]Harmonic streamfunction-vorticity formulation (Euler)}]
\label{def:HarmStreamEu}
\nextline Let $s\in (2,\infty)$ with dual exponent $r\in(1,2)$, i.e., $\frac{1}{s}+\frac{1}{r}=1$, and let $T\in(0,\infty)$.
Further assume that $M$ satisfies the embedding
\begin{align}
    H_0(\mathrm{div})\cap H(\mathrm{curl})\hookrightarrow H^1(\Sec),
    \label{ass:EulerForH1}
    \tag{$\mathrm{As}_{\mathrm{E}}$}
\end{align}
and let $H_0\in H_N$, $\omega_0\in L^2\cap L^r$ and $F\in L^2_T(L^r(\Sec))$ with $\mathrm{curl}(F)\in L^2_T(L^r)$ be given.
Choose some $\overline{r}\in(1,r)$ and let $\overline{s}$ be its dual exponent, i.e., $\frac{1}{\overline{r}}+\frac{1}{\overline{s}}=1$.
In the notation of \Cref{sec:MainContriVortEu,sec:MainContriPres} (an index $0$ indicates zero trace on $\partial M$, zero mean if $\partial M=\emptyset$, cf.\ \Cref{sec:ScopeOfWork}), the test spaces are
\begin{align*}
    \mathbb{X}_{\mathrm{kin},0}\coloneqq W^{1,r}_0,\qquad \mathbb{X}_\omega\coloneqq L^{\overline{s}},\qquad \mathbb{X}_H\coloneqq L^s(\Sec)\cap H_N,\qquad \mathbb{X}_p\coloneqq W^{1,s}.
\end{align*}
We seek
\begin{align*}
    (\psi,H)\in L^2_T(W^{1,s}\cap H^1_0)\times L^2_T(L^s(\Sec)\cap H_N),\qquad \omega\in L^2_T(W^{1,r})\cap L^\infty_T(L^2),\qquad p^*\in L^2_T(W^{1,r}\cap L^2_0),
\end{align*}
such that $V\coloneqq \mathrm{rot}(\psi)+H$ satisfies $\partial_tV\in L^2_T(L^r(\Sec))$, $\partial_t\omega \in L^2_T(L^r)$, and, for a.e.\ $t\in(0,T)$, the kinematic equation
\begin{align}
    (\omega,\chi)&=(\mathrm{rot}(\psi),\mathrm{rot}(\chi)), &&\forall\chi\in\mathbb{X}_{\mathrm{kin},0},
    \label{eq:001Eu}\tag{$\mathrm{E}^1_{\omega}$}
\end{align}
i.e., $(\omega,\cdot)=(\mathrm{rot}(\psi),\mathrm{rot}(\cdot))$ in $L^2_T((W^{1,r}_0)')$, the dual of the test space $\mathbb{X}_{\mathrm{kin},0}=W^{1,r}_0$ (on closed surfaces, $W^{1,r}_0$ consists of the functions in $W^{1,r}$ with zero mean, see \Cref{sec:not}).
The tuple $(\psi,H,\omega,p^*)$ solves the \emph{harmonic streamfunction-vorticity formulation} if
\begin{align}
    (\partial_t\omega,\phi)&=(\mathrm{curl}(F)-V(\omega),\phi), &&\forall\phi\in\mathbb{X}_\omega,
    \label{eq:002Eu}\tag{$\mathrm{E}^2_{\omega}$}\\
    \partial_t (H,Z)&=(F-\omega JV,Z), &&\forall Z\in\mathbb{X}_H,
    \label{eq:003Eu}\tag{$\mathrm{E}^3_{\omega}$}\\
    (\mathrm{grad}(p^*),\mathrm{grad}(q))&=(F-\omega JV,\mathrm{grad}(q)), &&\forall q\in\mathbb{X}_p,
    \label{eq:press01Eu}\tag{$\mathrm{E}_{p^*}$}
\end{align}
hold as identities in $L^2_T(L^{\overline{r}})$, $L^2_T((L^s(\Sec)\cap H_N)')$ and $L^2_T((W^{1,s})')$, respectively, and if the initial conditions $(H_0,\omega_0)=(H_t,\omega_t)\vert_{t=0}$ are satisfied.
Here, $\partial_t(H,\cdot)$ denotes the weak Bochner derivative of $(H,\cdot)\in L^2_T((L^s(\Sec)\cap H_N)')$, whereas $\partial_t\omega$ is a function by assumption, so that \eqref{eq:002Eu} is an identity in $L^2_T(L^{\overline{r}})$ tested with $\mathbb{X}_\omega=L^{\overline{s}}$, the dual space of $L^{\overline{r}}$.
The exponent $\overline{r}<r$ is needed because the transport term $V(\omega)$ is only controlled in $L^{\overline{r}}$ for $\overline{r}<r$ (see \Cref{lem:HarmStreamVortEu}).
For closed surfaces ($\partial M=\emptyset$), we additionally require that $\omega_0$ has zero mean and that $\omega\in L^2_T(L^1_0)$, where $L^1_0$ denotes the integrable functions with zero mean.
\end{definition}

\begin{remark}
\label{rem:HarmStreamEuWD}
The well-definedness of the equations and the attainment of the initial conditions are proved in \Cref{lem:HarmStreamVortEu} below.
The additional assumption \eqref{ass:EulerForH1} (see \Cref{thm:H0divHcurlInH1}) is needed because the no-penetration boundary condition is weaker than the no-slip condition.
Unlike in the Navier--Stokes case, the initial data are prescribed as $(H_0,\omega_0)$; for a given initial velocity $V_0$, these are obtained from the Hodge decomposition and $\omega_0=\mathrm{curl}(V_0)$, while the initial streamfunction is reconstructed from the kinematic equation.
Indeed, by \eqref{eq:001Eu} and the time continuity of $(\psi,\omega)$ established in \Cref{lem:HarmStreamVortEu} i.), the initial streamfunction $\psi_0\coloneqq\psi_t\vert_{t=0}$ is uniquely determined by $\omega_0$ as the solution of $(\mathrm{rot}(\psi_0),\mathrm{rot}(\varphi))=(\omega_0,\varphi)$ for all $\varphi\in H^1_0$, which is why no initial condition for $\psi$ is prescribed.
Prescribing $\omega_0\in L^2$ (rather than merely $\omega_0\in L^r$) is no restriction but a necessary compatibility condition: any solution satisfies $\omega\in L^\infty_T(L^2)\cap C([0,T];L^r)$, so that $\omega_0=\omega_t\vert_{t=0}$ necessarily belongs to $L^2$; likewise, on closed surfaces the zero-mean condition on $\omega_0$ is inherited from $\omega\in L^2_T(L^1_0)$ by continuity.
\end{remark}

\begin{lemma}[Well-definedness of \refwl{def:HarmStreamEu}]
\label{lem:HarmStreamVortEu}
\nextline Let $s\in (2,\infty)$ with dual exponent $r\in(1,2)$, i.e., $\frac{1}{s}+\frac{1}{r}=1$ and $T\in(0,\infty)$.
For any tuple $(\psi,H,\omega,p^*)$ satisfying the regularity assumptions of \Cref{def:HarmStreamEu}, the following holds:
\begin{enumerate}[i.)]
    \item $V\coloneqq\mathrm{rot}(\psi)+H\in L^2_T(H^1(\Sec)\cap J_{L^2})\cap C([0,T];J_{L^2})$ and $(\psi,H,\omega)\in C([0,T];H^1_0\times H_N\times L^r)$, ensuring that the initial conditions are well-defined.
    \item Equations \eqref{eq:002Eu}, \eqref{eq:003Eu} and \eqref{eq:press01Eu} are well-defined.
\end{enumerate}
\end{lemma}

\begin{proof}
~
\begin{enumerate}[i.)]
    \item Analogously to \Cref{lem:HarmStreamVort}, we derive from \Cref{eq:001Eu} that $V \in L^2_T(L^s(\Sec) \cap J_{L^2})$ and $\mathrm{curl}(V) = \omega \in L^2_T(W^{1,r}) \subset L^2_T(L^2)$ (if $\partial M=\emptyset$, the test functions in \eqref{eq:001Eu} have zero mean, so that $\mathrm{curl}(V)-\omega$ is constant, and the constant vanishes by $\omega\in L^2_T(L^1_0)$).
    Invoking assumption \eqref{ass:EulerForH1}, it follows that $V \in L^2_T(H^1(\Sec) \cap J_{L^2})$.
    By the Sobolev Embedding Theorem (\Cref{eq:SobEmbLp}), we have $\partial_t V \in L^2_T((H^1(\Sec) \cap J_{L^2})')$ within the context of the Gelfand triple $H^1(\Sec) \cap J_{L^2} \subset J_{L^2} \subset (H^1(\Sec) \cap J_{L^2})'$.
    This regularity implies that $V \in C([0,T]; J_{L^2})$.
    
    As in \Cref{lem:HarmStreamVort}, the $L^2$-orthogonality of the Hodge decomposition $V=\mathrm{rot}(\psi)+H$ implies that $(\psi, H) \in C([0,T]; H^1_0) \times C([0,T]; H_N)$.
    Furthermore, the condition $\partial_t \omega \in L^2_T(L^r)$ implies $\omega \in C([0,T]; L^r)$, ensuring that the initial conditions are well-defined.
    \item For \Cref{eq:003Eu} and \eqref{eq:press01Eu}, the arguments from \Cref{lem:HarmStreamVort} apply directly, ensuring these equations are well-defined.
    We therefore focus our investigation on the vorticity transport equation \eqref{eq:002Eu}.
    The regularity assumptions on $\partial_t \omega$ and $\mathrm{curl}(F)$ guarantee that both terms reside in $L^2_T(L^{\overline{r}})$ for any $\overline{r} \in (1,r)$.
    
    Finally, we address the advection term $V(\omega)$.
    Since $V \in L^\infty_T(J_{L^2})$ and $\mathrm{curl}(V) \in L^\infty_T(L^2)$ (as implied by the regularity of $\omega$), assumption \eqref{ass:EulerForH1} allows us to conclude that $V \in L^\infty_T(H^1(\Sec) \cap J_{L^2})$.
    Combining this spatial regularity with $\omega \in L^2_T(W^{1,r})$ and invoking the Sobolev Embedding Theorem (\Cref{eq:SobEmbLp}), we obtain $V(\omega) \in L^2_T(L^{\overline{r}})$ for any $\overline{r} \in (1,r)$.
\end{enumerate}
\end{proof}

With the well-definedness of the formulation established, we can now state the equivalence theorem for the Euler equations in its precise form (cf.\ the formal summary in \Cref{sec:MainContriVortEu}).

\begin{theorem}[Equivalence of Euler formulations]
\label{thm:EquiFormEu}
\nextline Let $s\in (2,\infty)$ with dual exponent $r\in(1,2)$, i.e., $\frac{1}{s}+\frac{1}{r}=1$, and let $T\in(0,\infty)$.
Assume that $M$ has a $C^{1,1}$ boundary (possibly empty). Then \eqref{ass:EulerForH1} holds by \Cref{thm:H0divHcurlInH1}, and the $H^1$ Hodge decomposition (\Cref{thm:HodgeDecomp} with $k=1$) is available.
\begin{enumerate}[1.)]
    \item Let $V$ solve the Euler equations in velocity formulation (\Cref{def:VeloForm}) with $F,\partial_t V\in L^2_T(L^r(\Sec))$, $\mathrm{curl}(F),\partial_t\mathrm{curl}(V)\in L^2_T(L^r)$ and $\mathrm{curl}(V)\in L^2_T(W^{1,r})$, and let $p$ be the associated pressure of \Cref{prop:pres}, so that $(V,p)$ satisfies \eqref{eq:NavStokesP} with $\nu=0$.
    Define $(\omega,p^*)\!\coloneqq\!\left(\!\mathrm{curl}(V),p\!+\!\frac{\vert V\vert^2}{2}\!-\!\frac{\Vert V\Vert_{L^2}^2}{2\mathrm{Vol}(M)}\!\right)$, and choose $\psi$ and $H$ uniquely such that $V=\mathrm{rot}(\psi)+H$.
    For the initial data, set $H_0$ to the harmonic component of $V_0$ and $\omega_0=\mathrm{curl}(V_0)$. The initial streamfunction is then determined by the kinematic equation \eqref{eq:001Eu}.
    Then the tuple $(\psi,H,\omega,p^*)$ solves the Euler equations in vorticity formulation (\Cref{def:HarmStreamEu}), including the Poisson equation \eqref{eq:press01Eu} for the total pressure.
    \item Let $(\psi,H,\omega,p^*)$ solve the Euler equations in vorticity formulation. Then $V\coloneqq\mathrm{rot}(\psi)+H$ has the regularity assumed in 1.) and solves the Euler equations in velocity formulation with initial velocity $V_0=\mathrm{rot}(\psi_0)+H_0$, where $\psi_0$ is determined by $\omega_0$ through the kinematic equation \eqref{eq:001Eu}, and $p\coloneqq p^*-\frac{\vert V\vert^2}{2}+\frac{\Vert V\Vert_{L^2}^2}{2\mathrm{Vol}(M)}$ is the associated pressure, i.e., $(V,p)$ satisfies \eqref{eq:NavStokesP} with $\nu=0$.
\end{enumerate}
\end{theorem}

\begin{remark}
\label{rem:EquiFormEuReg}
The assumptions $\mathrm{curl}(F),\partial_t\mathrm{curl}(V)\in L^2_T(L^r)$ ensure that the source term and the time derivative in the vorticity transport equation have the required $L^2_T(L^{\overline{r}})$ regularity for every $\overline{r}\in(1,r)$.
Together with $\mathrm{curl}(V)\in L^2_T(W^{1,r})$ and the $H^1$-regularity of the velocity implied by \eqref{ass:EulerForH1}, they also provide the integrability needed for the transport term $V(\omega)$.
These requirements are naturally satisfied, on each finite time interval $[0,T]$, by the regular solutions described in \Cref{thm:EulerReg}; in particular, the force class $F\in C([0,T];H^3(\Sec))$ assumed there provides the regularity of $F$ and $\mathrm{curl}(F)$ required above.
\end{remark}

We are now in a position to conclude the proof of the equivalence between the velocity and vorticity formulations by deriving the missing regularity statements.

\begin{proof}[Regularity proof of \refwl{thm:EquiFormEu}]
The structure of the proof mirrors the regularity proof of the equivalence theorem (\Cref{thm:EquiForm}) for the Navier--Stokes equations.
To streamline the presentation, we refer to the corresponding steps in the proof of \Cref{thm:EquiForm} using the subscript $\mathrm{NS}$.
\begin{enumerate}[1.)]
    \item We proceed with the same two steps as in Step \ref{1.)NS}.)$_{\mathrm{NS}}$.
    \begin{enumerate}[a.)]
        \item \label{regPEu} We show that $p,p^*\in L^2_T(W^{1,r}\cap L^2_0)$ and that $(V,p^*)$ satisfies \eqref{eq:momPstar} with $\nu=0$, where the fields paired with $W$ belong to $L^2_T(L^r(\Sec))$.
        For this, one can repeat the arguments given in Step \ref{RegP}.)$_{\mathrm{NS}}$.
        The only difference is that \Cref{thm:reg(V.d)V} does not apply due to the different boundary conditions.
        Instead, we exploit the higher regularity $V \in L^\infty_T(H^1(\Sec))$ and the estimate preceding \Cref{thm:reg(V.d)V} to obtain $\vert V\vert^2\in L^2_T(W^{1,r})$.
        \item \label{PsiHwpEu} As in Step \ref{PsiHwp}.)$_{\mathrm{NS}}$, we verify that $(\psi,H,\omega,p^*)$ possesses the regularity required by \refwl{def:HarmStreamEu} and satisfies its equations.
        We begin with the regularity and the initial conditions.
        Invoking the Hodge decomposition for $V \in L^\infty_T(H^1(\Sec) \cap J_{L^2})$ and using the Sobolev embeddings, we obtain:
        \begin{align*}
            &\psi\in L^2_T(W^{1,s}\cap H^1_0)\cap C([0,T];H^1_0),
            \\&H\in L^2_T(L^s(\Sec)\cap H_N)\cap C([0,T];H_N),
            \\&\partial_t(\mathrm{rot}(\psi)+H)\in L^2_T(L^r(\Sec)),
            \\&\omega\in L^2_T(W^{1,r})\cap L^\infty_T(L^2)\cap C([0,T];L^r).
        \end{align*}
        Since $V\in C([0,T];J_{L^2})$, we have $H_t\vert_{t=0}=H_0$ and $\omega_t\vert_{t=0}=\mathrm{curl}(V_0)=\omega_0\in L^2\cap L^r$ (using $\omega\in L^\infty_T(L^2)\cap C([0,T];L^r)$ for the $L^2$-regularity at $t=0$).
        On closed surfaces, $\omega\in L^2_T(L^1_0)$ and $\omega_0$ has zero mean, since $(\mathrm{curl}(V_t),1)=(V_t,\mathrm{rot}(1))=0$.\\
        Next, we verify the kinematic equation \eqref{eq:001Eu}.
        We start from $\omega=\mathrm{curl}(\mathrm{rot}(\psi))+\mathrm{curl}(H)$, i.e.,
        \begin{align*}
            (\omega,\chi)=(\mathrm{curl}(\mathrm{rot}(\psi))+\mathrm{curl}(H),\chi),\qquad\forall\chi\in\mathbb{X}_{\mathrm{kin},0}=W^{1,r}_0.
        \end{align*}
        Both sides are well-defined, since $V \in L^\infty_T(H^1(\Sec))$ and $H_N\subset H^1(\Sec)$ (\Cref{thm:H0divHcurlInH1}) give $\mathrm{rot}(\psi),H\in L^\infty_T(H^1(\Sec))$, and $\chi\in W^{1,r} \subset L^2$ (\Cref{lem:s&r}).
        Here, $\mathrm{curl}(H)=0$ by the definition of $H_N$.
        For the first term, Green's formula gives $(\mathrm{curl}(\mathrm{rot}(\psi)),\chi)=(\mathrm{rot}(\psi),\mathrm{rot}(\chi))$, where the boundary term vanishes for $\chi\in W^{1,r}_0$.
        Hence \Cref{eq:001Eu} holds in $L^2_T((W^{1,r}_0)')$.\\
        To derive \eqref{eq:002Eu} and \eqref{eq:003Eu}, we start from \eqref{eq:momPstar} with $\nu=0$ and test with $W=\mathrm{rot}(\phi)$, $\phi\in H^2\cap H^1_0\subset\mathbb{X}_{\mathrm{kin},0}$, and with $W=Z\in\mathbb{X}_H$, respectively.
        Both choices are admissible by Step \ref{regPEu}.) and observation (\ref{obs:ext}) in the proof of \Cref{thm:EquiForm}, and the pressure term vanishes by observation (\ref{obs:grad}).
        For $W=Z$, the arguments from \Cref{lem:HarmStreamVort} ii.) show that $\partial_t(H,Z)=(\partial_tV,Z)$, which gives \eqref{eq:003Eu} as in Step \ref{PsiHwp}.)$_{\mathrm{NS}}$.
        For $W=\mathrm{rot}(\phi)$, we exploit \eqref{eq:001Eu} and Green's formula to rewrite the remaining terms.
        By \eqref{eq:001Eu}, $(\partial_tV,\mathrm{rot}(\phi))=\partial_t(\mathrm{rot}(\psi),\mathrm{rot}(\phi))=(\partial_t\omega,\phi)$.
        Green's formula gives $(F,\mathrm{rot}(\phi))=(\mathrm{curl}(F),\phi)$ and $(\omega JV,\mathrm{rot}(\phi))=(V(\omega),\phi)$, where the boundary terms vanish as $\phi\in H^1_0$, and $\mathrm{curl}(\omega JV)=\mathrm{div}(\omega V)=V(\omega)$ since $\omega V\in W^{1,\overline{r}}(\Sec)$ and $\mathrm{div}(V)=0$.
        This is \eqref{eq:002Eu} for all $\phi\in H^2\cap H^1_0$.
        It remains to extend \eqref{eq:002Eu} to all $\phi\in\mathbb{X}_\omega=L^{\overline{s}}$.
        The arguments from \Cref{lem:HarmStreamVortEu} ii.) show that $\partial_t\omega$, $\mathrm{curl}(F)$ and $V(\omega)$ belong to $L^2_T(L^{\overline{r}})$ for any $\overline{r}\in (1,r)$.
        Hence \eqref{eq:002Eu} extends by density, where on closed surfaces we use that \eqref{eq:002Eu} also holds for constant $\phi$, as $\partial_t\omega$, $\mathrm{curl}(F)$ and $V(\omega)=\mathrm{div}(\omega V)$ have zero mean.\\
        Finally, to derive the pressure equation \eqref{eq:press01Eu}, we test \eqref{eq:momPstar} with $W=\mathrm{grad}(q)$, $q\in\mathbb{X}_p$, which is admissible by observation (\ref{obs:ext}).
        The time-derivative term vanishes by observation (\ref{obs:dt}), exactly as in Step \ref{PsiHwp}.)$_{\mathrm{NS}}$, and $p^*\in L^2_T(W^{1,r}\cap L^2_0)$ by Step \ref{regPEu}.).
    \end{enumerate}
    \item As in Step \ref{2.)NS}.)$_{\mathrm{NS}}$, we first verify the regularity of $V=\mathrm{rot}(\psi)+H$ and the initial condition, and then recover the momentum equation.\\
    As in the proof of \Cref{lem:HarmStreamVortEu} i.), \eqref{eq:001Eu} gives $\mathrm{curl}(V)=\omega$.
    Then $\omega \in L^\infty_T(L^2)$ and assumption \eqref{ass:EulerForH1} give $V \in L^\infty_T(H^1(\Sec) \cap J_{L^2}) \cap C([0,T]; J_{L^2})$, so that the initial velocity $V_0\coloneqq V_t\vert_{t=0}=\mathrm{rot}(\psi_0)+H_0$ is well-defined, with $\psi_0$ determined by $\omega_0$ as in \Cref{rem:HarmStreamEuWD}.
    The remaining regularity assumed in 1.) holds by \Cref{def:HarmStreamEu}, since $\mathrm{curl}(V)=\omega$.\\
    To recover the momentum equation, we first show \eqref{eq:momPstar} with $\nu=0$ for all $W\in H^1_0(\Sec)$.
    This follows from the splitting of $W$ as in Step \ref{2.)NS}.)$_{\mathrm{NS}}$, where, for $\phi\in H^2\cap H^1_0$, the transport equation \eqref{eq:002Eu} is first brought back into the form $(\partial_tV,\mathrm{rot}(\phi))=(F-\omega JV,\mathrm{rot}(\phi))$ by \eqref{eq:001Eu}, Green's formula and $\mathrm{curl}(\omega JV)=V(\omega)$.
    Next, we extend \eqref{eq:momPstar} to the test fields $W\in H^1(\Sec)\cap J_{L^2}$ of the velocity formulation \eqref{eq:Euler}, which need not vanish on $\partial M$.
    By \Cref{def:HarmStreamEu} and \Cref{lem:s&r}, the fields paired with $W$ belong to $L^2_T(L^r(\Sec))$, so that \eqref{eq:momPstar} extends to all $W\in L^s(\Sec)$ by observation (\ref{obs:ext}) in the proof of \Cref{thm:EquiForm}.
    Finally, we remove the gradient terms $(\mathrm{grad}(p^*),W)$ and $(\mathrm{grad}(\vert V\vert^2),W)$, the latter arising from \Cref{lem:Nonlinterm}.
    We have $p^*\in L^2_T(W^{1,r})$ by \Cref{def:HarmStreamEu} and $\vert V\vert^2\in L^2_T(W^{1,r})$ by $V\in L^\infty_T(H^1(\Sec))$, as in Step \ref{regPEu}.).
    Hence both terms vanish for $W\in H^1(\Sec)\cap J_{L^2}\subset L^s(\Sec)$ by observation (\ref{obs:grad}), and \Cref{lem:Nonlinterm} yields \eqref{eq:Euler}.
    For $W\in H^1_0(\Sec)$, we obtain \eqref{eq:NavStokesP} with $\nu=0$ as in Step \ref{2.)NS}.)$_{\mathrm{NS}}$.
\end{enumerate}
\end{proof}

\begin{remark}[Other boundary conditions]
\label{rem:otherBC}
The equivalence results are not tied to the two boundary conditions considered here.
The derivation in \Cref{sec:ScopeOfWork} only uses that the test fields $\mathrm{rot}(\phi)+Z$ span the admissible divergence-free velocities and are $L^2$-orthogonal to gradients; for other boundary conditions, one has to replace the Hodge decomposition \eqref{eq:HodgeAss} by a version adapted to the boundary conditions (e.g., the decomposition of \cite[Thm 1.1]{gol2011hodge} for mixed conditions) and to adjust the test space of the kinematic equation accordingly, exactly as done for the Euler equations above.
For mixed no-slip/no-penetration conditions and for the stress-based outflow condition used in the numerical examples of \Cref{sec:NumEx}, the corresponding weak vorticity formulation is derived in the supplementary material; existence and regularity results for these settings are, however, beyond the scope of this work.
\end{remark}

%% file: content/05NumericalExamples.tex
\section{Discretization and numerical examples}
\label{sec:NumericalResults}
In this section, we translate the streamfunction-vorticity framework developed in the previous sections into a computable numerical algorithm for general surfaces. 
The method mirrors the non-discrete theory by relying on a discrete representation of the Hodge decomposition to decouple the flow components. 
As established previously, this structural approach ensures that the resulting discrete velocity fields are exactly tangential and divergence-free, and that the method is pressure-robust.
Let us briefly explain how the discretization inherits these properties by exactly mimicking the structure of the continuous problem.
\emph{Tangentiality}: all discrete velocities belong to $\mathbb{X}_h$, which is a space of tangential vector fields on the discrete surface $M_h$ by construction.
\emph{Incompressibility}: the discrete velocities are sought in $\mathbb{V}_h=\ker(\mathrm{div}_h)$ (see below), so that $\mathrm{div}_h(V_h)=0$ holds exactly and not only up to a discretization error.
\emph{Pressure robustness} refers to the property that the discrete velocity does not depend on the pressure, so that velocity errors are not polluted by pressure errors \cite{LINKE2014782,john2017divergence}; here it holds by construction, since the pressure does not appear in the equations for $(\psi_h,H_h,\omega_h)$ at all and is recovered only a posteriori.
Algebraically, this rests on the orthogonality of the discrete Hodge decomposition introduced below with respect to a discrete $L^2$-type inner product $(\cdot,\cdot)_h$: $(\mathrm{grad}_h(q_h),W_h)_h=0$ for all $q_h\in\mathbb{Q}_h$ and all $W_h\in\mathbb{V}_h$, so that adding a discrete gradient field to the force leaves the updates of $(\psi_h,H_h,\omega_h)$ -- and hence the discrete velocity -- unchanged and only alters the recovered pressure.
If, moreover, the velocity space consists of exactly divergence-free $H(\mathrm{div})$-conforming fields with vanishing normal trace (cf.\ \Cref{rem:higherorder}) and the integration is exact, then also $(\mathrm{grad}(q),W_h)=0$ for continuous potentials $q$ on the discrete surface $M_h$.
All three properties hold up to the errors introduced by the approximation of $M$ by $M_h$ and by quadrature.

We proceed by first establishing an abstract discrete setting based on this underlying discrete Hodge decomposition. 
We then formulate the semi-discrete (spatially discrete) equations for both the Navier--Stokes and Euler cases. 
Subsequently, we present the fully discrete algorithms using an explicit time-stepping scheme. 
Following this, we introduce the concrete low-order finite element realization (using $\mathbb{P}^1$-based streamfunction and $\mathbb{P}^0$-velocity) employed in our implementation. 
Finally, we conclude with numerical examples that demonstrate the scheme's ability to correctly deal with surfaces with non-trivial topology, thereby validating our vorticity formulation.

\subsection{Abstract semi-discrete setting}
\label{sec:AbstractSetting}
The core of our numerical strategy relies on a discrete realization of the Hodge decomposition. 
Let $M_h$ be a triangulated approximation of the surface $M$ with a corresponding set of triangles $\mathcal{T}_h$, and let $\mathbb{X}_h$ be a finite-dimensional (Hilbert) space approximating the space of tangential vector fields $L^2(\Sec)$.
Our formulation requires two auxiliary finite element (Hilbert) spaces, a potential space $\mathbb{Q}_h$ and a streamfunction space $\mathbb{S}_h$, together with discrete gradient and rotated gradient operators $\mathrm{grad}_h:\mathbb{Q}_h\to \mathbb{X}_h$ and $\mathrm{rot}_h:\mathbb{S}_h\to\mathbb{X}_h$.
We denote all discrete inner products by $(\cdot,\cdot)_h$, and we assume that the inner products properly extend to terms of the form $fX$, where $f$ is a discrete scalar function and $X\in \mathbb{X}_h$.
The discrete divergence is defined via duality as $\mathrm{div}_h\coloneqq-\mathrm{grad}_h^*:\mathbb{X}_h\to \mathbb{Q}_h$, mimicking the fact that the continuous divergence is the negative formal $L^2$-adjoint of the gradient.
In the same way, the discrete curl is defined as $\mathrm{curl}_h\coloneqq\mathrm{rot}_h^*:\mathbb{X}_h\to \mathbb{S}_h$, mimicking the fact that the continuous curl is the formal $L^2$-adjoint of the rotated gradient.
We require that $\mathrm{curl}_h\circ\mathrm{grad}_h=0$, i.e., that $\mathrm{grad}_h(\mathbb{Q}_h)$ and $\mathrm{rot}_h(\mathbb{S}_h)$ are orthogonal.
In analogy to the definition of $H_N$, the space of discrete harmonic fields is $\mathbb{H}_{h,N}\coloneqq\ker(\mathrm{div}_h)\cap\ker(\mathrm{curl}_h)$, and we obtain the orthogonal decomposition
$$\mathbb{X}_h = \mathrm{grad}_h(\mathbb{Q}_h) \oplus \mathrm{rot}_h(\mathbb{S}_h) \oplus \mathbb{H}_{h,N}.$$
A crucial topological requirement is that the dimension of the discrete harmonic space $\mathbb{H}_{h,N}$ matches the first Betti number $b_1$ of the (non-discrete) surface $M$, i.e., the dimension of $H_N$.
This is a condition on the choice of the spaces $\mathbb{X}_h$, $\mathbb{Q}_h$ and $\mathbb{S}_h$ and of the operators $\mathrm{grad}_h$ and $\mathrm{rot}_h$, since $\mathrm{div}_h$ and $\mathrm{curl}_h$ are then determined by duality.

Finally, for the discretization of velocity fields in $J_{L^2}$, we restrict our attention to the subspace $\mathbb{V}_h\subset\mathbb{X}_h$ characterized by the reduced decomposition $\mathbb{V}_h=\mathrm{rot}_h(\mathbb{S}_h)\oplus\mathbb{H}_{h,N}$.
This corresponds to the kernel of $\mathrm{div}_h$.
We denote the discrete space for the approximation of the vorticity $\omega$ by $\mathbb{W}_h$.
We require that $\mathbb{S}_h \subset \mathbb{W}_h$, that $1 \in \mathbb{W}_h$, and that $\mathrm{rot}_h$ extends to a mapping $\mathbb{W}_h \to \mathbb{X}_h$.
The discrete curl, $\mathrm{curl}_h$, always denotes the adjoint of $\mathrm{rot}_h$ on $\mathbb{S}_h$.

Based on the abstract spaces defined above, we now present the semi-discrete (spatially discrete) streamfunction-vorticity formulations. 
Throughout this formulation, the discrete velocity field is defined by $V_h = \mathrm{rot}_h(\psi_h) + H_h$ with $\psi_h\in\mathbb{S}_h$.

\paragraph{Navier--Stokes equations.}
The semi-discrete solution $(\psi_h, H_h, \omega_h)$ is defined as a Bochner function taking values in $\mathbb{S}_h \times \mathbb{H}_{h,N} \times \mathbb{W}_h$. 
This solution satisfies the following system, which is the discrete counterpart to \eqref{eq:001}, \eqref{eq:002} and \eqref{eq:002H}:
\begin{align}
    (\omega_h,\chi_h)_h&=(V_h,\mathrm{rot}_h(\chi_h))_h, &&\forall\chi_h\in\mathbb{W}_h,
    \tag{$\mathrm{NS}^1_{\omega_h}$}
    \label{eq:003} \\
    (\partial_t \omega_h,\varphi_h)_h&=(F_h-\omega_h {J}_h V_h-\nu \mathrm{rot}_h(\omega_h)+2\nu \kappa_h V_h,\mathrm{rot}_h(\varphi_h))_h, &&\forall\varphi_h\in\mathbb{S}_h,\label{eq:004}
    \tag{$\mathrm{NS}^2_{\omega_h}$}\\
    (\partial_t H_h,Z_h)_h&=(F_h-\omega_h {J}_h V_h-\nu \mathrm{rot}_h(\omega_h)+2\nu \kappa_h V_h,Z_h)_h, &&\forall Z_h\in\mathbb{H}_{h,N}.\label{eq:004H}
    \tag{$\mathrm{NS}^3_{\omega_h}$}
\end{align}
Here, $F_h$ is the external force represented by a Bochner function with values in $\mathbb{X}_h$, $\kappa_h$ is a discrete Gaussian curvature, which we assume to be given as a discrete scalar function (see \Cref{sec:LowOrderAlgo} for the choice made in our experiments), and $J_h:\mathbb{X}_h\to\mathbb{X}_h$ denotes the discrete rotation operator. 
The system is supplemented with initial conditions $(\psi_{h,0}, H_{h,0})$ obtained by decomposing the discrete initial velocity $V_{h,0} \in \mathbb{V}_h$, which is defined as the $(\cdot,\cdot)_h$-orthogonal projection of the non-discrete initial velocity $V_0$ onto $\mathbb{V}_h$.

\paragraph{Euler equations.}
For the Euler equations, the spatially discrete system takes the form, mirroring the non-discrete formulation \eqref{eq:001Eu} to \eqref{eq:003Eu}:
\begin{align}
    (\omega_h, \varphi_h)_h&=(\mathrm{rot}_h(\psi_h), \mathrm{rot}_h(\varphi_h))_h, &&\forall\varphi_h\in\mathbb{S}_h,
    \tag{$\mathrm{E}^1_{\omega_h}$} 
    \label{eq:004Eu}\\
    (\partial_t\omega_h,\chi_h)_h&=(\mathrm{curl}_h(F_h)-V_h(\omega_h),\chi_h)_h, &&\forall\chi_h\in\mathbb{W}_h,
    \tag{$\mathrm{E}^2_{\omega_h}$} 
    \label{eq:005Eu} \\
    (\partial_tH_h,Z_h)_h&=(F_h-\omega_h J_h V_h,Z_h)_h, &&\forall Z_h\in\mathbb{H}_{h,N}.
    \tag{$\mathrm{E}^3_{\omega_h}$} 
    \label{eq:006Eu}
\end{align}
Here, $V_h(\omega_h) \coloneqq \langle V_h, \mathrm{grad}_h(\omega_h) \rangle$ denotes the discrete advection term.
The force enters \eqref{eq:005Eu} through its discrete curl, $\mathrm{curl}_h(F_h)\in\mathbb{S}_h\subset\mathbb{W}_h$, so that only $F_h$ is required and not the curl of the force.
Since $\mathrm{curl}_h$ vanishes on $\mathrm{grad}_h(\mathbb{Q}_h)$ and on $\mathbb{H}_{h,N}$, adding a discrete gradient field to the force changes neither \eqref{eq:005Eu} nor \eqref{eq:006Eu}.
Note that the adjoint of $\mathrm{rot}_h$ on all of $\mathbb{W}_h$ must not be used here: it ignores the boundary term $\langle\langle F,T\rangle,\chi\rangle_{\partial M}$ of Green's formula, which does not vanish for a general force $F$.
Note that due to the exchanged roles of the test functions $\varphi_h$ and $\chi_h$ compared to \eqref{eq:003} and \eqref{eq:004} (cf.\ \Cref{rem:TestSpaces}), we can exploit the orthogonality $\mathbb{H}_{h,N} \perp \mathrm{rot}_h (\mathbb{S}_h)$ to simplify \eqref{eq:004Eu}.
Similar to before, the system requires initial conditions $(H_{h,0},\omega_{h,0})$.
We obtain $H_{h,0}$ by extracting the harmonic component of $V_{h,0} \in \mathbb{V}_h$, while $\omega_{h,0}$ is defined as the $(\cdot,\cdot)_h$-orthogonal projection of the initial vorticity $\omega_0\coloneqq \mathrm{curl}(V_0)$ onto $\mathbb{W}_h$.
In the case of a closed surface ($\partial M_h=\emptyset$), we impose the additional requirement that $\omega_{h}$ has a vanishing spatial mean, i.e.\ $(\omega_h,1)_h=0$.
This ensures consistency with the non-discrete framework, where the zero-mean property is a necessary condition for $\omega$ to be the curl of $V$.

\paragraph{Pressure.}
For both the Navier--Stokes and Euler ($\nu=0$) equations, the discrete total pressure $p^*_h\in \mathbb{Q}_h$ is recovered a posteriori from $(\psi_h,H_h,\omega_h)$ via the discrete counterpart of the pressure Poisson equation of \Cref{sec:MainContriPres},
\begin{align}
    (\mathrm{grad}_h(p^*_h),\mathrm{grad}_h(q_h))_h&= (F_h-\omega_h J_h V_h -\nu\mathrm{rot}_h(\omega_h)+2\nu\kappa_h V_h,\mathrm{grad}_h(q_h))_h, &&\forall q_h\in\mathbb{Q}_h.
    \label{eq:press02}
    \tag{$P_h$}
\end{align}

\subsection{Algorithms and low-order finite element realization}
\label{sec:LowOrderAlgo}

In this subsection, we discretize the semi-discrete Navier--Stokes and Euler equations in time. 
For this purpose, we employ a standard first-order explicit Euler scheme. 
We note that a similar fully discrete scheme was implemented in \cite{LiuWeinan} for the Navier--Stokes equations, though restricted to the special case of simply connected domains in $\mathbb{R}^2$ (which implies $\mathbb{H}_{h,N}=\{0\}$). 
Furthermore, our spatial procedure coincides with the one described in \cite[Algo 2]{FluidCohomology} for the Euler equations, with the distinction that they employ a higher-order Runge--Kutta scheme in time.
After stating the algorithms, we conclude the subsection by detailing the concrete low-order finite element realization used in our experiments.

We use explicit Euler primarily to keep each time step transparent: the updates below reduce to mass-matrix, harmonic, and scalar Poisson solves, and the absence of a pressure constraint is visible directly.
As usual for an explicit treatment, the time step must satisfy a mesh- and solution-dependent stability restriction.
Schematically, one expects a convective condition $\Delta t\,\Vert V_h\Vert/h\lesssim 1$ and, for Navier--Stokes, in addition a parabolic restriction $\nu\Delta t/h^2\lesssim 1$.
In the Schäfer--Turek experiment below we use $\Delta t=40h^2$ with $\nu=10^{-3}$, hence $\nu\Delta t/h^2=0.04$.
An implicit or semi-implicit treatment is also possible.
It would couple the streamfunction, harmonic component, and vorticity through nonlinear or linearized transport/diffusion blocks, but it would still not introduce a velocity--pressure saddle-point constraint because the pressure is absent from the velocity update and is reconstructed a posteriori.

\paragraph{Navier--Stokes algorithm.}
Let $\Delta t>0$ be the time step size. We denote the approximation of a variable $X$ at time $t_n=n\Delta t$ by $X^n$.\\
Given the initial data $\psi_h^0\in\mathbb{S}_h$ and $H_h^0\in\mathbb{H}_{h,N}$, we use $V_h^0=\mathrm{rot}_h(\psi_h^0)+H_h^0$ to compute $\omega_h^0$ via \eqref{eq:003}.
Thus, assuming we are given the solution $(\psi_h^n, H_h^n, \omega_h^n)$ at time $t_n$, we use $V_h^n = \mathrm{rot}_h(\psi_h^n) + H_h^n$ to define the functional $$R^n(W)\coloneqq (F_h^n-\omega_h^n J_h V_h^n-\nu\mathrm{rot}_h(\omega_h^n)+2\nu\kappa_h V_h^n,W)_h,$$ for $W\in\mathbb{V}_h$.
With this definition, we perform the step $t_n \to t_{n+1}$ as follows:
\begin{enumerate}[1.)]
    \item Compute $\psi_h^{n+1}\in\mathbb{S}_h$ by solving the Laplace problem
    \begin{align*}
        (\mathrm{rot}_h(\psi_h^{n+1}),\mathrm{rot}_h(\varphi_h))_h&=(\omega_h^n,\varphi_h)_h+\Delta t R^n(\mathrm{rot}_h(\varphi_h)),~\forall \varphi_h \in \mathbb{S}_h.
    \intertext{\item Compute $H_h^{n+1}\in\mathbb{H}_{h,N}$ by solving}
        (H_h^{n+1},Z_h)_h&=(H_h^n,Z_h)_h+\Delta t R^n(Z_h),~\forall Z_h\in\mathbb{H}_{h,N}.
    \intertext{\item Using $V_h^{n+1}=\mathrm{rot}_h(\psi_h^{n+1})+H_h^{n+1}$, we compute $\omega_h^{n+1}\in\mathbb{W}_h$ by solving}
        (\omega_h^{n+1},\chi_h)_h&=(V_h^{n+1},\mathrm{rot}_h(\chi_h))_h,~\forall \chi_h\in\mathbb{W}_h.
    \end{align*}
\end{enumerate}

\begin{remark}
    Similar to \Cref{rem:SpecCases}, the algorithm simplifies in specific geometric settings:
    \begin{itemize}
        \item If the discrete Gaussian curvature $\kappa_h$ is globally constant, the curvature terms decouple due to orthogonality:
        \begin{align*}
            (\kappa_h V_h^n,\mathrm{rot}_h(\varphi_h))_h&=\kappa_h(\mathrm{rot}_h(\psi_h^n),\mathrm{rot}_h(\varphi_h))_h,\\
            (\kappa_h V_h^n,Z_h)_h&=\kappa_h(H_h^n,Z_h)_h.
        \intertext{\item If the surface is closed ($\partial M_h=\emptyset$), then part of the viscous term vanishes from the harmonic update}
        (\mathrm{rot}_h(\omega_h^n),Z_h)_h&=0.
        \end{align*}
    \end{itemize}
\end{remark}

\paragraph{Euler algorithm.}
Given the initial data $H_h^0 \in \mathbb{H}_{h,N}$ and $\omega_h^0 \in \mathbb{W}_h$, where we impose a zero-mean condition on $\omega_h^0$ if $\partial M_h = \emptyset$, we solve \eqref{eq:004Eu} to find $\psi_h^0$, which implies $V_h^0=\mathrm{rot}_h(\psi_h^0) + H_h^0$.
Assuming we are given the solution $(\psi_h^n,H_h^n,\omega_h^n)$ at time $t_n$, we use $V_h^n=\mathrm{rot}_h(\psi_h^n)+H_h^n$ and perform the step $t_n\to t_{n+1}$ as follows:
\begin{enumerate}[1.)]
    \item Update $\omega_h^{n+1}\in\mathbb{W}_h$ by solving
    \begin{align*}
        (\omega_h^{n+1},\chi_h)_h&=(\omega_h^n,\chi_h)_h+\Delta t (\mathrm{curl}_h(F_h^n)-V_h^n(\omega_h^n),\chi_h)_h,~\forall \chi_h \in \mathbb{W}_h.
    \intertext{\item Update $H_h^{n+1} \in \mathbb{H}_{h,N}$ by solving}
        (H_h^{n+1},Z_h)_h&=(H_h^n,Z_h)_h+\Delta t (F_h^n-\omega_h^n J_h V_h^n,Z_h)_h,~\forall Z_h\in\mathbb{H}_{h,N}.
    \intertext{\item Recover $\psi_h^{n+1}\in\mathbb{S}_h$ by solving the Laplace problem}
        (\mathrm{rot}_h(\psi_h^{n+1}),\mathrm{rot}_h(\varphi_h))_h&=(\omega_h^{n+1},\varphi_h)_h,~\forall\varphi_h\in\mathbb{S}_h.
    \end{align*}
\end{enumerate}
In step 1.), $\mathrm{curl}_h(F_h^n)\in\mathbb{S}_h$ is obtained from $(\mathrm{curl}_h(F_h^n),\varphi_h)_h=(F_h^n,\mathrm{rot}_h(\varphi_h))_h$ for all $\varphi_h\in\mathbb{S}_h$, which is a mass-matrix problem.
Finally, we note that if $\partial M_h=\emptyset$, the update in step 1.) preserves  the vanishing mean of $\omega_h^{n+1}$ by construction.
A key distinction from the Navier--Stokes algorithm lies in the solution order: here, we explicitly update $\omega_h^{n+1}$ and $H_h^{n+1}$ before recovering $\psi_h^{n+1}$.
In the Navier--Stokes case, the computational roles of $\omega_h^{n+1}$ and $\psi_h^{n+1}$ are effectively interchanged.
\paragraph{Low-order finite element realization.}
To implement the fully discrete scheme, we select concrete finite element spaces that fulfill the structural requirements of the abstract Hodge decomposition on a given triangulation $\mathcal{T}_h$ of the polyhedral surface $M_h$.
This is one concrete realization of the order-independent abstract setting above.
We choose the lowest-order pair because the resulting velocity is piecewise constant and therefore makes explicit why the weak vorticity formulation is useful even when the discrete velocity is not $H^1$-conforming.
We adopt the following low-order setting:
\begin{itemize}
    \item $\mathbb{X}_h \coloneqq [\mathbb{P}^0(\mathcal{T}_h)]^2$: piecewise constant vector fields.
    \item $\mathbb{W}_h \coloneqq S^1(\mathcal{T}_h)$: continuous $\mathbb{P}^1$-Lagrange elements.
    \item $\mathbb{S}_h \coloneqq S^1_0(\mathcal{T}_h)$: $S^1(\mathcal{T}_h)$ elements $\varphi$ with $\varphi\vert_{\partial M_h}=0$ or zero mean if $\partial M_h=\emptyset$.
    \item $\mathbb{Q}_h \coloneqq \mathrm{CR}^1(\mathcal{T}_h) \cap L^2_0(\mathcal{T}_h)$: Crouzeix--Raviart elements with zero mean.
\end{itemize}
Throughout this realization, $(\cdot,\cdot)_h$ denotes the standard $L^2$-inner product on $M_h$.
We employ the standard element-wise gradient for $\mathrm{grad}_h$ and define the discrete rotation via $\mathrm{rot}_h(\cdot)\coloneqq -J_h\mathrm{grad}_h(\cdot)$. 
Here, $J_h$ represents the element-wise rotation by $90^\circ$ in the tangent plane. 
Since $M_h \subset \mathbb{R}^3$, this operation is realized as $J_h(X_h) = N_{M_h} \times X_h$ for any $X_h \in \mathbb{X}_h$, where $N_{M_h}$ denotes the element-wise unit normal field.

Here, $\mathrm{CR}^1(\mathcal{T}_h)$ denotes the space of piecewise linear functions that are continuous at the edge midpoints, i.e., whose mean values on each edge coincide from both sides.
Crucially, it holds that $\mathrm{grad}_h(\mathrm{CR}^1(\mathcal{T}_h))\perp \mathrm{rot}_h(S^1_0(\mathcal{T}_h))$, i.e., $\mathrm{curl}_h\circ\mathrm{grad}_h=0$.
Indeed, element-wise integration by parts uses $\mathrm{div}(\mathrm{rot}_h\varphi_h)=0$; the interior edge terms cancel because the Crouzeix--Raviart edge averages are single-valued, while the boundary contribution vanishes for $\varphi_h\in S^1_0(\mathcal{T}_h)$.
This is the discrete analogue of the orthogonality of gradient and rotated-gradient components in the continuous Hodge decomposition; see also the discrete Helmholtz decomposition of piecewise constant vector fields of Arnold and Falk \cite{ArnoldFalk1989}.
The discrete harmonic space is therefore the orthogonal complement within $\mathbb{X}_h$
$$
\mathbb{H}_{h,N} = \ker(\mathrm{div}_h)\cap\ker(\mathrm{curl}_h) = \big( \mathrm{grad}_h(\mathbb{Q}_h) \oplus \mathrm{rot}_h(\mathbb{S}_h) \big)^\perp \subset \mathbb{X}_h.
$$
It can be shown that the elements of $\mathbb{H}_{h,N}$ have continuous normal components across edges and are exactly tangential to the boundary $\partial M_h$ \cite[Lem 3.1]{Poelke}. 
Moreover, since $M_h$ is homeomorphic to $M$, the dimension of this space satisfies $\mathrm{dim}(\mathbb{H}_{h,N}) = \mathrm{dim}(H_N)$ (see \cite[Thm 3.2]{Poelke}), ensuring that the discrete cohomology matches the non-discrete one.
Consequently, this choice of spaces satisfies all requirements of the abstract discrete setting.
Finally, the discrete Gaussian curvature $\kappa_h$ used in our experiments is computed as a discrete scalar function as in \cite{Gopalakrishnan_2023,NeunteufelSchoeberl2019}, to which we refer for details.

\begin{remark}[Higher-order realizations]
\label{rem:higherorder}
The abstract setting of \Cref{sec:AbstractSetting} is not tied to the lowest order.
Any discrete de Rham complex on $\mathcal{T}_h$ whose cohomology matches that of $M$ provides a discrete Hodge decomposition with $\dim\mathbb{H}_{h,N}=b_1$ \cite{ArnoldFalkWinther2006}; for instance, continuous Lagrange elements $\mathbb{S}_h=S^{k+1}_0(\mathcal{T}_h)$ for the streamfunction together with the Raviart--Thomas space $\mathbb{X}_h=\mathrm{RT}_k(\mathcal{T}_h)$ or the Brezzi--Douglas--Marini space $\mathbb{X}_h=\mathrm{BDM}_{k+1}(\mathcal{T}_h)$ (indexed such that $\mathrm{div}$ maps both onto $\mathbb{P}^k(\mathcal{T}_h)$), discontinuous $\mathbb{Q}_h=\mathbb{P}^k(\mathcal{T}_h)\cap L^2_0(\mathcal{T}_h)$ and $\mathrm{grad}_h\coloneqq-(\mathrm{div}\vert_{\mathbb{X}_h})^*$ (so that $\mathrm{div}_h=-\mathrm{grad}_h^*$ is the $L^2$-projection of $\mathrm{div}\vert_{\mathbb{X}_h}$ onto $\mathbb{Q}_h$) yield higher-order realizations with exactly divergence-free velocities.
We restrict ourselves to the lowest-order pair here for two reasons: it is the case with the least regularity of the discrete velocities (piecewise constants are not in $H^1$), which is precisely the situation the $L^s$-setting of the vorticity formulation is designed for, and it is the setting used in the discrete-exterior-calculus and geometry-processing literature \cite{elcott2007stable,FluidCohomology,ZhuYinChern2025} on which our algorithms build.
\end{remark}

\paragraph{Basis for discrete harmonic fields.}
For completeness, we briefly outline the construction of a basis for $\mathbb{H}_{h,N}$, following the procedure described in \cite[Algo 4]{FluidCohomology} (see e.g.\ \cite[Sec 4.1]{Poelke} for alternative methods).
We first determine the dimension $l \coloneqq \mathrm{dim}(\mathbb{H}_{h,N})$ using the Euler characteristic relation
\begin{align*}
    l=\vert \mathcal{E}_h\vert-\vert \mathcal{V}_h\vert-\vert \mathcal{T}_h\vert +1+\begin{cases}1,~&\partial M=\emptyset\\0,~&\partial M\neq\emptyset\end{cases},
\end{align*}
where $\vert \mathcal{V}_h\vert$, $\vert \mathcal{E}_h\vert$, and $\vert \mathcal{T}_h\vert$ denote the number of vertices, edges, and triangles in the mesh, respectively.
With $l$ determined, we generate $l$ random vector fields $X_1, \ldots, X_l \in \mathbb{X}_h$.
To extract the harmonic component of each field, we compute its discrete Hodge decomposition by projecting out the rotational and gradient parts.
This is achieved by solving the following Poisson problems for $(\xi_k, \zeta_k) \in \mathbb{S}_h \times \mathbb{Q}_h$:
\begin{align*}
    (\mathrm{rot}_h(\xi_k), \mathrm{rot}_h(\varphi_h))_h&=(X_k, \mathrm{rot}_h(\varphi_h))_h,~\forall \varphi_h \in \mathbb{S}_h, \\
    (\mathrm{grad}_h(\zeta_k),\mathrm{grad}_h(q_h))_h&=(X_k, \mathrm{grad}_h(q_h))_h,~\forall q_h \in \mathbb{Q}_h,
\end{align*}
and subsequently setting $Y_k\coloneqq X_k-\mathrm{rot}_h(\xi_k)-\mathrm{grad}_h(\zeta_k)\in\mathbb{H}_{h,N}$.
Since the initial fields are chosen randomly, the resulting projections $\{Y_k\}_{k=1}^l$ are almost surely linearly independent.
Finally, we apply the Gram--Schmidt process to $\{Y_k\}_{k=1}^l$ to obtain an orthonormal basis for $\mathbb{H}_{h,N}$.

\subsection{Numerical examples}
\label{sec:NumEx}
To validate our approach, we consider three test cases that highlight the method's performance on domains with non-trivial topology -- i.e., domains where harmonic fields play a fundamental role -- and varying boundary conditions.\\
First, we address the planar Schäfer Turek benchmark.
This example serves to quantitatively assess our discrete schemes by computing relevant quantities of interest and comparing them to established values from the literature.
It further poses the challenge of mixed boundary conditions that differ from the standard Dirichlet settings.
For a detailed discussion on a more general discrete Hodge decomposition capable of incorporating this wider class of boundary conditions, we refer the reader to the supplementary material.\\
Next, we extend our validation to curved manifolds by simulating the Schäfer Turek benchmark on a pierced ring and the Kelvin--Helmholtz instability on a torus.
These examples serve to verify that our discrete schemes exhibit the correct qualitative behavior on curved surfaces.

\paragraph{Schäfer Turek benchmark}
In this section, we analyze the Schäfer--Turek benchmark problem 2D-2 from \cite{schafer1996benchmark} in the vorticity formulation.
This benchmark models laminar flow ($Re=100$) through a rectangular channel containing a slightly off-center circular obstacle, leading to a \emph{Kármán vortex street}; see \Cref{fig:TurSchäfer} for a visualization.
For the precise setup, we refer the reader to the supplementary material.

This benchmark poses a significant challenge for vorticity-based formulations for two reasons.
First, the domain is not simply connected, necessitating the correct handling of harmonic fields to capture the physics.
Second, the model requires mixed boundary conditions (Dirichlet on inlet, walls, and obstacle; vanishing tangential velocity and pressure on the outlet) that differ from the pure Dirichlet settings analyzed previously.
These can be incorporated by utilizing a generalized Hodge decomposition from \cite[Thm 1.1]{gol2011hodge} that accommodates mixed boundary conditions.
For more details on the class of admissible boundary conditions and their relation to the stress tensor, we refer the reader to the supplementary material.

Note that the implementation of this mixed Hodge decomposition and the discretization process follows the strategy outlined previously.
However, the enforcement of the boundary condition $\langle V,T\rangle = 0$ at the outlet requires careful treatment, particularly on coarser meshes, as it creates a jump in the tangential velocity across the boundary elements; this in turn induces spurious oscillations in the vorticity near the outlet.
To suppress these artifacts and stabilize the solution, we introduce a local coupling term of order $\mathcal{O}(h^2)$ in the vicinity of the outlet.

To evaluate the simulation quantitatively, we compute the total hydrodynamic force acting on the obstacle boundary $\partial M_{\circ}$, which simplifies under the no-slip condition to $F_\circ = \int_{\partial M_\circ} (-p N + \nu \omega T) \, \mathrm{d}s$.
From the $x$- and $y$-components of $F_\circ$, we compute the dimensionless drag ($C_D$) and lift ($C_L$) coefficients.
Additionally, we compute the Strouhal number ($\mathrm{St}$), which characterizes the frequency of the periodic vortex shedding.
For the precise definitions of these dimensionless quantities and the flow parameters used, we refer the reader to the supplementary material.

Numerical values for these quantities are reported in \Cref{tab:SchTurek} for various mesh resolutions.
While the drag coefficients $C_D^{\min}$ and $C_D^{\max}$ converge closely to the reference values, the lift coefficients $C_L^{\min}$ and $C_L^{\max}$ converge to values slightly distinct from the literature.
This discrepancy likely stems from the outlet boundary conditions.
Our formulation imposes $\langle V,T\rangle=0$ and $\langle\sigma(V,p)N,N\rangle=0$ at $\partial M_{\mathrm{out}}$, whereas the reference values were computed using the standard ``do-nothing'' condition ($\sigma(V,p)N=0$) common in velocity--pressure formulations.

\begin{figure}[tbp]
    \begin{center}
        \hspace*{-2.0cm}
        \rotatebox{37.5}{\includegraphics[width=0.42\textwidth]{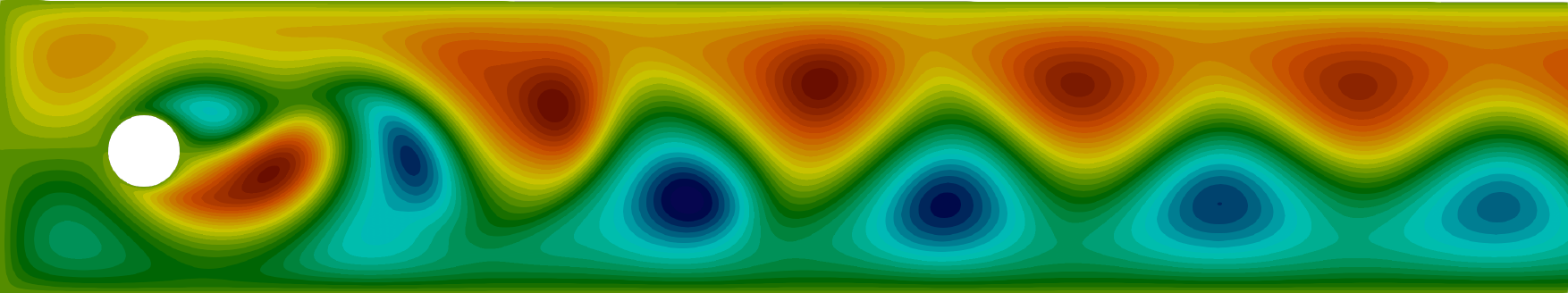}}
        \hspace*{-5.1cm}
        \rotatebox{37.5}{\hspace{1.25cm}\textbf{Streamfunction $\psi$}}
        \hspace*{-1.85cm}
        \rotatebox{37.5}{\includegraphics[width=0.42\textwidth]{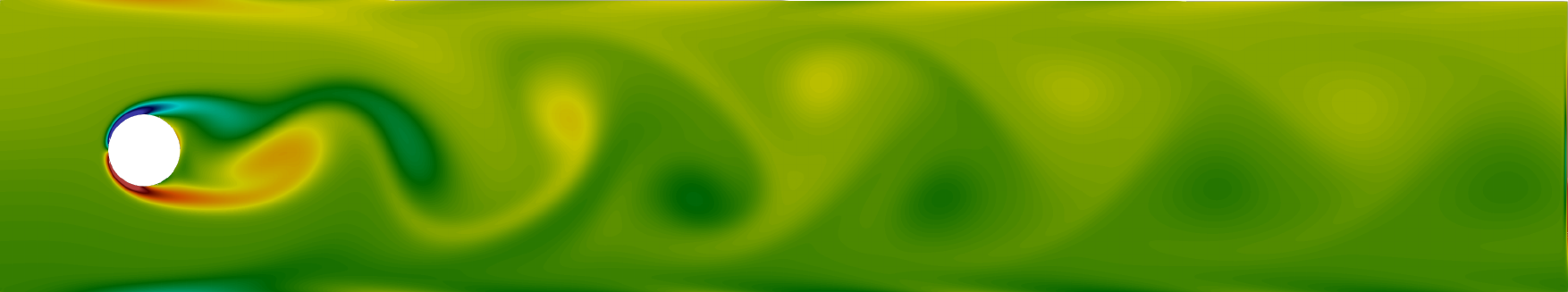}}
        \hspace*{-5.1cm}
        \rotatebox{37.5}{\hspace{1.25cm}\textbf{Vorticity $\omega$}}
        \hspace*{-.8cm}
        \rotatebox{37.5}{\includegraphics[width=0.42\textwidth]{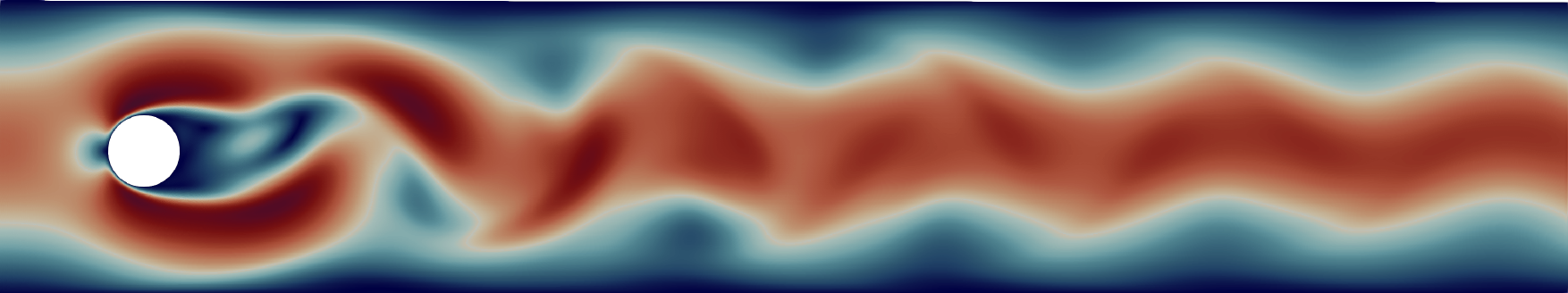}}
        \hspace*{-5.1cm}
        \rotatebox{37.5}{\hspace{1.25cm}\textbf{Velocity $\vert V\vert$}}
    \end{center}
    \caption{Schäfer Turek benchmark. The streamfunction $\psi$, vorticity $\omega$, and velocity field $\vert V\vert$ are displayed, visualizing the shedding of vortices behind the hole.}
    \label{fig:TurSchäfer}
\end{figure}

\begin{table}[tbp]
    \begin{tabular}{c|c|l|l|l|l|l}
        $\Delta t$ & $h$ & $C_D^{\text{min}}$ 
        & $C_D^{\text{max}}$ & $C_L^{\text{min}}$ & $C_L^{\text{max}}$ & St \\\hline
        $6.4 \cdot 10^{-4}$ & $4 \cdot 10^{-3}$ & 3.02488 & 3.09034 & -1.02589 & 1.01468 & 0.30390 \\\hline
        $3.2 \cdot 10^{-4}$ & $2 \sqrt{2} \cdot 10^{-3}$ & 3.11382 & 3.17828 & -1.01779 & 1.00241 & 0.30283\\\hline 
        $1.6 \cdot 10^{-4}$ & $2 \cdot 10^{-3}$ & 3.14485 & 3.20881 & -1.01494 & 0.99732 & 0.30234 \\\hline 
        $0.8 \cdot 10^{-4}$ & $\sqrt{2} \cdot 10^{-3}$ & 3.15727 & 3.22098 & -1.01429 & 0.99501 & 0.30208 \\\hline
        \multicolumn{2}{c|}{Ref. values \cite{LS_CMAME_2016} and \cite{2D-2}} & 3.16430 & 3.22757 & -1.02053 &  0.98580 & 0.30188
    \end{tabular}
    \caption{Schäfer Turek benchmark 2D-2 results for different resolutions.}
    \label{tab:SchTurek}
\end{table}

\FloatBarrier
\paragraph{Schäfer Turek benchmark on a pierced ring}
In this section, we analyze the classical Schäfer Turek benchmark on a pierced ring.
Specifically, we define our surface $M$ as the lateral area of a cylinder aligned with the $z$-axis, having unit diameter and unit height.
This domain is pierced slightly above the equator ($z=0.025$) by a circular hole of diameter $d=1/8$, located on the negative $y$-axis.

To drive the flow, we apply a constant external force with magnitude $\vert F\vert=0.1$ in the clockwise azimuthal direction (tangential to the cylinder surface).
We simulate the dynamics governed by the Navier--Stokes equations with a viscosity of $\nu=0.001$.
The resulting flow accelerates until the driving force is balanced by viscous dissipation, establishing a mean velocity $V_{\mathrm{mean}}$.
Using the hole diameter $d$ as the characteristic length scale and $V_{\mathrm{mean}}$ as the velocity scale, we obtain a Reynolds number of $\Rey\approx 100$.
Consistent with expectations from the flat case discussed above, we observe the formation of a characteristic \emph{Kármán vortex street}, as illustrated in \Cref{fig:PiercedRing}.

\begin{figure}[tbp]
    \begin{center}
    \begin{minipage}{0.32\textwidth}
        \hspace{0.4cm} \includegraphics[width=0.8\textwidth,trim= 2.3cm 4.5cm 1.5cm 5.9cm, clip=True]{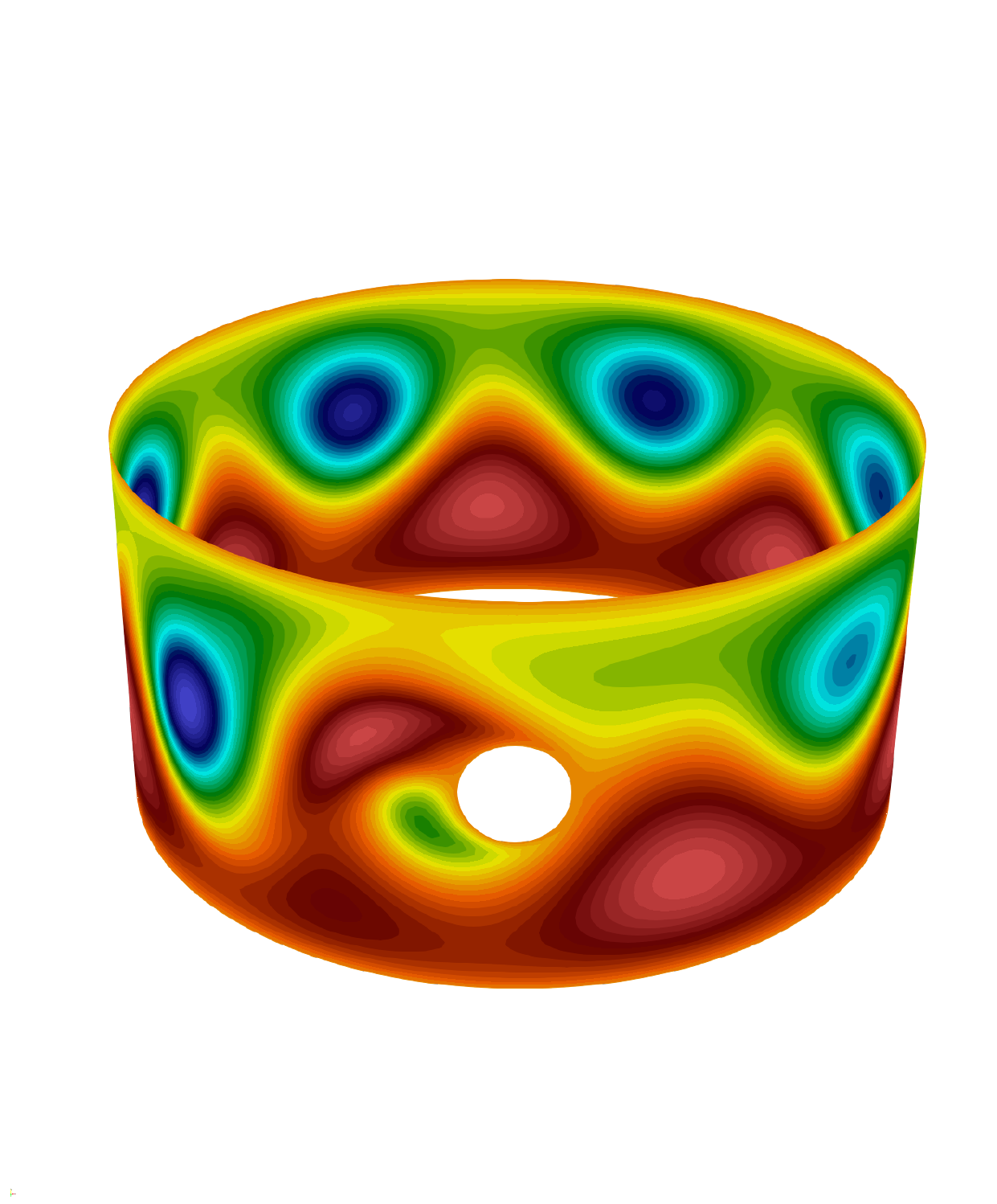}
    \end{minipage}
    \begin{minipage}{0.32\textwidth}
        \hspace{0.4cm} \includegraphics[width=0.8\textwidth,trim= 2.3cm 4.5cm 1.5cm 5.9cm, clip=True]{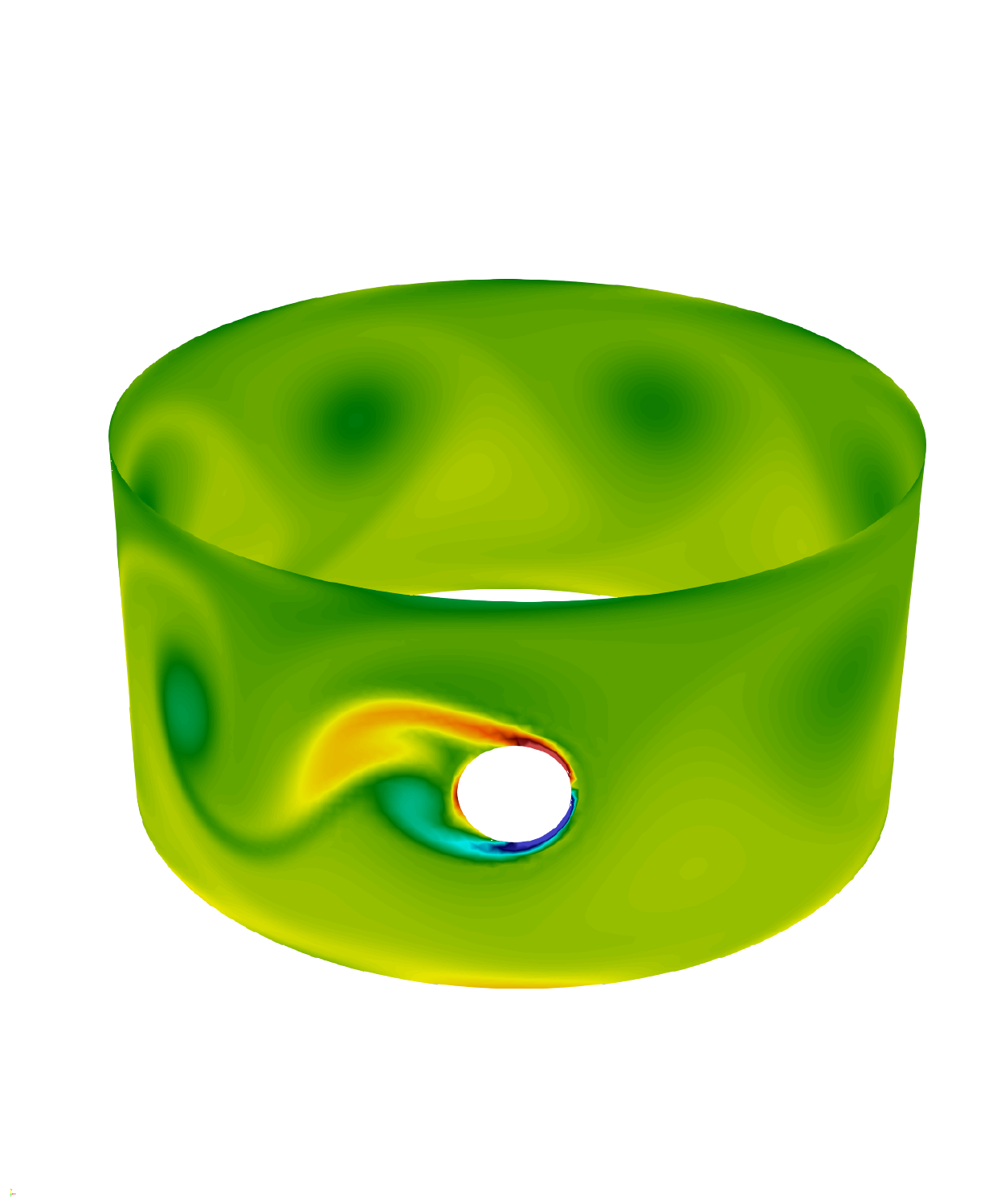}
    \end{minipage}
    \begin{minipage}{0.32\textwidth}
        \hspace{0.4cm} \includegraphics[width=0.8\textwidth,trim= 2.3cm 4.5cm 1.5cm 5.9cm, clip=True]{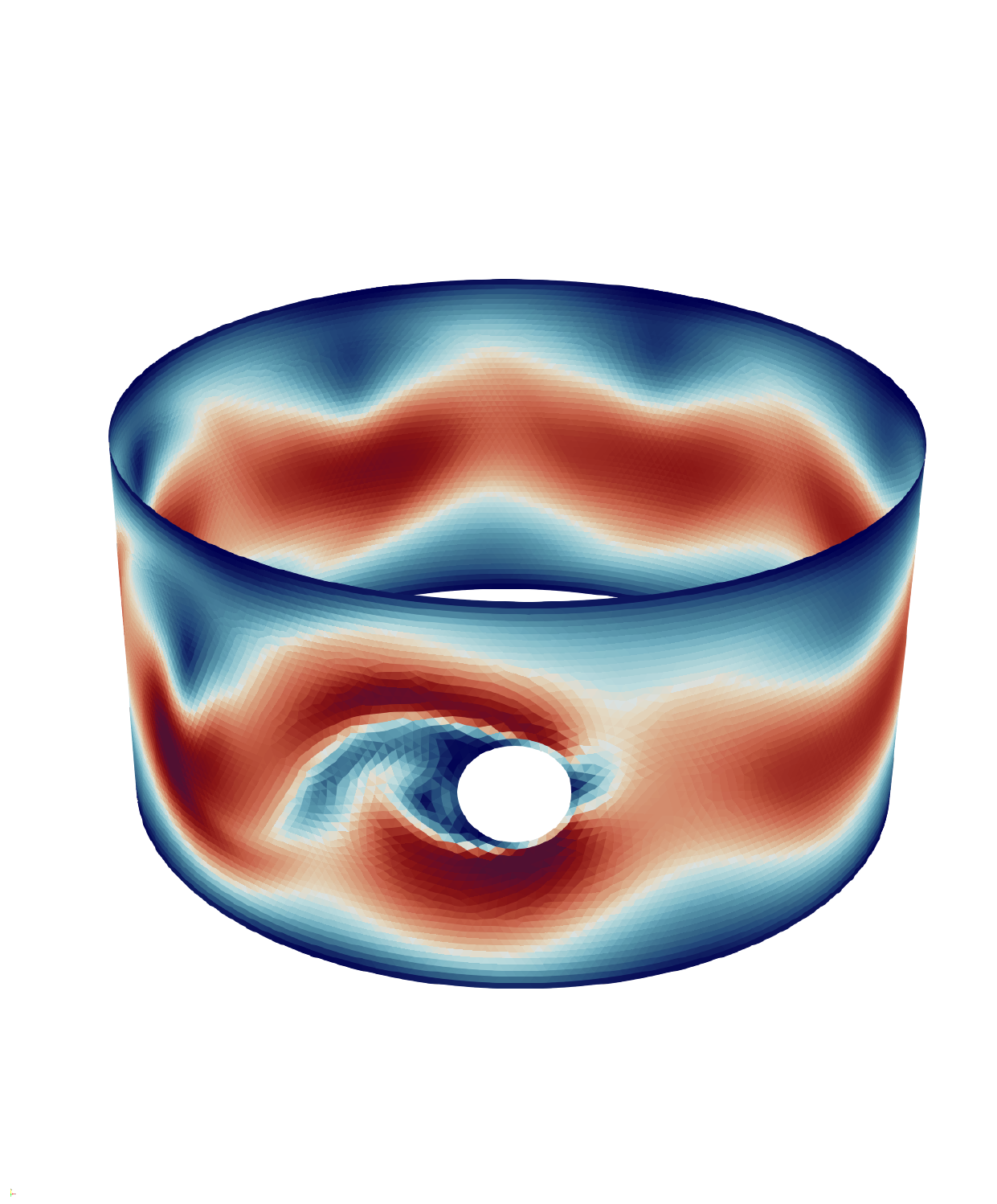}
    \end{minipage}\\
    \vspace{0.1cm}
    \begin{minipage}{0.32\textwidth}
        \begin{center}
            \hspace{0cm}\textbf{Streamfunction $\psi$}
        \end{center}
    \end{minipage}
    \begin{minipage}{0.32\textwidth}
        \begin{center}
            \hspace{0cm}\textbf{Vorticity $\omega$}
        \end{center}
    \end{minipage}
    \begin{minipage}{0.32\textwidth}
        \begin{center}
            \hspace{0cm}\textbf{Velocity $\vert V\vert$}    
        \end{center}
    \end{minipage}\\
    \vspace{-0.4cm}
    \end{center}
    \caption{Schäfer Turek benchmark on a pierced ring. The streamfunction $\psi$, vorticity $\omega$, and velocity field $\vert V\vert$ are displayed, visualizing the shedding of vortices behind the hole.}
    \label{fig:PiercedRing}
\end{figure}

\paragraph{Kelvin--Helmholtz instability on a torus}
In this section, we analyze the Kelvin--Helmholtz instability on a torus governed by the Navier--Stokes equations with viscosity $\nu=0.001$ and vanishing external force ($F=0$).
It is important to emphasize that for the torus, the first Betti number is $b_1=2$, implying that the dimension of the harmonic space is $\dim(\mathbb{H}_{h,N})=2$.
Consequently, harmonic fields play a significant role in the flow dynamics.

We utilize a setup adapted from the classical square periodic domain (see e.g.\ \cite[Ch 2]{schroeder2019}) to the toroidal geometry.
The flow is driven by an initial velocity field $V_0$ designed to generate a shear layer with a Reynolds number of $\Rey=200$.
Physically, $V_0$ creates a counter-rotating flow between the upper and lower halves of the torus, superimposed with a perturbation that breaks the rotational symmetry.
This destabilizes the shear layer, leading to the characteristic roll-up of vortices and their subsequent interaction shown in \Cref{fig:Torus}.
For the precise parameterization and the explicit construction of $V_0$, we refer the reader to the supplementary material.

In the flat planar case with periodic boundary conditions in one direction (topologically representing a ring or annulus), it is well-established that the instability evolves into four distinct vortices, which subsequently merge into two and finally into a single vortex structure, with the long-time behavior being extremely sensitive to small perturbations \cite[Sec 4]{schroeder2019}.
Theoretically, this aligns with the fact that for the Euler equations, the dynamics of four interacting vortices are known to be chaotic, whereas configurations with 1, 2, or 3 vortices are integrable \cite[Sec 11.B]{ArnoldKhesin}.

We observe a qualitatively similar evolution on the torus.
In particular, the chaotic nature of the four-vortex configuration makes the long-time behavior extremely sensitive to small perturbations.
In our simulation, the lack of perfect symmetry in the unstructured mesh acts as such an additional perturbation.
Consequently, while the flow maintains symmetry initially, there exists a time $t^*$ after which the solution becomes visibly asymmetric.
This is particularly evident during the first merging phase, where the two vortex pairs do not merge at diametrically opposite locations.
This axial asymmetry grows with increasing time.
We note that similar symmetry breaking behavior was observed in \cite[Sec 4.2]{KHTorus} during the process of four vortices merging into two.
\begin{figure}[tbp]
    \begin{minipage}{1\textwidth}
    \begin{center}
    \hspace{-0.25cm}
    \includegraphics[width=0.32\textwidth]{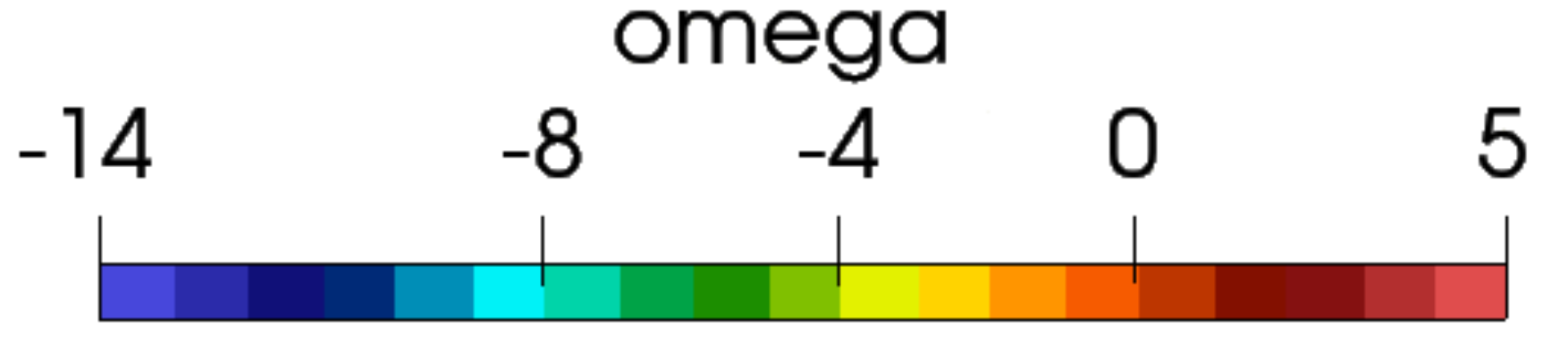} \\
    \includegraphics[width=0.28\textwidth,trim= 5.5cm 0cm 5.5cm 0cm, clip=True]{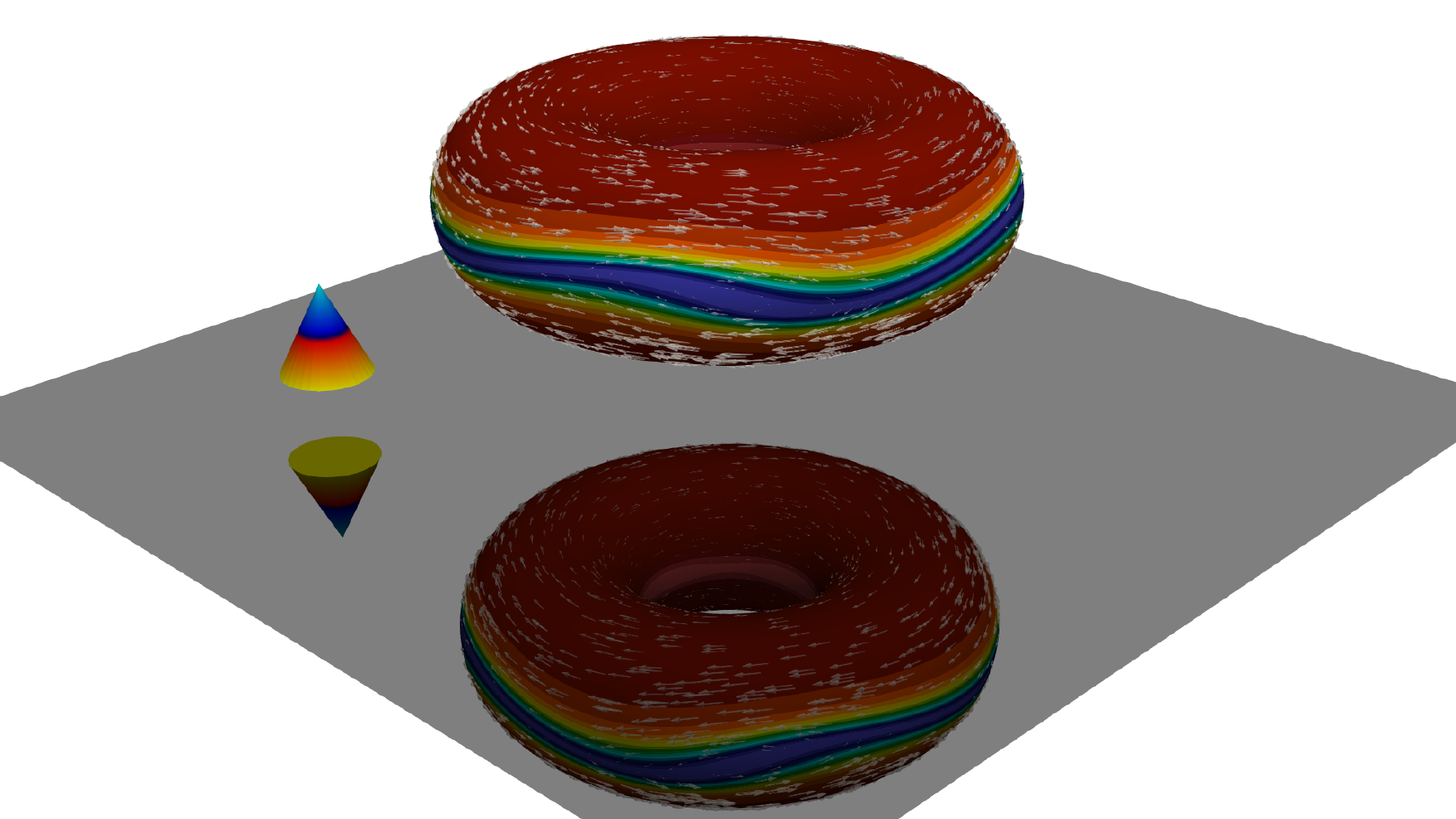}
    \includegraphics[width=0.28\textwidth,trim= 5.5cm 0cm 5.5cm 0cm, clip=True]{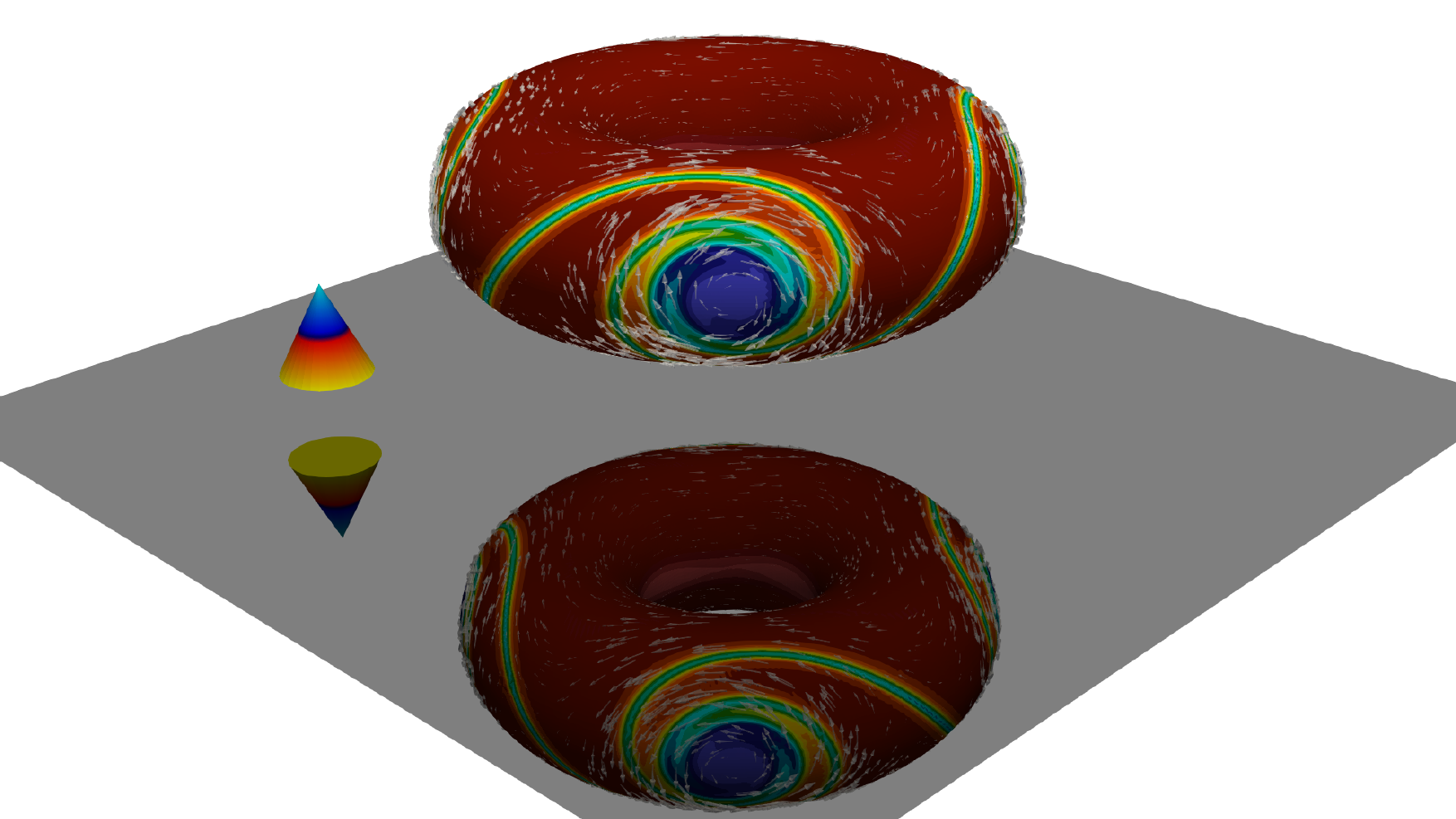}
    \includegraphics[width=0.28\textwidth,trim= 5.5cm 0cm 5.5cm 0cm, clip=True]{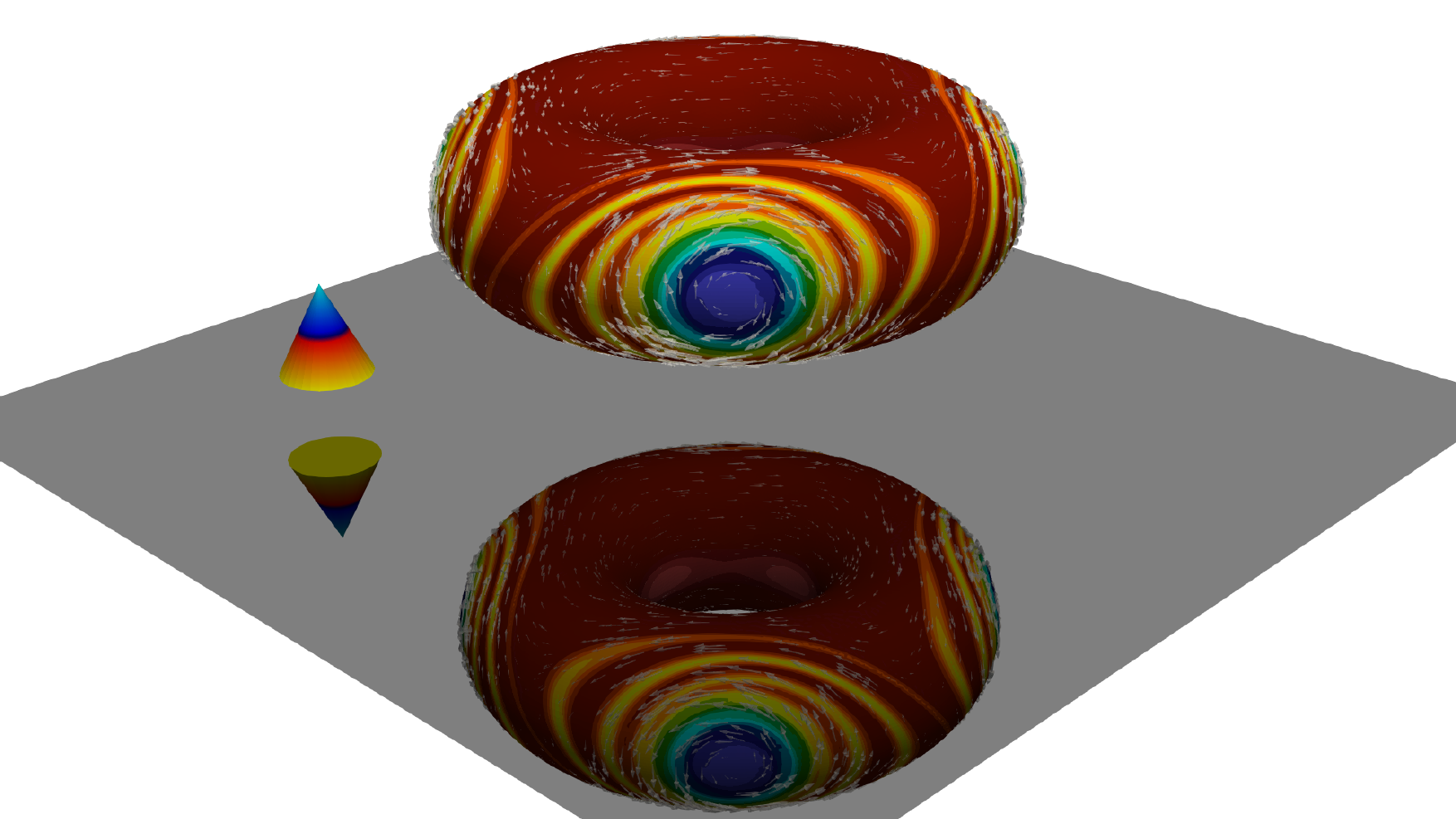} \\
    \includegraphics[width=0.28\textwidth,trim= 5.5cm 0cm 5.5cm 0cm, clip=True]{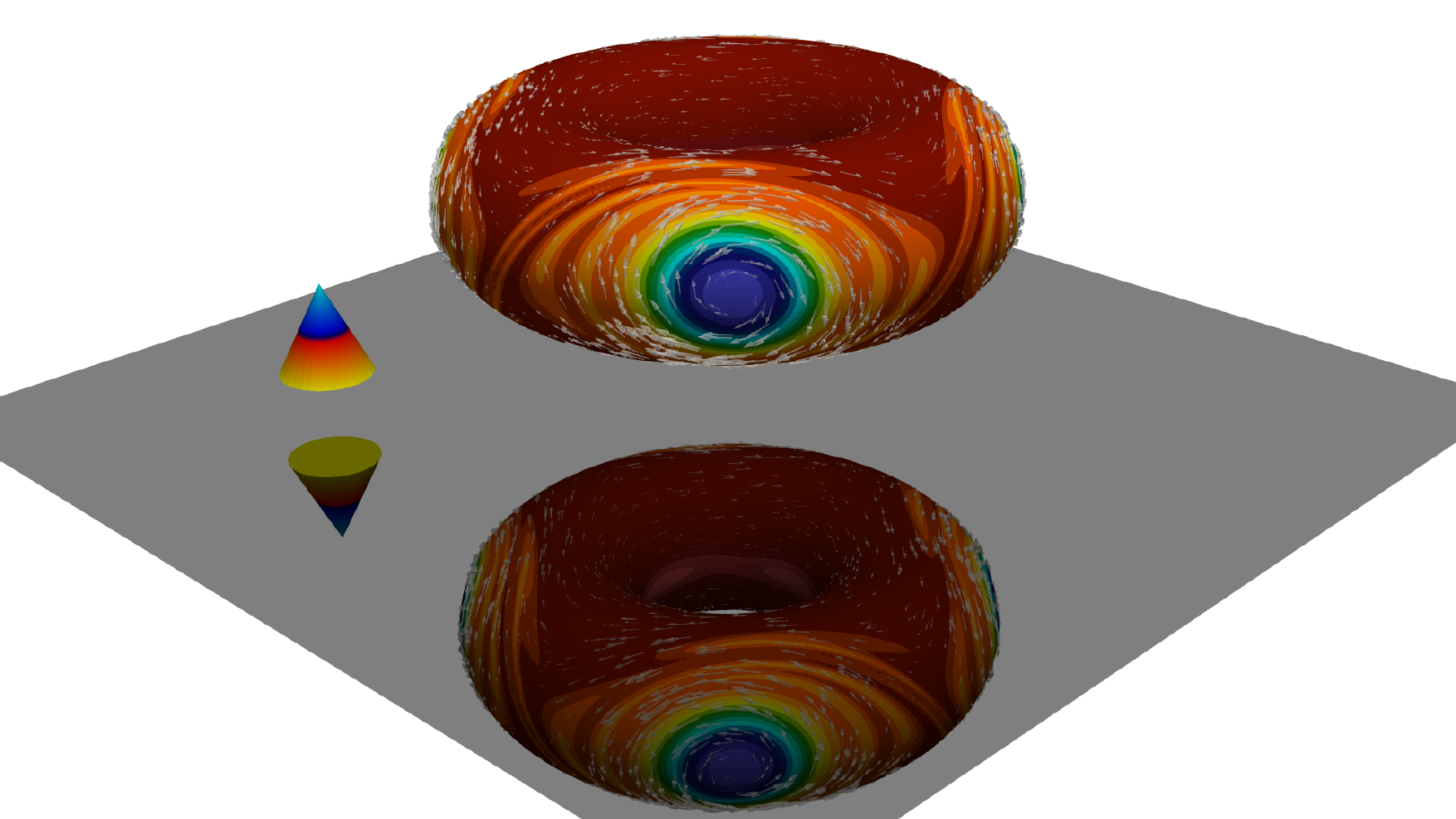}
    \includegraphics[width=0.28\textwidth,trim= 5.5cm 0cm 5.5cm 0cm, clip=True]{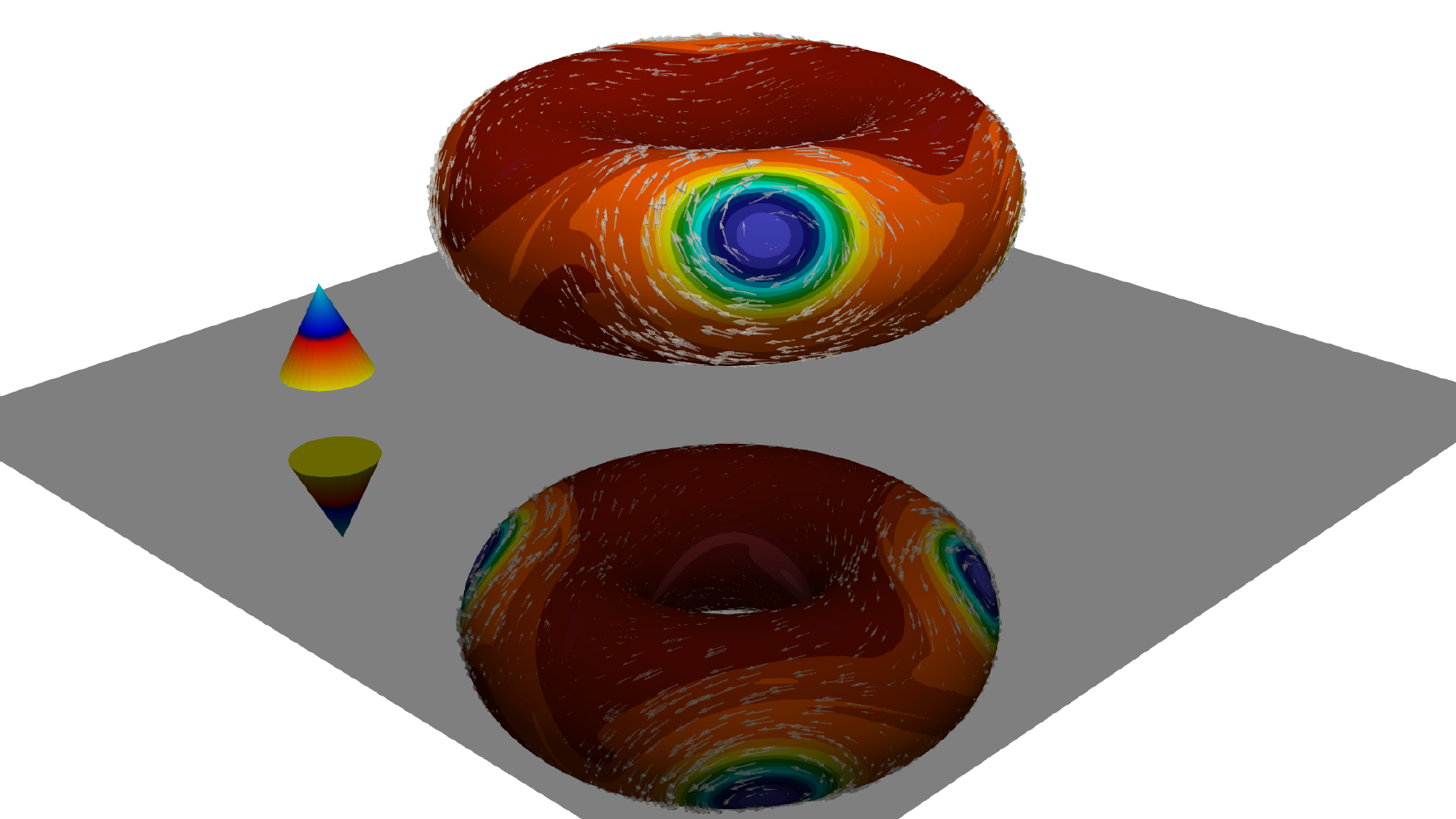}
    \includegraphics[width=0.28\textwidth,trim= 5.5cm 0cm 5.5cm 0cm, clip=True]{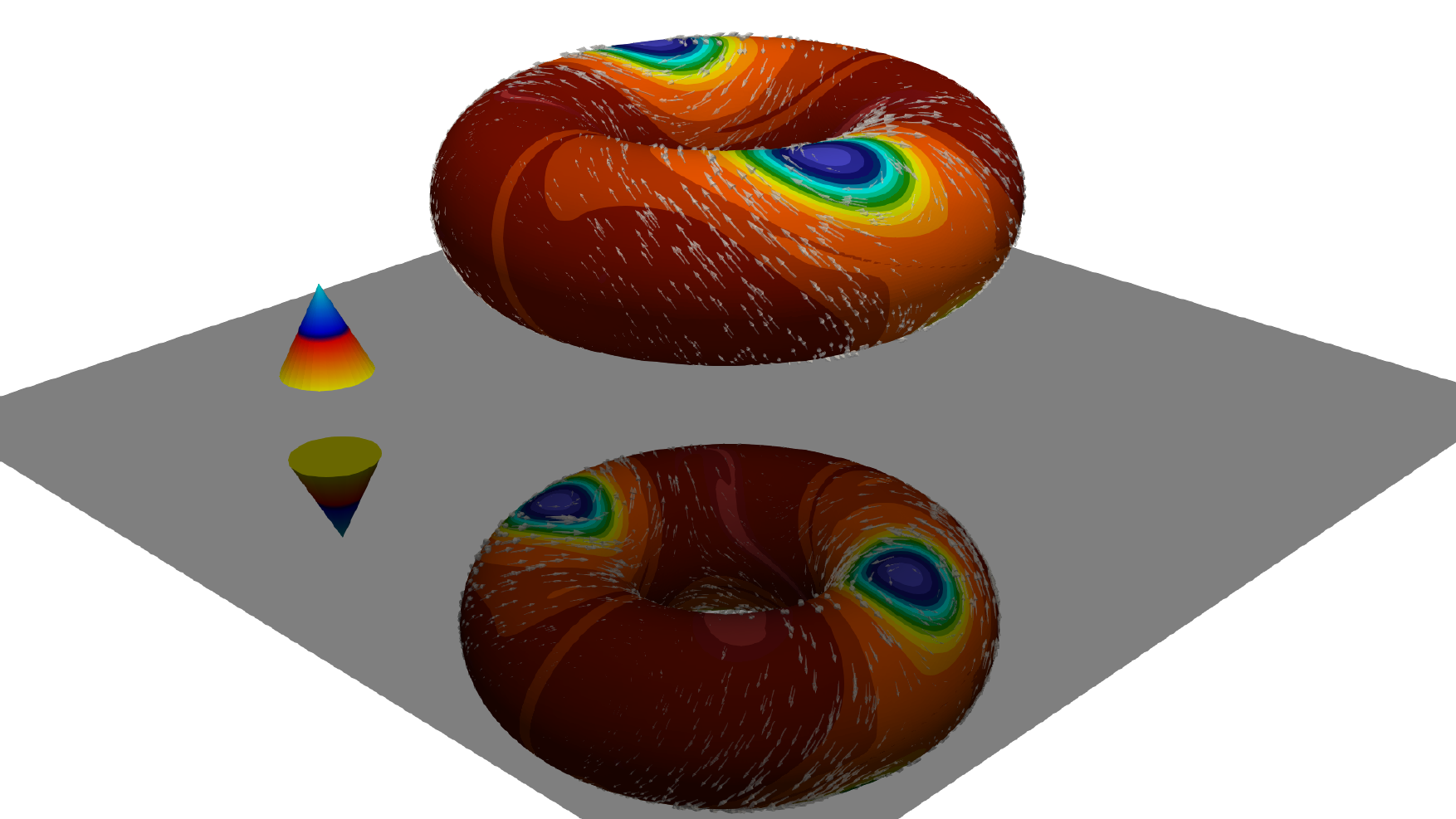}
    \end{center}
    \end{minipage}
    \caption{Kelvin--Helmholtz instability on a torus. Time evolution of vorticity $\omega$ and velocity $V$ (white arrows) governed by the Navier--Stokes equations. A mirror below reveals the rear dynamics.}
    \label{fig:Torus}
\end{figure}

%% file: content/06Appendix.tex
\FloatBarrier
\appendix
\section{Notation}
\label{sec:not}
In this section, we introduce the necessary background on Sobolev spaces, boundary traces, and Bochner spaces.
The notation employed here extends standard conventions from the Euclidean setting.

\paragraph{Surfaces.}
We begin by specifying the class of surfaces considered in this work.

\begin{definition}[Surface with $C^{k,1}$ or $C^{k+1}$ boundary]
\label{def:surface}
    \nextline Let $W$ be a smooth, oriented, two-dimensional Riemannian manifold, possibly with smooth boundary.
    A subset $M \subset W$ is called a \textit{surface with $C^{k,1}$ (or $C^{k+1}$) boundary}, for $k \in \mathbb{N}_0$, if $M$ is the closure of an open, connected subset of $W$, $M$ is compact and if for every $x \in\partial M$, there exist a neighborhood $U\subset W$ of $x$ and a chart $\varphi:U\to\mathbb{R}^2$ of $W$ such that $\varphi(U\cap\partial M)$ is the graph of a $C^{k,1}$ (or $C^{k+1}$) function, with $\varphi(U\cap M^\circ)$ lying on one side of this graph.
\end{definition}

It is worth noting that a surface $M$ with a $C^{k,1}$ (or $C^{k+1}$) boundary constitutes a $C^{k,1}$ (or $C^{k+1}$) submanifold of $W$ \cite[Thm 1.2.1.5]{Grisvard}.
The converse, however, holds only if $M$ is at least a $C^1$ submanifold.\\
Prominent examples of surfaces with $C^{k,1}$ or $C^{k+1}$ boundary include: (i) the closure of any open, bounded, and connected subset of $\mathbb{R}^2$ with a $C^{k,1}$ (or $C^{k+1}$) boundary (see, e.g., \Cref{fig:domains} a)), (ii) any compact, connected, oriented two-dimensional smooth Riemannian manifold $M$, with or without smooth boundary (see, e.g., \Cref{fig:domains} b) and c)).\\
Throughout the remainder of this work, we assume that $M$ possesses at least a $C^{0,1}$ (Lipschitz) boundary.
Any requirements for higher boundary regularity will be stated explicitly.
\FloatBarrier
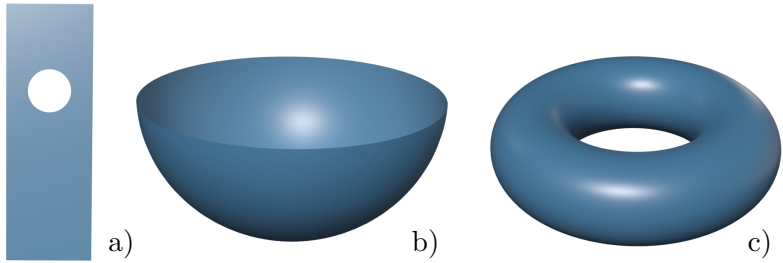
\begin{figure}[tbp]
\begin{center}
    \begin{tikzpicture}
        \draw (3.35, 0) node[inner sep=0]{\includegraphics[width=0.625\textwidth, trim={0cm 0cm 0cm 0cm}, clip]{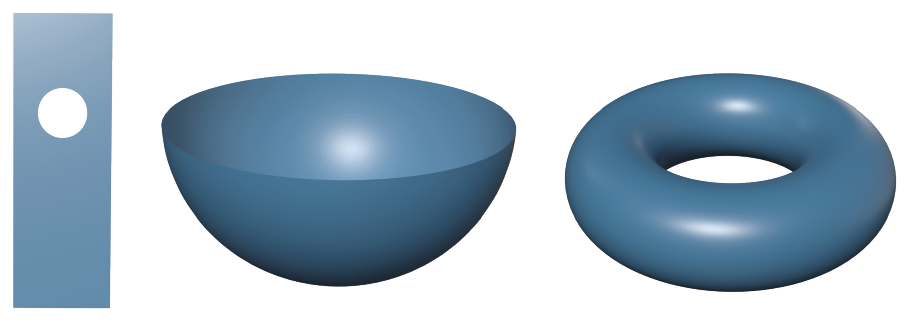}};
        \draw (-0.3, -1.45) node {a)};
        \draw (3.7, -1.45) node {b)};
        \draw (8.1, -1.45) node {c)};
    \end{tikzpicture}
    \caption{Examples of surfaces $M$: a) A subset of $\mathbb{R}^2$ with a $C^{0,1}$ boundary (rectangle with an open disk removed); b) A smooth manifold with boundary (hemisphere); c) A smooth manifold without boundary (torus).}
    \label{fig:domains}
\end{center}
\end{figure}

\paragraph{Function spaces.}
We now introduce the Sobolev spaces on $M$ relevant to this work, extending standard definitions from bounded domains in $\mathbb{R}^2$ \cite{Navier} to $(r,s)$-tensors on manifolds.
We rely on the standard operators from Riemannian geometry defined in \Cref{sec:RiemGeo}.

We write $\Sec$ for the smooth tangential vector fields on $M$ and use the same letter for tensor fields:
let $\Sec^{(r,s)}(W)$ be the set of smooth $(r,s)$-tensor fields on $W$, i.e., smooth sections of $T^{(r,s)}W$, $\Sec^{(r,s)}\coloneqq \Sec^{(r,s)}(W)\vert_{M}$ and $\Sec^{(r,s)}_0$ be the subset of $\Sec^{(r,s)}$ consisting of tensor fields with compact support in the interior of $M$.
For the specific cases of smooth functions, vector fields, and 1-forms, we adopt the following notation:
\begin{align*}
    C^\infty\coloneqq \Sec^{(0,0)},~C^\infty_0\coloneqq \Sec^{(0,0)}_0,~\Sec\coloneqq \Sec^{(1,0)},~\Sec_0\coloneqq \Sec^{(1,0)}_0,~\Omega^1\coloneqq \Sec^{(0,1)},~\Omega^1_0\coloneqq \Sec^{(0,1)}_0.
\end{align*}
Further, we denote by $(L^p(\Sec^{(r,s)}), \Vert\cdot\Vert_{L^p})$ the standard Lebesgue spaces for $1\leq p\leq \infty$.
We define the dual pairing between $U\in L^p(\Sec^{(r,s)})$ and $V\in L^q(\Sec^{(r,s)})$ (where $\frac{1}{p}+\frac{1}{q}=1$) by $(U,V) \coloneqq \int_M \langle U,V\rangle\mathrm{d}V$.

For $X,Y\in \Sec$, $\nabla_Y X$ denotes the Levi-Civita connection.
This extends to the covariant derivative $\nabla: \Sec^{(r,s)} \to \Sec^{(r,s+1)}$ for smooth tensor fields.
Utilizing the formal $L^2$-adjoint $\delta(A) \coloneqq -\mathrm{tr}_{12}(\nabla A)$ for $A \in \Sec^{(r,s)}$ ($s\geq 1$), we generalize the concept of weak derivatives to $L^p(\Sec^{(r,s)})$.
This yields the Sobolev spaces $W^{m,p}(\Sec^{(r,s)})$ for $m\in \mathbb{N}$ and $1\leq p\leq \infty$, equipped with the norm $\Vert A\Vert_{W^{m,p}} \coloneqq \left( \sum_{k=0}^m \Vert\nabla^k A\Vert_{L^p}^p \right)^{\frac{1}{p}}$.

It is folklore that standard results for Sobolev spaces in $\mathbb{R}^2$ carry over to surfaces via a partition of unity argument.
In particular, the density results $L^p(\Sec^{(r,s)})=\overline{\Sec^{(r,s)}_0}^{\Vert\cdot\Vert_{L^{p}}}$ and $W^{m,p}(\Sec^{(r,s)})=\overline{\Sec^{(r,s)}}^{\Vert\cdot\Vert_{W^{m,p}}}$ hold for $1\leq p<\infty$ \cite[Lem I.1.1, Thm I.1.2]{Navier}.
This motivates the definition of the following subspace
\begin{align*}
    W^{m,p}_0(\Sec^{(r,s)})\coloneqq \begin{cases}W^{m,p}(C^\infty)\cap L^1_0(C^\infty)
    ~&\mathrm{if~}\partial M=\emptyset~\mathrm{and~}r=0=s
    \\\overline{\Sec^{(r,s)}_0}^{\Vert\cdot\Vert_{W^{m,p}}},~&\mathrm{else}
    \end{cases}.
\end{align*}
Here, $L^1_0(C^\infty)\coloneqq\{f\in L^1(C^\infty)\vert~\int_M f\mathrm{d}V=0\}$ ensures a vanishing mean value.
The corresponding dual space is denoted by $W^{-m,q}(\Sec^{(r,s)})$, where $q$ is the dual exponent.

Since a surface with a $C^{k,1}$ boundary constitutes a $C^{k,1}$ submanifold of $W$, we define the boundary trace spaces $W^{m,p}(\Sec^{(r,s)}\vert_{\partial M})$ for $\partial M\neq \emptyset$ and $m\leq k+1$, following the construction in \cite[Def 1.3.3.2]{Grisvard}.
In particular, when $m\geq 1$ and $m-\frac{1}{p}\notin\mathbb{N}$, the boundary trace map $W^{m,p}(\Sec^{(r,s)}) \to W^{m-1,p}(\Sec^{(r,s)}\vert_{\partial M})$ is continuous and compact, with kernel $W^{m,p}(\Sec^{(r,s)}) \cap W^{1,p}_0(\Sec^{(r,s)})$.\\
Furthermore, for $W\in W^{1,q}(\Sec^{(1,1)})$ and $V\in W^{1,p}(\Sec)$ with $1<p,q< \infty$ and $\frac{1}{p}+\frac{1}{q}=1$, the following Green's formula holds
\begin{align*}
    (W,\nabla V)=(\delta W, V)+\langle W(N),V\rangle_{\partial M}.
\end{align*}
Here, $N \in L^\infty(\Sec\vert_{\partial M})$ denotes the unique outward-pointing unit normal vector field on $\partial M$, and $W$ is identified with a linear map $TM \to TM$.\\
Moreover, the boundary trace induces the fractional Sobolev spaces $W^{m-\frac{1}{p},p}(\Sec^{(r,s)}\vert_{\partial M})$.
These are defined as the image of $W^{m,p}(\Sec^{(r,s)})$ under the trace operator, equipped with the natural quotient norm of $W^{m,p}(\Sec^{(r,s)})/\mathrm{ker}(\gamma_0)$, where $\gamma_0$ denotes the trace operator.

Finally, we adopt the following abbreviations for function spaces: $W^{m,p}\coloneqq W^{m,p}(C^\infty),~W^{m,p}(\partial M)\coloneqq W^{m,p}(C^\infty\vert_{\partial M})$.
For the Hilbert space case ($p=2$), we use the standard notation: $H^m(\Sec^{(r,s)}) \coloneqq W^{m,2}(\Sec^{(r,s)}),~H^{m-\frac{1}{2}}(\Sec^{(r,s)}\vert_{\partial M}) \coloneqq W^{m-\frac{1}{2},2}(\Sec^{(r,s)}\vert_{\partial M})$.
Analogous notation applies to the corresponding dual spaces.

\paragraph{Additional differential operators.}
We introduce the following differential operators:
First, the gradient $\mathrm{grad}(\cdot)$ and its negative formal $L^2$-adjoint, the divergence $\mathrm{div}(\cdot)$.
Second, the rotated gradient operator $\mathrm{rot}(\cdot) \coloneqq -J\mathrm{grad}(\cdot)$ and its formal $L^2$-adjoint, the scalar curl $\mathrm{curl}(\cdot) \coloneqq -\mathrm{div}(J\cdot)$, where $J$ denotes the Hodge star operator, i.e., a rotation by 90°.

These definitions ensure that the weak divergence and curl operators are well-defined, giving rise to the Hilbert spaces $H(\mathrm{div})$ and $H(\mathrm{curl})$, equipped with their natural graph norms $\Vert\cdot\Vert_{H(\mathrm{div})}$ and $\Vert\cdot\Vert_{H(\mathrm{curl})}$.
Analogous to the standard Sobolev spaces, density arguments yield $H(\mathrm{div}) = \overline{\Sec}^{\Vert\cdot\Vert_{H(\mathrm{div})}}$ and $H(\mathrm{curl}) = \overline{\Sec}^{\Vert\cdot\Vert_{H(\mathrm{curl})}}$ \cite[Thm 1.1]{Temam}.

Furthermore, defining the tangential vector field $T \coloneqq JN$, the standard trace construction \cite[Thm 2.5, Thm I.2.11]{Navier} provides the continuous trace operators:
 $\langle\,\cdot\,,N\rangle\vert_{\partial M}:H(\mathrm{div})\to H^{-\frac{1}{2}}(\partial M)$ and $\langle\,\cdot\,,T\rangle\vert_{\partial M}:H(\mathrm{curl})\to H^{-\frac{1}{2}}(\partial M)$.
These traces satisfy the following Green's formulas for $V_1 \in H(\mathrm{div})$, $V_2 \in H(\mathrm{curl})$, and $f \in H^1$:
\begin{align*}
    &(\mathrm{div}(V_1),f)=-(V_1,\mathrm{grad}(f))+\langle\langle V_1,N\rangle ,f\rangle_{\partial M},
    \\&(\mathrm{curl}(V_2),f)=(V_2,\mathrm{rot}(f))+\langle\langle V_2,T\rangle ,f\rangle_{\partial M}.
\end{align*}
We denote the kernels of these trace operators by $H_0(\mathrm{div})\coloneqq \mathrm{ker}(\langle\,\cdot\,,N\rangle)$ and $H_0(\mathrm{curl})\coloneqq\mathrm{ker}(\langle\,\cdot\,,T\rangle)$.
It is well-known that these kernels coincide with the closures of $\Sec_0$, i.e., $H_0(\mathrm{div}) = \overline{\Sec_0}^{\Vert\cdot\Vert_{H(\mathrm{div})}}$ and $H_0(\mathrm{curl}) = \overline{\Sec_0}^{\Vert\cdot\Vert_{H(\mathrm{curl})}}$ \cite[Thm I.2.6, Thm I.2.12]{Navier}.

Finally, we define the velocity spaces:
\begin{align*}
    &J_{L^2}\coloneqq\{V\in H(\mathrm{div})\vert~\mathrm{div}(V)=0~\mathrm{and}~\langle V, N\rangle\vert_{\partial M}=0\},
    \\&J_{H^1_0}\coloneqq\{V\in H^1_0(\Sec)\vert~\mathrm{div}(V)=0\}.
\end{align*}
The notation for these spaces is motivated by the density results $J_{L^2} = \overline{\mathscr{V}}^{\Vert\cdot\Vert_{L^2}}$ and $J_{H^1_0} = \overline{\mathscr{V}}^{\Vert\cdot\Vert_{H^1}}$, where $\mathscr{V} \coloneqq \{V \in \Sec_0 \mid \mathrm{div}(V)=0\}$ (see \Cref{thm:DecompH-1AndDensity}).

\paragraph{Bochner spaces.}
Let $(L^p_T(B), \Vert \cdot\Vert_{L^p_T(B)})$ denote the Bochner space of functions $f:(0,T) \to B$, where $T \in (0,\infty]$ and $B$ is a Banach space.
We adopt the notation $f_t \coloneqq f(t) \in B$ for $t \in (0,T)$.
To interpret equations holding almost everywhere in time $t$, we recall the following equivalence
\begin{align*}
    f=0\in L^p_T(B) \Leftrightarrow \mathrm{for~a.e.~t}\in(0,T):~\langle b',f_t\rangle=0,~\forall b'\in B'.
\end{align*}
where $\langle \cdot, \cdot \rangle$ denotes the duality pairing between $B'$ and $B$.

To handle initial conditions, we utilize the space $C([0,T]; B)$ of continuous functions $h:[0,T]\to B$.
For $T < \infty$, this space is equipped with the standard maximum norm $\Vert h\Vert_{C([0,T];B)}=\max_{t\in[0,T]}\Vert h_t\Vert$.
A key criterion for time continuity arises in the context of Gelfand triples $\mathcal{V} \subset \mathcal{H} \subset \mathcal{V}'$, where $\mathcal{V}$ and $\mathcal{H}$ are Hilbert spaces.
Specifically, if $V \in L^2_T(\mathcal{V})$ and its weak time derivative satisfies $\partial_t V \in L^2_T(\mathcal{V}')$, then it follows that $V \in C([0,T]; \mathcal{H})$ \cite[p.150-151]{NavierTime}.
Furthermore, under these regularity assumptions, the following Green's formula holds for any $V, W$ satisfying these conditions
\begin{align*}
    \int_0^T\left(\langle \partial_tV_t,W_t \rangle+\langle V_t,\partial_tW_t \rangle\right)\mathrm{d}t=(V_T,W_T)_{\mathcal{H}}-(V_0,W_0)_{\mathcal{H}}.
\end{align*}

\section{Riemannian geometry}
\label{sec:RiemGeo}
\begin{definition}[Covariant derivative on {$\Sec^{(r,s)}$} {\cite{Lee2}}]
    \label{def:connection}
    \nextline The Levi-Civita connection gives rise to the covariant derivative $\nabla:\Sec^{(r,s)}\to\Sec^{(r,s+1)}$, which is defined for $\alpha_1,\ldots,\alpha_r\in \Omega^1$ and $X,Y_1,\ldots,Y_s\in \Sec$ by:
    \begin{align*}
        &{\textstyle(\nabla A)(\underbrace{\ldots}_{r~\mathrm{times}},X,\underbrace{\ldots}_{s~\mathrm{times}})\!\coloneqq\! (\nabla_XA)(\ldots)},
        \\&(\nabla_XA)(\alpha_1,\ldots,\alpha_r,Y_1,\ldots Y_s)\!\coloneqq\! X(A(\alpha_1,\ldots,\alpha_r,Y_1,\ldots Y_s))\!-\!{\textstyle\sum_{i=1}^{r+s}}A(\alpha_1,\ldots,\nabla_XW_i,\ldots, Y_s),
    \end{align*}
    where $W_i=\alpha_i$ for $i\leq r$ and $W_i=Y_{i-r}$ for $i>r$.
\end{definition}

\begin{definition}[Traces {\cite{Lee2}}]
    \label{def:trace}
    \nextline The basis independent traces of $A\in T^{(1,1)}M$ and $B\in T^{(r,s)}M$ (for $s\geq 2$) are defined by:
    \begin{align*}
        &\mathrm{tr}(A)\coloneqq {\textstyle \sum_{k=1}^n}dx^k(A(\partial_k)),
        \\&\mathrm{tr}_{12}(B)={\textstyle \sum_{i,j=1}^ng^{ij}B(\underbrace{\ldots}_{r~\mathrm{times}},\partial_i,\partial_j,\underbrace{\ldots}_{s-2~\mathrm{times}})}.
    \end{align*}
    Notice that $\mathrm{tr}_{12}$ explicitly depends on the Riemannian metric $g$, whereas $\mathrm{tr}$ does not.
\end{definition}

\begin{defnpro}[Curvature {\cite{Lee2}}]
    \label{def:curv}
    \nextline The curvature endomorphism is defined by $R(X,Y)\coloneqq\nabla_Y\nabla_X-\nabla_X\nabla_Y+\nabla_{[X,Y]}$ for $X,Y\in \Sec$.
    This induces the curvature tensor $Rm(X,Y,Z,W)\coloneqq\langle R(X,Y)Z,W\rangle$ for $Z,W\in \Sec$.
    Moreover on 2 dimensional manifolds the Gaussian curvature may be characterized as $\kappa\coloneqq Rm(E_1,E_2,E_1,E_2)$ for any local orthonormal frame $E_1,E_2$.
\end{defnpro}

\begin{definition}[Hodge operator on vector fields {\cite[p.105]{Jost}}]
    \label{def:Hodge*}
    \nextline Let ${*}_k$ denote the Hodge-${*}$-operator on differential $k$-forms, and let $\flat$ and $\sharp$ denote the musical isomorphisms that lower and raise indices via the metric, i.e., $\flat X=\langle X,\cdot\rangle$ is the 1-form associated with a vector field $X$ and $\sharp$ is the inverse operation.
    The Hodge operator $J$ acting on a vector field $X \in \Sec$ is defined by: $J(X)\coloneqq\sharp({*}_1(\flat X))$.
    In an oriented orthonormal frame $(E_1,E_2)$ this reads $JE_1=E_2$ and $JE_2=-E_1$, i.e., $J$ is the rotation by $90^\circ$ in the tangent plane; in particular, $J^2=-\mathrm{id}$, $J$ is parallel, and $J$ is antisymmetric, $\langle JX,Y\rangle=-\langle X,JY\rangle$.
\end{definition}

\begin{theorem}[Killing field]
\label{thm:KillingField}
\nextline The space of Killing fields $\mathcal{K}=\{K\in H^1(\Sec)\vert ~D_K=0,~\langle K,N\rangle\vert_{\partial M}=0\}$ satisfies the following properties:
\begin{enumerate}[1.)]
    \item $\mathcal{K}\subset H^1(\Sec)\cap J_{L^2}$. 
    \item For any $k \in \mathbb{N}$, the embedding $\mathcal{K} \hookrightarrow H^k(\Sec)$ holds. In particular, $\mathcal{K} \subset \Sec$, i.e., every Killing field is smooth.    
    \item If $\partial M \neq \emptyset$ and $K \in \mathcal{K}$ vanishes on the boundary (i.e., $K|_{\partial M}=0$), then $K=0$.
\end{enumerate}
\end{theorem}

\begin{proof}
\begin{enumerate}[1.)]
    \item This is a direct consequence of $\mathrm{div}(K)=\mathrm{tr}(\nabla K)=\mathrm{tr}(D_K)=0$.
    \item A detailed proof in local coordinates is provided in \cite[Lem 3]{priebe1994solvability}.
    Combining this local result with a partition of unity yields the global statement for surfaces.
    \item The proof proceeds similarly to \cite[Prop 6.2]{ShaoSimonettWilke}.
    First, one establishes that $K|_{\partial M} = 0$ implies $\nabla K|_{\partial M} = 0$.
    Next, we choose a surface $\Tilde{M}$ with smooth boundary such that $M \subset \Tilde{M}\subset W$.   
    Extending $K$ by zero to $\Tilde{M}$, we obtain $K \in H^1(\Sec(\Tilde{M}))$.
    Thus, $K$ remains a Killing field on $\Tilde{M}$, and by property 2.), it follows that $K \in \Sec(\Tilde{M})$.   
    Finally, \cite[Prop 8.1.4]{Petersen} states that a Killing field is uniquely determined by its value $K(p)$ and its covariant derivative $(\nabla K)(p)$ at any single point $p \in \Tilde{M}$.
    Choosing $p \in \partial M$, where both $K$ and $\nabla K$ vanish, we conclude that $K=0$ on $\Tilde{M}$, and consequently on $M$.
\end{enumerate}
\end{proof}

%% file: content/07AppendixProofs.tex
\section{Proofs of auxiliary results}
\label{sec:proofs}
In this appendix we collect the proofs of the auxiliary results that were postponed in \Cref{sec:Preliminaries} and \Cref{sec:VeloPrecise} in order not to interrupt the main line of argument.
We begin with a Poincar\'e inequality that is used in the proof of the Gagliardo--Nirenberg inequality.

\begin{lemma}[Poincar\'e inequality]
\label{lem:Poincare}
Let $\partial M\neq \emptyset$ and $\Sigma$ be a portion of $\partial M$ with strictly positive measure.
Then, for any $1 < p < \infty$, there exists a positive constant $C(p,\Sigma,M)$ such that every $A \in W^{1,p}$ vanishing on $\Sigma$ (i.e., $A\vert_\Sigma = 0$) satisfies
\begin{align}
    \Vert A\Vert_{L^p}\leq C(p,\Sigma,M)\Vert\nabla A\Vert_{L^p}.
    \label{eq:Poincare}
    \tag{P$_\leq$}
\end{align}
\end{lemma}

\begin{proof}
    This result is standard and can be found, for example, in \cite[Lem I.3.1]{Navier}.
    To prove the estimate \eqref{eq:Poincare}, we argue by contradiction.\\
    Assume the inequality does not hold.
    Then there exists a sequence $A_n \in W^{1,p}$ with $A_n\vert_\Sigma = 0$ such that $\Vert A_n\Vert_{L^p}=1$ and $\Vert \nabla A_n\Vert_{L^p}<\frac{1}{n}$.
    By Rellich's lemma, since the sequence is bounded in $W^{1,p}$, we can extract a subsequence that is Cauchy in $\Vert \cdot\Vert_{L^p}$.
    Combined with the convergence of the gradients to zero, this subsequence is Cauchy in $\Vert \cdot\Vert_{W^{1,p}}$.
    Its limit $A \in W^{1,p}$ must satisfy $\nabla A=0$ and $A\vert_\Sigma = 0$, but $\Vert A\Vert_{L^p}=1$.
    Since $\nabla A=0$ (i.e., $A$ is parallel) and $M$ is connected, the scalar function $\langle A,A\rangle$ must be constant.
    Because $A$ vanishes on $\Sigma$ (a set of strictly positive measure), this constant must be zero.
    Consequently, $A=0$ on $M$, which contradicts the normalization $\Vert A\Vert_{L^p}=1$.
\end{proof}

\begin{remark}
    A very similar argument establishes the Poincar\'e inequality on $W^{1,p}_0$ for $1\leq p<\infty$ in the case where $\partial M=\emptyset$ and $r=s=0$.
\end{remark}

\begin{proof}[Proof of \Cref{thm:GagNirenberg}]
    For bounded domains in $\mathbb{R}^2$ and scalar functions (i.e., $r=0=s$), the estimate \eqref{eq:GagNirenberg} was originally established in a generalized form in \cite{GagNirenberg1} and \cite{GagNirenberg2}.
    This result extends to $(r,s)$-tensor fields by applying the scalar estimate component-wise and using standard estimates for discrete norms.\\
    For surfaces, we combine this Euclidean tensor estimate with a partition of unity argument.
    For any $A \in \Sec^{(r,s)}$, this initially yields an inequality involving the full Sobolev norm
    \begin{align*}
        \Vert A\Vert_{L^p} \leq C(p,\theta,M)\left(\Vert A\Vert_{W^{1,a}}^\theta\Vert A\Vert_{L^a}^{1-\theta}+\Vert A\Vert_{L^a}\right).
    \end{align*}
    Algebraic manipulations lead directly to the standard form presented in \eqref{eq:GagNirenberg}.
    Since smooth sections are dense in $W^{1,a}(\Sec^{(r,s)})$, the result extends to the full Sobolev space by a standard density argument.
    Finally, the tightened estimate \eqref{eq:GagNirenbergPoincare} is a direct consequence of the Poincar\'e inequality (\Cref{lem:Poincare}).
\end{proof}

\begin{proof}[Proof of \Cref{cor:LsJL2ToJH10}]
    The inclusion ``$\subset$'' follows immediately from the Sobolev embedding theorem (\Cref{eq:SobEmbLp}), \Cref{lem:s&r} and Green's formula.\\ 
    For the inclusion ``$\supset$'', let $V \in L^s(\Sec) \cap J_{L^2}$.
    By definition, $\mathrm{div}(V)=0$ and $\langle V,N\rangle\vert_{\partial M}=0$.
    Testing the variational equation with $\chi \in C^\infty_0$ implies that $\mathrm{curl}(V)=\omega \in L^2$.
    Since $H^1 \subset W^{1,r}$, we may test with arbitrary $\chi \in H^1$.
    Applying Green's formula for $H(\mathrm{curl})$ then yields $\langle\langle V, T \rangle,\chi\rangle_{\partial M} = 0$.
    By definition of $H^{1/2}(\partial M)$, this implies $\langle V,T\rangle=0$ in $H^{-1/2}(\partial M)$.
    Hence, Gaffney's inequality implies $V \in H^1_0(\Sec)$, and thus $V \in J_{H^1_0}$.
    A final application of Green's formula, combined with the density of $C^\infty_0$ in $L^2$, confirms that the equation holds for all $\chi \in L^2$.
\end{proof}

\begin{proof}[Proof of $\vert V\vert^2\in L^2_T(W^{1,r})$ from \Cref{thm:reg(V.d)V}]
We begin with the case $\partial M\neq \emptyset$.
First, we verify that $\vert V\vert^2 \in L^2_T(L^2)$. Since $M$ is compact, this implies $\vert V\vert^2 \in L^2_T(L^r)$.
Applying the Gagliardo--Niremberg inequality (\Cref{eq:GagNirenbergPoincare}) with $p=4$ and $a=2$, we estimate
\begin{align*}
    \Vert \vert V\vert^2\Vert_{L^2_T(L^2)}&
    \leq C(M)\Vert \Vert \nabla V\Vert_{L^2}\Vert V\Vert_{L^2}\Vert_{L^2_T(\mathbb{R})}
    \leq C(M)\Vert V\Vert_{L^2_T(H^1)}\Vert V\Vert_{L^\infty_T(L^2)}.
\end{align*}
Next we prove that $\mathrm{grad}(\vert V\vert^2)\in L^2_T(L^r)$.
Let $s \in (2,\infty)$ be the dual exponent of $r$ (i.e., $\frac{1}{s}+\frac{1}{r}=1$) and choose $q\in (2,\infty)$ such that $\frac{1}{q}=\frac{1}{2}-\frac{1}{s}$ (which ensures $\frac{1}{r}=\frac{1}{2}+\frac{1}{q}$).
Applying Hölder's inequality yields
\begin{align*}
    \Vert \mathrm{grad}(\vert V\vert^2)\Vert_{L^2_T(L^r)}&\leq 2\left\Vert\Vert \nabla V\Vert_{L^2} \Vert V\Vert_{L^q}\right\Vert_{L^2_T(\mathbb{R})},
\end{align*}
We observe that the regularity $V \in L^2_T(J_{H^1_0}) \cap L^\infty_T(J_{L^2})$ alone is insufficient to bound the remaining term.
However, we can exploit the additional regularity of $\mathrm{curl}(V)$ combined with Gaffney's inequality (\Cref{eq:GaffneyIneq}).
Applying the Gagliardo--Niremberg inequality (\Cref{eq:GagNirenbergPoincare}) with $p=q$ and $a=2$ yields
\begin{align*}
    \Vert V\Vert_{L^q}\leq C(M)\Vert \nabla V\Vert_{L^2}^{\frac{2}{s}}\Vert V\Vert_{L^2}^{1-\frac{2}{s}}\leq C(M)\Vert \mathrm{curl}(V)\Vert_{L^2}^{\frac{2}{s}}\Vert V\Vert_{L^2}^{1-\frac{2}{s}},
\end{align*}
where the second inequality follows from Gaffney's estimate.
Substituting this into our previous bound, we obtain
\begin{align}
    \Vert \Vert \nabla V\Vert_{L^2} \Vert V\Vert_{L^q}\Vert_{L^2_T(\mathbb{R})}&\leq C(M)\left\Vert \Vert \mathrm{curl}(V)\Vert_{L^2}^{1+\frac{2}{s}}\Vert V\Vert_{L^2}^{1-\frac{2}{s}}\right\Vert_{L^2_T(\mathbb{R})}\nonumber
    \\&\leq C(M)\left\Vert \Vert \mathrm{curl}(V)\Vert_{L^2}^{1+\frac{2}{s}}\right\Vert_{L^2_T(\mathbb{R})}\Vert V\Vert_{L^\infty_T(J_{L^2})}^{1-\frac{2}{s}}.
    \label{proof:1+2s}
\end{align}
Finally, to bound the term $\Vert \mathrm{curl}(V)\Vert_{L^2}$, we utilize the hypothesis $\mathrm{curl}(V) \in L^2_T(W^{1,r}) \cap L^\infty_T(L^r)$.
Since no zero boundary trace is imposed on $\mathrm{curl}(V)$, the tightened estimate \eqref{eq:GagNirenbergPoincare} is not applicable here; instead, a final application of the general estimate \eqref{eq:GagNirenberg} with $p=2$ and $a=r$, whose additional term $\Vert\mathrm{curl}(V)\Vert_{L^r}$ is absorbed into the full $W^{1,r}$-norm, gives
\begin{align*}
    \Vert \mathrm{curl}(V)\Vert_{L^2}\leq C(M)\Vert \mathrm{curl}(V)\Vert_{W^{1,r}}^{1-\frac{2}{s}}\Vert \mathrm{curl}(V)\Vert_{L^{r}}^{\frac{2}{s}}.
\end{align*}
Combining this inequality with \eqref{proof:1+2s} yields
\begin{align*}
    \left\Vert \Vert \mathrm{curl}(V)\Vert_{L^2}^{1+\frac{2}{s}}\right\Vert_{L^2_T(\mathbb{R})}&\leq C(s,M)\left\Vert \left(\Vert \mathrm{curl}(V)\Vert_{W^{1,r}}^{1-\frac{2}{s}}\Vert \mathrm{curl}(V)\Vert_{L^{r}}^{\frac{2}{s}}\right)^{1+\frac{2}{s}}\right\Vert_{L^2_T(\mathbb{R})}
    \\&\leq C(s,M)\left\Vert\Vert \mathrm{curl}(V)\Vert_{W^{1,r}}^{1-\frac{4}{s^2}}\right\Vert_{L^2_T(\mathbb{R})}\Vert \mathrm{curl}(V)\Vert_{L^\infty_T(L^{r})}^{\frac{2}{s}+\frac{4}{s^2}}
    \\&\leq C(s,T,M)\Vert \mathrm{curl}(V)\Vert_{L^2_T(W^{1,r})}^{1-\frac{4}{s^2}}\Vert \mathrm{curl}(V)\Vert_{L^\infty_T(L^{r})}^{\frac{2}{s}+\frac{4}{s^2}}.
\end{align*}
In the last step we used Hölder's inequality in time: for $\alpha\coloneqq 1-\frac{4}{s^2}\in(0,1)$, it holds that $\Vert f^{\alpha}\Vert_{L^2_T(\mathbb{R})}\leq T^{\frac{1-\alpha}{2}}\Vert f\Vert_{L^2_T(\mathbb{R})}^{\alpha}$, so that the dependence on the (finite) time horizon $T$ is absorbed into the constant $C(s,T,M)$.
This concludes the proof for the case where $\partial M\neq \emptyset$.

If $\partial M=\emptyset$, the proof proceeds analogously.
The only modification required is the use of the general Gaffney inequality (\Cref{eq:GenGaffneyIneq}) and the general Gagliardo--Nirenberg inequality (\Cref{eq:GagNirenberg}), rather than their simplified counterparts (\Cref{eq:GaffneyIneq,eq:GagNirenbergPoincare}).
This introduces additional terms involving $\Vert V\Vert_{L^2}$ and $\Vert \mathrm{curl}(V)\Vert_{L^r}$.
However, since $V\in L^\infty_T(J_{L^2})$ and $\mathrm{curl}(V)\in L^\infty_T(L^r)$, these terms remain bounded.

Finally, we prove that $\nabla_VV\in L^2_T(L^r(\Sec))$.
By \refwl{lem:s&r} and the fact that $V\in L^\infty_T(J_{L^2})$, we deduce that the product $\mathrm{curl}(V) JV$ belongs to $L^2_T(L^r(\Sec))$.
Combining this observation with \refwl{lem:Nonlinterm} (which relates $\nabla_V V$ to $\mathrm{grad}\vert V\vert^2$ and the $\mathrm{curl}$ term) and the established regularity $\vert V\vert^2\in L^2_T(W^{1,r})$, we conclude that $\nabla_VV\in L^2_T(L^r(\Sec))$.
\end{proof}

\begin{proof}[Proof of \Cref{prop:pres}]
The proof follows the construction in \cite[p.160]{NavierTime}.  
To define the pressure, we first introduce the functional $L_t \in H^{-1}(\Sec)$ for each $t\in[0,T]$ by
\begin{align*}
    \langle L_t, W\rangle\coloneqq\int_0^t\left(\langle F_s,W\rangle-(\nabla_{V_s}V_s,W)-2\nu(D_{V_s},\nabla W)\right)\mathrm{d}s-(V_t-V_0,W),
\end{align*}
for all $W \in H^1_0(\Sec)$.
Using the momentum equation (\Cref{eq:dtVNavStokes}, or \eqref{eq:dtVEuler} if $\nu=0$) and Green's formula, we observe that $\langle L_t, W\rangle = 0$ whenever $\mathrm{div}(W)=0$ (i.e., for all $W \in J_{H^1_0}$).
Consequently, by \Cref{thm:DecompH-1AndDensity}, there exists a unique scalar function $h_t \in L^2_0$ such that $\langle L_t,W\rangle=-(h_t,\mathrm{div}(W))$ for all $W \in H^1_0(\Sec)$.
We then define the (distributional) pressure $p\in (C^\infty_0(M\times(0,T)))'$ -- the distinguished representative of \Cref{prop:pres} -- as $p \coloneqq\frac{\mathrm{d}}{\mathrm{d}t}h$.
A straightforward computation verifies that $p$ satisfies \Cref{eq:DistribP}, where the operators $\mathrm{grad}$, $\delta$, and $\frac{\mathrm{d}}{\mathrm{d}t}$ are understood in the distributional sense.
Now assume $\partial_t V \in L^2_T(H^{-1}(\Sec))$. We claim that this implies $p \in L^2_T(L^2_0)$.
To verify this, we define the functional $\tilde{L}_t \in H^{-1}(\Sec)$ for a.e. $t \in (0,T)$ by
\begin{align*}
    \langle \tilde{L}_t, W\rangle\coloneqq \langle F_t-\partial_tV,W\rangle-(\nabla_{V_t}V_t,W)-2 \nu(D_{V_t},\nabla W),
\end{align*}
for all $W \in H^1_0(\Sec)$.
Analogous to the previous argument, the weak formulation implies that $\langle \tilde{L}_t, W\rangle = 0$ for all $W \in J_{H^1_0}$.
By \Cref{thm:DecompH-1AndDensity}, for a.e. $t$, there exists a unique $p_t \in L^2_0$ such that $\langle \tilde{L}_t, W\rangle = -(p_t, \mathrm{div}(W))$.
This confirms that \Cref{eq:NavStokesP} holds.\\
Finally, to establish that $p \in L^2_T(L^2_0)$, we recall that the divergence operator restricts to a topological isomorphism $\mathrm{div}:(J_{H^1_0})^\perp \to L^2_0$ (\Cref{thm:divIsom}).
Since the functional $\tilde{L}$ belongs to $L^2_T(H^{-1}(\Sec))$ (following estimates similar to those in \Cref{lem:VeloForm}), the isomorphism property ensures that $p$ inherits this time regularity.
\end{proof}